\documentclass[a4paper, 10pt]{amsart}
\usepackage[left=2.7cm,right=2.7cm,top=3.5cm,bottom=3cm]{geometry}
\usepackage{amsmath}
\usepackage{amscd,amsthm,amssymb,amsfonts}
\usepackage{mathrsfs}
\usepackage{dsfont}
\usepackage{stmaryrd}
\usepackage{euscript}
\usepackage{expdlist}
\usepackage{enumerate}

\input xy
\xyoption{all}
\theoremstyle{plain}
\newtheorem{thm}{Theorem}[section]

\newtheorem*{thma}{Theorem A}
\newtheorem*{thmb}{Theorem B}
\newtheorem*{thmc}{Theorem C}

\newtheorem*{thm*}{Theorem}

\newtheorem{lm}[thm]{Lemma}
\newtheorem{cor}[thm]{Corollary}
\newtheorem*{cor*}{Corollary}
\newtheorem{prop}[thm]{Proposition}
\newtheorem*{conj*}{Conjecture}

\theoremstyle{remark}

\theoremstyle{definition}
\newtheorem*{defn*}{Definition}

\newtheorem{I_Remark*}{Remark}
\newtheorem{defn}[thm]{Definition}

\newtheorem*{hypH}{Hypothesis (H)}

\newcommand{\nc}{\newcommand}

\newcommand{\beq}{\begin{equation}}
\newcommand{\eeq}{\end{equation}}
\newcommand{\bpmx}{\begin{pmatrix}}
\newcommand{\epmx}{\end{pmatrix}}
\newcommand{\bbmx}{\begin{bmatrix}}
\newcommand{\ebmx}{\end{bmatrix}}
\newcommand{\wh}{\widehat}
\newcommand{\wtd}{\widetilde}

\newcommand{\beqcd}[1]{\begin{equation*}\label{#1}\tag{#1}}
\newcommand{\eeqcd}{\end{equation*}}

\numberwithin{equation}{section}
\newenvironment{mylist}{
  \begin{enumerate}{}{%
      \setlength{\itemsep}{5pt} \setlength{\parsep}{0in}
      \setlength{\parskip}{0in} \setlength{\topsep}{0in}
      \setlength{\partopsep}{0in}
      \setlength{\leftmargin}{0.17in}}}{\end{enumerate}}

\def\parref#1{\ref{#1}}
\def\thmref#1{Theorem~\parref{#1}}

\def\propref#1{Proposition~\parref{#1}}
     
\def\secref#1{\S\parref{#1}}

\def\lmref#1{Lemma~\parref{#1}}
\def\subsecref#1{\S\parref{#1}}
\def\defref#1{Definition~\parref{#1}}

\def\makeop#1{\expandafter\def\csname#1\endcsname
  {\mathop{\rm #1}\nolimits}\ignorespaces}
\makeop{Hom}   \makeop{End}   \makeop{Aut}   %\makeop{Isom}
\makeop{Pic} \makeop{Gal}       \makeop{Div} \makeop{Lie}
\makeop{PGL}   \makeop{Corr} \makeop{PSL} \makeop{sgn} \makeop{Spf}
 \makeop{Tr} \makeop{Nr} \makeop{Fr} \makeop{disc}
\makeop{Proj} \makeop{supp} \makeop{ker}   \makeop{Im} \makeop{dom}
\makeop{coker} \makeop{Stab} \makeop{SO} \makeop{SL} \makeop{SL}
\makeop{Cl}    \makeop{cond} \makeop{Br} \makeop{inv} \makeop{rank}
\makeop{id}    \makeop{Fil} \makeop{Frac}  \makeop{GL} \makeop{SU}
\makeop{Trd}   \makeop{Sp} \makeop{Tr}    \makeop{Trd} \makeop{Res}
\makeop{ind} \makeop{depth} \makeop{Tr} \makeop{st} \makeop{Ad}
\makeop{Int} \makeop{tr}    \makeop{Sym} \makeop{can} \makeop{SO}
\makeop{torsion} \makeop{GSp} \makeop{Tor}\makeop{Ker} \makeop{rec}
\makeop{Ind} \makeop{Coker}
 \makeop{vol} \makeop{Ext} \makeop{gr} \makeop{ad}
 \makeop{Gr}\makeop{corank} \makeop{Ann}
\makeop{Hol} %Holomorphic
\makeop{Fitt} \makeop{Mp} \makeop{CAP}

\def\Sel{Sel}

\def\Ord{{\mathrm{ord}}}
\def\makebb#1{\expandafter\def
  \csname bb#1\endcsname{{\mathbb{#1}}}\ignorespaces}
\def\makebf#1{\expandafter\def\csname bf#1\endcsname{{\bf
      #1}}\ignorespaces}
\def\makegr#1{\expandafter\def
  \csname gr#1\endcsname{{\mathfrak{#1}}}\ignorespaces}
\def\makescr#1{\expandafter\def
  \csname scr#1\endcsname{{\EuScript{#1}}}\ignorespaces}
\def\makecal#1{\expandafter\def\csname cal#1\endcsname{{\mathcal
      #1}}\ignorespaces}
\def\doLetters#1{#1A #1B #1C #1D #1E #1F #1G #1H #1I #1J #1K #1L #1M
                 #1N #1O #1P #1Q #1R #1S #1T #1U #1V #1W #1X #1Y #1Z}
\def\doletters#1{#1a #1b #1c #1d #1e #1f #1g #1h #1i #1j #1k #1l #1m
                 #1n #1o #1p #1q #1r #1s #1t #1u #1v #1w #1x #1y #1z}
\doLetters\makebb   \doLetters\makecal  \doLetters\makebf
\doLetters\makescr
\doletters\makebf   \doLetters\makegr   \doletters\makegr

\normalsize

\makeop{Ram} \makeop{Rep} \makeop{mass}

\makeop{Bl}
\def\abs#1{\left|#1\right|}

\def\Fpbar{\bar{\mathbb F}_p}

\def\Qp{\Q_p}
\def\Qbar{\ol{\Q}}

\def\Zp{\Z_p}

\def\rmN{{\mathrm N}}

\def\cA{{\mathcal A}}  %automorphic forms
\def\cB{\EuScript B}
\def\cD{\mathcal D}

\def\cF{{\mathcal F}}  %Hida family
\def\cG{{\mathcal G}}
\def\cL{{\mathcal L}}
\def\cH{{\mathcal H}}

\def\cK{{\mathcal K}}  %imaginary quadratic field

\def\cO{\mathcal O}
\def\cS{{\mathcal S}}
\def\cf{{\mathcal f}}
\def\cW{{\mathcal W}}

\def\cX{\mathcal X}
\def\cV{{\mathcal V}}

\def\cC{\mathcal C}

\def\cN{\mathcal N}
\def\cT{\mathcal T}
\def\cY{\mathcal Y}

\def\cU{\mathcal U}

\def\bfc{\mathbf c}

\def\bfL{\mathbf L}

\def\bff{\mathbf f}

\def\bD{\mathbf D}

\def\sF{\mathscr F}
\def\sL{\mathscr L}

\def\sC{\mathscr C}

\def\sA{\mathscr A}

\def\sK{\mathscr K}

\newcommand{\Z}{\mathbf Z}
\newcommand{\Q}{\mathbf Q}
\newcommand{\R}{\mathbf R}
\newcommand{\C}{\mathbf C}
\newcommand{\A}{\mathbf A}    % for adele

\def\bbH{\mathbb H}

\def\bbmu{\boldsymbol{\mu}}

\def\fraka{{\mathfrak a}}
\def\frakb{{\mathfrak b}}
\def\frakc{{\mathfrak c}}

\def\frakp{{\mathfrak p}}

\def\frakq{\mathfrak q}

\def\frakg{\mathfrak g}
\def\frakm{\mathfrak m}

\def\frakF{{\mathfrak F}}

\def\frakC{{\mathfrak C}}

\def\frakX{\mathfrak X}

\def\frakN{\mathfrak N}

\def\bfone{{\mathbf 1}}

\def\Zhat{\widehat{\Z}}

\def\zbar{\ol{z}}

\def\et{{\acute{e}t}}

\def\ab{abelian variety }

\def\pont{Pontryagin } % Pontryagin dual
\def\padic{\text{$p$-adic }}

\def\Neron{N\'{e}ron }
\def\Frob{\mathrm{Frob}}
\newcommand{\<}{\langle}   %\< is not defined yet.
\renewcommand{\>}{\rangle} %\> is already defined.

\def\isoto{\stackrel{\sim}{\to}}

\def\surjto{\twoheadrightarrow}

\def\ot{\otimes}

\def\hookto{\hookrightarrow}
\def\longto{\longrightarrow}
\def\ol{\overline}  \nc{\opp}{\mathrm{opp}} \nc{\ul}{\underline}

\newcommand{\pair}[2]{\< #1, #2\>}

\newcommand{\pairing}{\pair{\,}{\,}}

\def\XYmatrix{\xymatrix@M=8pt} % make \xymatrix not too cluttered
\def\ncmd{\newcommand}
\ncmd{\xysubset}[1][r]{\ar@<-2.5pt>@{^(-}[#1]\ar@<2.5pt>@{_(-}[#1]}
\ncmd{\XYmatrixc}[1]{\vcenter{\XYmatrix{#1}}}
\ncmd{\xyto}[1][r]{\ar@{->}[#1]}
\ncmd{\xyinj}[1][r]{\ar@{^(->}[#1]}
\ncmd{\xysurj}[1][r]{\ar@{->>}[#1]}
\ncmd{\xyline}[1][r]{\ar@{-}[#1]}
\ncmd{\xydotsto}[1][r]{\ar@{.>}[#1]}
\ncmd{\xydots}[1][r]{\ar@{.}[#1]}
\ncmd{\xyleadsto}[1][r]{\ar@{~>}[#1]}
\ncmd{\xyeq}[1][r]{\ar@{=}[#1]} \ncmd{\xyequal}[1][r]{\ar@{=}[#1]}
\ncmd{\xyequals}[1][r]{\ar@{=}[#1]}
\ncmd{\xymapsto}[1][r]{l\ar@{|->}[#1]}\ncmd{\xyimplies}[1][r]{\ar@{=>}[#1]}
\ncmd{\xyiso}{\ar[r]_-{\sim}}
\def\injxy{\ar@{^(->}}

\newcommand{\pMX}[4]{\begin{pmatrix}
{#1}& {#2}\\
{#3}&{#4}\end{pmatrix} }

 \newcommand{\pDII}[2]{\begin{pmatrix}{#1}&0
 \\0&{#2}\end{pmatrix}}

\newcommand{\seesaw}[4]{{#1}\ar@{-}[rd]\ar@{-}[d]&{#2}\ar@{-}[d]\\
{#3}\ar@{-}[ru]&{#4}}

\def\ie{i.e. }

\def\cf{\mbox{{\it cf.} }}

\def\can{{can}}

\def\uf{\varpi} %uniformizer
\def\Abs{{|\!\cdot\!|}} %adelic absolute value
\def\Sg{{\varSigma}}  %%CM type

\def\ndivides{\nmid}

\def\x{{\times}}

\def\onehalf{{\frac{1}{2}}}
\def\e{\varepsilon} % episilon factor
\def\al{\alpha}

\def\Lam{\Lambda}

\def\om{\omega}
\def\dirlim{\varinjlim}
\def\prolim{\varprojlim}
\def\iso{\simeq}
\def\con{\equiv}
\def\bksl{\backslash}
\newcommand\stt[1]{\left\{#1\right\}}
\def\ep{\epsilon}

\def\pd{\partial}
\def\lam{\lambda}
\def\pii{\pi i}

\def\vep{\varepsilon}

\def\sg{\sigma}
\def\vp{\varphi}
\def\disjoint{\sqcup}

\def\AK{\A_\cK}
\def\setp{{(p)}}

\newcommand{\powerseries}[1]{\llbracket{#1}\rrbracket}

\renewcommand\pmod[1]{\,(\mbox{mod }{#1})}

\renewcommand\Re{\text{Re}\,}
\def\vphi{\varphi}
\def\Cp{\C_p}

\def\alg{\mathrm{alg}}

\makeop{Sel}
\makeop{fin}
\makeop{sing}
\makeop{loc}
\makeop{res}
\def\cores{{\rm cor}}

\def\BK{f}
\newcommand{\rcf}[1]{K_{#1}}
\def\HIw{H_{\rm Iw}}

\def\cH{\mathcal H}
\newcommand{\Hg}[2]{\kappa_{#2}(#1)}

\def\Epr{T_{M}}
\def\Hclass{z}

\newcommand{\sk}{\vspace{0.1in}}
\newtheorem{def-thm}[thm]{Definition-Theorem}
\newtheorem{lem}[thm]{Lemma}

\theoremstyle{definition}
\theoremstyle{remark}
\newtheorem{rem}[thm]{Remark}
\newtheorem*{intro-rem}{Remark}
\numberwithin{equation}{section}
\newcommand{\pp}{\mathfrak{p}}

\newcommand{\bQ}{\mathbf{Q}}
\newcommand{\bZ}{\mathbf{Z}}

\newcommand{\heeg}{{\rm heeg}}

\usepackage{hyperref}
\makeop{CH}
\makeop{AJ}

\def\CMP{\delta}
\def\xx{\,\x\,}

\def\frakpbar{\ol{\frakp}}
\def\cK{K}

\title{Heegner cycles and \padic $L$-functions}
\author[F. Castella and M.-L. Hsieh]{Francesc Castella and Ming-Lun Hsieh}

\date{\today}

\address[Castella]{
Mathematics Department, Princeton University, Fine Hall, Washington Road, Princeton,
NJ 08544-1000, USA}
\email{fcabello@math.princeton.edu}
\address[Hsieh]{
Institute of Mathematics, Academia Sinica~\\ Taipei 10617, Taiwan\and National Center for Theoretic Sciences\and
Department of Mathematics, National Taiwan University~
}
\email{mlhsieh@math.sinica.edu.tw}
\thanks{During the preparation of this paper, F.Castella was partially supported by grant MTM2012-34611 and by Prof.~Hida's NSF grant DMS-0753991. M.-L.Hsieh was partially supported by a MOST grant 103-2115-M-002-012-MY5. }
\begin{document}
\begin{abstract}
In this paper, we deduce the vanishing of Selmer groups
for the Rankin--Selberg convolution of a cusp form with a theta series of higher weight
from the nonvanishing of the associated $L$-value,
thus establishing the rank $0$ case of the Bloch--Kato conjecture in these cases.
Our methods are based on the connection between Heegner cycles and \padic $L$-functions,
building upon recent work of Bertolini, Darmon and Prasanna,
and on an extension of Kolyvagin's method of Euler systems to the anticyclotomic setting.
In the course of the proof, we also obtain a higher weight analogue of Mazur's conjecture
(as proven in weight $2$ by Cornut--Vatsal),
%for Heegner cycles and ranks of Selmer groups, %\cite{mazur-icm83}.
and as a consequence of our results, we deduce from Nekov{\'a}{\v{r}}'s work a proof of the parity conjecture
in this setting.\end{abstract}

\maketitle
\tableofcontents
\section{Introduction}
\def\divides{\mid}

% !TEX root = HeegnerCycles.tex

Let $f\in S_{2r}^{\rm new}(\Gamma_0(N))$
be a newform of weight $2r\geq 2$ and level $N$. Fix an odd prime $p\ndivides N$.
%Denote by $\overline{\bQ}$ the algebraic closure of $\bQ$ in $\bC$ and fix an algebraic closure
%$\overline{\bQ}_p$ of $\bQ_p$.
Let $F/\Q_p$ be a finite extension containing the image of the Fourier coefficients of $f$ under a fixed embedding $\imath_p:\overline{\bQ}\hookrightarrow\C_p$, and denote by
\[
\rho_f:{\rm Gal}(\overline{\bQ}/\bQ)\longrightarrow{\rm Aut}_F(V_f(r))\simeq{\rm GL}_2(F)
\]
the self-dual Tate twist of the $p$-adic Galois representation associated to $f$. Let $K/\Q$
be an imaginary quadratic field of odd discriminant $-D_K<-3$ and let $\chi:G_K:=\Gal(\Qbar/K)\to F^\x$ be a
locally algebraic anticyclotomic %\padic Galois
character. %with values in $F$.
The $G_K$-representation
\[
V_{f,\chi}:=V_f(r)\otimes \chi
\]
is then conjugate self-dual, and the associated Rankin $L$-series $L(f,\chi,s)$ satisfies a functional equation relating its values at $s$ and $2r-s$. The Bloch--Kato conjectures (see \cite{BK}, \cite{F-PR}), which provide a vast generalization of
the Birch--Swinnerton-Dyer conjecture and Dirichlet's class number formula, predict in this context the equality
\begin{equation}\label{BKconj}
{\rm ord}_{s=r}L(f,\chi,s)
\overset{?}={\rm dim}_F{\rm Sel}(K,V_{f,\chi})\nonumber\tag{BK}
\end{equation}
between the order of vanishing at the central point of the
Rankin $L$-series $L(f,\chi,s)$ and the size of the
Bloch--Kato Selmer group ${\rm Sel}(K,V_{f,\chi})$ for the representation $V_{f,\chi}$.

%Our basic assumptions are
%\beqcd{Heeg}\text{$N$ is a product of primes split in $\cK$.} \eeqcd
%\beqcd{spl} \text{$p=\frakp\ol{\frakp}$ is split in $\cK$,}\eeqcd
%\beqcd{gp}p\ndivides 2(2r-1)!N\varphi(N).\eeqcd
\begin{hypH}The following hypotheses are assumed throughout.
\begin{itemize}
\item[(a)] $p\ndivides 2(2r-1)!N\varphi(N)$;
\item[(b)]the conductor of $\chi$ is prime to $N$;
\item[(c)]$N$ is a product of primes split in $\cK$;
\item[(d)]$p=\frakp\ol{\frakp}$ is split in $\cK$, where $\frakp$ is induced by $\imath_p$.
\end{itemize}
\end{hypH}

Our first arithmetic application is the proof of new ``rank zero'' cases of
conjecture (\ref{BKconj}).

\begin{thma}
\label{mainthm:BK}
Assume further that the newform $f$ is ordinary at $p$. If $L(f,\chi,r)\neq 0$,  then \[\dim_F{\rm Sel}(K,V_{f,\chi})=0.\]
\end{thma}

\begin{intro-rem}
Let $\ep(V_{f,\chi})=\pm 1$ be the sign of the functional equation of $L(f,\chi,s)$.
The non-vanishing of $L(f,\chi,r)$ implies $\ep(V_{f,\chi})=+1$. On the other hand, under our Hypothesis (H), the global sign $\ep(V_{f,\chi})$ is completely determined by the local sign at the archimedean place, which in turn depends on the infinity type of $\chi$. More precisely, let $c\cO_K$ be the conductor of $\chi$ and let $(j,-j)$ be its infinity type, so that for every $\al\in K^\x$ with $\al\con 1\pmod{c\cO_K}$ we have
\[
\chi(\rec_p(\al))=(\al/\ol{\al})^{j},
\]
where $\rec_p:(K\ot\Qp)^\x\to G_K^{ab}$ is the geometrically normalized local reciprocity law map at $p$. Then one can show that
\[ \ep(V_{f,\chi})=+1\iff j\geq r\text{ or }j\leq -r.\]
In particular, the characters $\chi$ for which Theorem~A applies are all of infinite order.
\end{intro-rem}

Let $\Gamma^-_K:={\rm Gal}(K_\infty/K)$ be the Galois group of the anticyclotomic $\Z_p$-extension of $K$. Write $c=c_o p^s$ with $p\ndivides c_o$. Suppose that $\chi=\psi\phi_0$, where $\psi$ is
an anticyclotomic character of infinity type $(r,-r)$ and conductor $c_o\cO_K$ and $\phi_0$ is a \padic character of $\Gamma_K^-$. The proof of Theorem~A rests on the study of a $p$-adic $L$-function $\mathscr{L}_{\frakp,\psi}(f)\in\overline{\Z}_p\powerseries{\Gamma^-_K}$ defined by the interpolation of the central critical values $L(f,\psi\phi,r)$, as $\phi$ runs over a Zariski-dense subset of $p$-adic characters of $\Gamma_K^-$. In a slightly different form, this $p$-adic $L$-function
was introduced in the earlier work of Bertolini, Darmon and Prasanna \cite{BDP}, where they proved a remarkable formula
relating the values of $\mathscr{L}_{\frakp,\psi}(f)$ at unramified characters \emph{outside} the range of interpolation
to the $p$-adic Abel--Jacobi images of generalized Heegner cycles.

Let $T_f(r)$ be a $\Gal(\Qbar/\Q)$-stable $\cO_F$-lattice in $V_f(r)$. As a key step toward the proof of Theorem~A, we produce Iwasawa cohomology classes
\[
\mathbf{z}_{f}\in H^1_{\rm Iw}(K_\infty,V_f(r)):=\prolim_{K\subset K'\subset K_\infty}H^1(K',T_f(r))\ot_{\cO_F}F
\]
interpolating %the \'etale Abel--Jacobi images
generalized Heegner cycles over the anticyclotomic tower. Moreover,
based on an extension of the calculations of \cite{BDP} we prove
an ``explicit reciprocity law'':
\[
\left\langle\mathcal{L}_{\frakp,\psi}(\mathbf{z}_f),
\omega_{f}\ot t^{1-2r}\right\rangle=(-c_o^{r-1})\cdot \mathscr{L}_{\frakp,\psi}(f)
\]
(\emph{cf.} \thmref{thm:walds}) relating the $p$-adic $L$-function $\mathscr{L}_{\frakp,\psi}(f)$ to the image
of the classes $\mathbf{z}_{f}$ under a variant of Perrin-Riou's \emph{big logarithm map} $\cL_{\frakp,\psi}$.
The assumption that $p=\frakp\frakpbar$ splits in $K$ and the $p$-ordinarity of $f$ are crucially used at this point. The non-ordinary case will be treated in a forthcoming work of S. Kobayashi.

With the result at hand, the proof of Theorem~A follows easily. Indeed, by the interpolation property
of $\mathscr{L}_{\frakp,\psi}(f)$, the nonvanishing of the $L$-value $L(f,\chi,r)$ in the statement
implies the nonvanishing of the value of $\mathscr{L}_{\frakp,\psi}(f)$ at $\phi_0=\psi^{-1}\chi$;
by our explicit reciprocity law, this translates into the nonvanishing of the natural image of $\mathbf{z}_f$ 
in $H^1(K_\frakp,V_f(r)\otimes\chi^{-1})$.
Combined with a suitable extension of Kolyvagin's method of Euler systems with local conditions at $p$
(see $\S\ref{S:EulerSystem}$), we then use this to establish the vanishing of ${\rm Sel}(K,V_{f,\chi})$.

\begin{intro-rem}
Under more stringent hypotheses, a version of Theorem~A was proven in \cite{cas-variation}.
The strategy followed in \emph{loc.cit.} is the same as in this paper,
but with our classes $\mathbf{z}_f$ replaced by the specializations $\nu_f(\mathfrak{Z}_\infty)$
of Howard's system of big Heegner points \cite{howard-invmath} attached to the Hida family passing through $f$.
In particular, a key ingredient in \cite{cas-variation} is the proof of a certain
``two-variable'' explicit reciprocity law, which specializes to a relation between $\mathscr{L}_{\frakp,\psi}(f)$
and the image of $\nu_f({\mathfrak{Z}_\infty})$ under $\mathcal{L}_{\pp,\psi}$.
Comparing the resulting two formulas for $\mathscr{L}_{\frakp,\psi}(f)$, the equality
\[
\nu_f(\mathfrak{Z}_\infty)=\mathbf{z}_f
\]
follows easily, yielding an important refinement of the main result of \cite{cas-higher}. 
\end{intro-rem}

Next we consider the case \[\ep(V_{f,\chi})=-1\iff -r<j<r,\]
so the central $L$-value $L(f,\chi,r)$ vanishes, and we expect the nonvanishing of Selmer groups. In \subsecref{SS:twisted}, we construct the classes $\Hclass_{f,\chi,n}\in H^1(\rcf{n},V_{f,\chi})$ over ring class fields $\rcf{n}$ of $K$. These classes are obtained by taking the $\chi$-component of the $p$-adic Abel--Jacobi image of generalized Heegner cycles, and they enjoy the properties of an \emph{anticyclotomic Euler system}. The aforementioned extension of Kolyvagin's methods to the anticyclotomic setting, which follows from a
combination of arguments developed by Nekov{\'a}{\v{r}} \cite{nekovar302} and Bertolini--Darmon \cite{BD1990},
also applies to Hecke characters $\chi$ with infinity types $(j,-j)$ with $-r<j<r$, and by these methods we obtain a proof of the following result
without the $p$-ordinary hypothesis on $f$. Put $\Hclass_{f,\chi}:=\cores_{\rcf{c}/K}(\Hclass_{f,\chi,c})$.
%where we let $\mathbf{z}_f^{\chi^{\pm}}$ denote the image of $\mathbf{z}_f$ in $H^1(K,V_{f,\chi^\pm})$.

\begin{thmb}\label{T:main1}
Assume that $\ep(V_{f,\chi})=-1$.  If $\Hclass_{f,\chi}\not =0$, then \[{\rm Sel}(K,V_{f,\chi})=F\cdot\Hclass_{f,\chi}.\]
\end{thmb}

\begin{intro-rem}
The expected extension %\foornote{which is forthcoming work of A.Shnidman.}
of the Gross--Zagier formula of \cite{zhang130} %(which covers the unramified case with $j=r$)
to generalized Heegner cycles, together with the conjectural injectivity of the $p$-adic Abel--Jacobi map
\cite[Conj.~(2.1.2)]{nekovarCRM}, would yield a proof of the implication
$L'(f,\chi,r)\neq 0\Longrightarrow\Hclass_{f,\chi}\neq 0$, for any $\chi$ as above
with $\ep(V_{f,\chi})=-1$. In these favorable circumstances,
our Theorem~B would imply conjecture (\ref{BKconj}) in the ``rank one'' case.
\end{intro-rem}

Appealing to the nonvanishing results of \cite{hsieh}, in Theorem~3.7 we show that the $p$-adic $L$-function
$\mathscr{L}_{\frakp,\psi}(f)\in\overline{\Z}_p[\![\Gamma^-_K]\!]$ is
nonzero, and hence, as $\chi$ varies, all but finitely many of the values $L(f,\chi,r)$ appearing in Theorem~A are nonzero;
our result thus covers most cases of conjecture (\ref{BKconj}) for those $\chi$.
Moreover, combined with \cite[Corollary (5.3.2)]{nekovar-parity3}, the above generic nonvanishing and our
Theorems~A and B yield a proof of the ``parity conjecture'' for $V_{f,\chi}$.

\begin{thmc}Suppose that $f$ is ordinary at $p$. Then
\[
{\rm ord}_{s=r}L(f,\chi,s)\equiv{\rm dim}_F{\rm Sel}(K,V_{f,\chi})\quad({\rm mod}\;2).
\]
That is, the equality predicted by conjecture {\rm (\ref{BKconj})} holds modulo $2$.
\end{thmc}

Finally, we note that the nontriviality of $\mathscr{L}_{\frakp,\psi}(f)$, combined with our extension
of the $p$-adic Gross--Zagier formula of \cite{BDP}, immediately yields an analogue of
Mazur's nonvanishing conjecture \cite{mazur-icm83} for generalized Heegner cycles and ranks of Selmer groups (see \thmref{T:Mazur}).

\sk

\noindent\emph{Acknowledgements.}
Fundamental parts of this paper were written during the visits of the first-named author
to the second-named author in Taipei during February 2014 and August 2014; it is a pleasure to thank NCTS
and the National Taiwan University for their hospitality and financial support.
We would also like to thank Ben Howard, Shinichi Kobayashi and David Loeffler for their comments and enlightening conversations related to this work.
%During the preparation of this paper, F.C. was partially supported by
%Grant MTM20121-34611 and by Prof.~Hida's NSF Research Grant DMS-0753991. M.-L. Hsieh was partially supported by a MOST grant 103-2115-M-002-012-MY5. 

%!TEX root = HeegnerCycles.tex
%ML, Jan 7, 2017

\def\CMP{\vartheta}
\def\vep{\varepsilon}
\def\ab{{\rm ab}}
\def\rmt{{\rm t}}
\def\rmp{{\rm p}}
\def\bbG{\bfG}
\def\can{{\rm can}}
\def\OK{\cO_{\cK}}
\def\ctame{c_o}
\def\brch{\psi}
\def\deltaK{\sqrt{-D_K}}
\def\addchar{\boldsymbol{\psi}}
\subsection*{Notation and definitions}We let $p$ be a prime and fix embeddings $\imath_p:\Qbar\hookto\Cp$ and $\imath_\infty:\Qbar\hookto\C$ throughout.
Let $\A=\A_\Q$ be the adele ring of $\Q$. Let $\addchar=\prod_q\addchar_q:\Q\bksl \A\to \C^\x$ be the standard additive character with $\addchar_\infty(x)=\exp(2\pii x)$. For each finite prime $q$, denote by $\Ord_q:\Q^\x\to \Z$ the normalized valuation with $\Ord_q(q)=1$. If $N$ is a positive integer, denote by $\bbmu_N$ the group scheme of $N$-th roots of unity.
We set $\mu_N=\bbmu_N(\Qbar)$ and $\zeta_N:=\exp(\frac{2\pii}{N})$.

If $\phi:\Z_q^\x\to\C^\x$ is a continuous character of conductor $q^n$, define the Gauss sum
\[\frakg(\phi)=\sum_{u\in (\Z/q^n\Z)^\x}\phi(u)\zeta_{p^n}^{u}.\]
By definition, $\frakg(\bfone)=1$ for the trivial character $\bfone$. If $F$ is a finite extension of $\Q_q$ and $\pi$ is an irreducible representation of $\GL_n(F)$ ($n=1,2$), we let \[\vep(s,\pi):=\vep(s,\pi,\addchar_q\circ\Tr_{F/\Q_q})\] be the local $\vep$-factor attached to the additive character $\addchar_q\circ\Tr_{F/\Q_q}$ (see \cite[Section 1.1]{schmidt_localnewforms} for the definition and basic properties). If $\chi:\Q_q^\x\to\C^\x$ is a character of conductor $q^n$, then we have
\beq\label{E:formulaepsilon}\vep(s,\chi)=\frakg(\chi^{-1})\cdot \chi(-q^n)q^{-ns},\quad \vep(s,\chi)\vep(1-s,\chi^{-1})=\chi(-1).\eeq

If $L$ is a number field or a local field, we denote by $G_L$ the absolute Galois group of $L$ and by $\cO_L$ the ring of integers of $L$.
\section{Modular curves and CM points}
\def\cU{U}
\def\Qhat{\wh\Q}
\def\Ig{{\rm Ig}}
\def\Sh{Sh}
\def\lp{i}
\def\frakqbar{\ol{\frakq}}
\def\cB{\mathcal B}

\subsection{Igusa schemes and modular curves}\label{subsec:Igusa}
Let $N\geq 3$ be an integer prime to $p$, and let $\Ig(N)_{/\Z_\setp}$ be the Igusa scheme over $\Z_\setp$, which is the moduli space parameterizing elliptic curves with $\Gamma_1(Np^\infty)$-level structure. More precisely, for each locally noetherian scheme $S$ over $\Z_\setp$, $\Ig(N)(S)$ is the set of isomorphism classes of pairs $(A,\eta)$ consisting of an elliptic curve $A$ over $S$ and a $\Gamma_1(Np^\infty)$-level structure $\eta=(\eta^\setp,\eta_p):\bbmu_N\oplus \bbmu_{p^\infty}\hookto A[N]\oplus A[p^\infty]$, an immersion as group schemes over $S$. For a non-negative integer $n$, let $Y_1(Np^n)_{/\Q}$ be the usual open modular curve of level $\Gamma_1(Np^n)$. Put
\begin{align*}U_1(Np^n)=&\stt{g\in \GL_2(\Zhat)\mid g\con \pMX{1}{*}{0}{*}\pmod{Np^n}},\\
U_0(Np^n)=&\stt{g\in \GL_2(\Zhat)\mid g\con \pMX{*}{*}{0}{*}\pmod{Np^n}}.
\end{align*}

Letting $\bbH$ be the complex upper half-plane, the curve $Y_1(Np^n)$ admits the complex uniformization
\[Y_1(Np^n)(\C)=\GL_2(\Q)^+\bksl \bbH\times\GL_2(\Qhat)/U_1(Np^n),\]
where $\GL_2(\Q)^+$ is the subgroup of $\GL_2(\Q)$ with positive determinants. Since
the generic fiber $\Ig(N)_{/\Q}$ is given by
\[\Ig(N)_{/\Q}=\prolim_n Y_1(Np^n)_{/\Q},\]
this yields a map
\[ \bbH\times \GL_2(\Qhat)\to \Ig(N)(\C),\quad x=(\tau_x,g_x)\mapsto [(A_x,\eta_x)]. \]

We now give an explicit construction of pairs $(A_x,\eta_x)$ of complex elliptic curves with $\Gamma_1(Np^\infty)$-level structure.
Let $V=\Q e_1\oplus \Q e_2$ be the two-dimensional $\Q$-vector space equipped with the symplectic pairing
\[\pair{ae_1+be_2}{ce_1+de_2}=ad-bc,\]
and let $\GL_2(\Q)$ act on $V$ from the right via
\[(xe_1+ye_2)\cdot\pMX{a}{b}{c}{d}=(xa+cy)e_1+(xb+yd)e_2.\]
For $\tau\in \bbH$, define the map $\rmp_\tau: V\to\C$ by 
\[
\rmp_\tau(ae_1+be_2)=a\tau+b.
\]
Then $\rmp_\tau$ induces an isomorphism $V_\R:=\R\ot_\Q V\iso \C$. Let $\bfL$ be the standard lattice $\Z e_1\oplus \Z e_2$, and for every $g=\pMX{a}{b}{c}{d}\in \GL_2(\Qhat)$ define the $\Z$-lattice $\bfL_g\subset V$ by
\[\bfL_g:=(\Zhat e_1\oplus \Zhat e_2)g'\cap V, \]where $g'$ is the main involution defined by \[g'=\pMX{d}{-b}{-c}{a}=g^{-1}\det g.\]
The $\C$-pair $(A_x,\eta_x)$ attached to $x=(\tau_x,g_x)\in\bbH\times\GL_2(\Qhat)$ is then given by
%\beq\label{E:EC1}
\[A_x=\C/L_x,\quad L_x:=\rmp_{\tau_x}(\bfL_{g_x}),\]
%\eeq
and the $\Gamma_1(Np^\infty)$-level structure $\eta_x=(\eta_x^\setp,\eta_{x,p})$ is given by the immersions
%\beq\label{E:EC2}
\[\begin{aligned}\eta^\setp_x\colon&\mu_{N}\hookto N^{-1}\Z/\Z\ot L_x,&\zeta_N^j\mapsto  \rmp_{\tau_x}(j/N\ot e_2 g_x'),\\
\eta_{x,p}\colon&\mu_{p^\infty}\hookto \Qp/\Zp\ot L_x,& \zeta_{p^n}^j\mapsto \rmp_{\tau_x}(j/p^n\ot e_2g'_x).\end{aligned}
\]
Here we have used the identification $\Q/\Z\ot_\Z\bfL g_x'=\Q/\Z\ot_\Z \bfL_{g_x}$.
%\eeq
The lattice $L_x\subset \C$ is called the \emph{period lattice} of $A_x$ attached to the standard differential form $dw$,
with $w$ the standard complex coordinate of $\C$.

\subsection{Modular forms}\label{S:GME}
We briefly recall the definitions and standard facts about geometric and $p$-adic 
modular forms. The basic references are \cite{Katz350}, \cite{katzCMfields} and \cite{hida_yellowbook}.
\subsubsection*{Geometric modular forms}
\def\baseR{\Z_\setp}
\begin{defn}
Let $k$ be an integer, and let $B$ be a $\baseR$-algebra. A geometric modular form $f$ of weight $k$ on $\Ig(N)$ defined over $B$
is a rule assigning to every triple $(A,\eta,\om)_{/C}$ over a $B$-algebra $C$, consisting of a point $[(A,\eta)]\in \Ig(N)(C)$ and a basis $\om$ of $H^0(A,\omega_{A/C})$ over $C$, a value $f(A,\eta,\om)\in C$ such that the following conditions are satisfied:
\begin{mylist}
\item[(G1)] $f(A,\eta,\om)=f(A',\eta',\om')\in C$ if $(A,\eta,\om)\iso (A',\eta',\om')$ over $C$.
\item[(G2)] If $\varphi:C\to C'$ is any $B$-algebra homomorphism, then
\[f((A,\eta,\om)\ot_C C')=\varphi(f(A,\eta,\om)).\]
\item[(G3)]$f(A,\eta,t\om)=t^{-k}f(A,\eta,\om)$ for all $t\in C^\x$.
\item[(G4)]Letting $({\rm Tate}(q),\eta_{\can},\om_{\can})$ be the Tate elliptic curve $\bbG_m/q^\Z$
with the canonical level structure $\eta_{\can}$ and the canonical differential $\om_{\can}$ over $\Z(\!(q)\!)$,
the value $f({\rm Tate}(q),\eta_{\can},\om_{\can})$ lies in $B\powerseries{q}$. We call
\[
f({\rm Tate}(q),\eta_{\can},\om_{\can})\in B\powerseries{q}
\]
the algebraic Fourier expansion of $f$.
\end{mylist}
\end{defn}
If $f$ is a geometric modular form of weight $k$ defined over a subring $\cO\subset\C$, then $f$ gives rise to a holomorphic function $\bff:\bbH\x\GL_2(\Qhat)\to\C$ by the rule
\[\bff(x)=f(A_x,\eta_x,2\pii dw),\quad  x\in \bbH\xx\GL_2(\Qhat),\]
where $w$ is the standard complex coordinate of $A_x=\C/L_x$. This function $\bff$ satisfies the transformation rule:
\[\bff(\al\tau,\al g)=(\det \al)^{-\frac{k}{2}}J(\al,\tau)^k\cdot \bff(\tau,g)\quad(\al\in\GL_2(\Q)^+),\]
where $J:\GL_2(\R)^+\xx\bbH\to\C$ is the automorphy factor defined by
\[J(g,\tau)=(\det g)^{-\onehalf}\cdot (c\tau+d)\quad (g=\pMX{a}{b}{c}{d}).\]
Moreover, the function $\bff(-,1):\bbH\to\C$ is a classical elliptic modular form of weight $k$ with analytic Fourier expansion
\[\bff(\tau,1)=\sum_{n\geq 0}\bfa_n(f) e^{2\pii n\tau},\]%\text{ for }\Im \tau\gg 0,\]
and we have the equality between algebraic and analytic Fourier expansions (\cf \cite[\S 1.7]{katzCMfields})
\[f({\rm Tate}(q),\eta_{\can},\om_{\can})=\sum_{n\geq 0}\bfa_n(f) q^n\in \cO\powerseries{q}.\]
We say that $f$ is of level $\Gamma_0(Np^n)$ if $\bff(\tau,gu)=\bff(\tau,g)$ for all $u\in U_0(Np^n)$.

\subsubsection*{\padic modular forms}
Let $R$ be a \padic ring, and let $\widehat{\Ig}(N)_{/R}:=\dirlim_m \Ig(N)_{/R/p^m R}$ be the formal completion of $\Ig(N)_{/R}$. Define the space $V_p(N,R)$ of \padic modular forms of level $N$ by
\begin{align*}V_p(N,R):=\;&H^0(\wh\Ig(N)_{/R},\cO_{\wh\Ig(N)_{/R}})\\
=\;&\prolim_m H^0(\Ig(N),\cO_{\Ig(N)}\ot R/p^mR).
\end{align*}
Thus elements in $V_p(N,R)$ are formal functions on the Igusa tower $\Ig(N)$.
We say that a \padic modular form $f$ is of weight $k\in\Zp$ if for every $u\in\Zp^\x$, we have
\beq\label{E:transformpadic}f(A,\eta)=u^{-k}f(A,\eta^\setp,\eta_p u),\quad [(A,\eta)]=[(A,\eta^\setp,\eta_p)]\in\wh\Ig(N)_{/R}.\eeq

If $f$ is a geometric modular form defined over $R$, then we can associate to $f$ a \padic modular form $\wh f$,
called the \emph{\padic avatar} of $f$, as follows. Let $C$ be a complete local $R$-algebra, and let $(A,\eta)$ be an elliptic curve with $\Gamma_1(Np^\infty)$-level structure. The $p^\infty$-level structure $\eta_p:\bbmu_{p^\infty}\hookto A[p^\infty]$ induces an isomorphism $\wh\eta_{p}:\wh\bbG_m\iso \wh A$ (here $\wh A$ is the formal group of $A$), which in turn gives rise to a differential
$\wh\om(\eta_p)\in\Lie(A)=\Lie(\wh A)\iso \Lie(\wh\bbG_m)=C$.
Then $\wh f$ is the \padic modular form defined by the rule
\[\wh f(A,\eta)=f(A,\eta,\wh\om(\eta_p)),\quad[(A,\eta)]\in\wh\Ig(N)_{/R}
\]
(\cf \cite[(1.10.15)]{katzCMfields}). It follows from the definition that if $f$ is a geometric modular form of weight $k$ and level $\Gamma_0(Np^n)$,
then $\wh f$ is a \padic modular form of weight $k$.

\subsection{CM points (I)}\label{subsec:CMpts}
Let $\cK$ be an imaginary quadratic field of discriminant $-D_\cK<0$, and
denote by $z\mapsto \ol{z}$ the complex conjugation on $\C$, which gives the non-trivial automorphism of $\cK$.
In this section, we assume that $p>2$ is a prime split in $\cO_\cK$ and write
\[p\cO_\cK=\frakp\frakpbar,\]where $\frakp$ is the prime ideal above $p$ determined by the embedding $\Qbar\hookto\Cp$.

Define $\CMP\in\cK$ by
%\[\CMP=\frac{\delta}{2}\text{ if }4\mid D_\cK\,;\, \CMP=\frac{\delta^2-\delta}{2}\text{ if } 4\ndivides D_\cK.\]
\[\CMP=\frac{D'+\deltaK}{2},\quad D'=\begin{cases}D_\cK&\text{ if }2\ndivides D_\cK,\\
D_\cK/2&\text{ if } 2\divides D_\cK.
\end{cases}\]
Then $\cO_\cK=\Z+\Z\CMP$ and $\CMP\ol{\CMP}$ is a local uniformizer of $\Q_q$ for $q$ ramified in $\cK$. %Put\[\deltaK=\sqrt{-D_K}=\CMP-\ol{\CMP}.\]
If $M$ is a positive integer, we decompose $M=M^+M^-$, with the prime factors of $M^+$ (resp. $M^-$) split (resp. inert or ramified) in $K$.
For each prime  $q=\frakq\frakqbar$ split in $\cK$, we write \[\Z_q\ot_\Z\cO_\cK=\Z_q e_{\frakqbar}\oplus \Z_q e_\frakq,\]
where $e_\frakq$ and $e_{\frakqbar}$ are the idempotents in $\Z_q\ot_\Z\cO_\cK$ corresponding to $\frakq$ and $\frakqbar$, respectively.

We assume that $N\OK=\frakN\ol{\frakN}$ for some ideal $\frakN$ of $\cO_\cK$.
Let $c$ be a positive integer, let $\cO_c:=\Z+c\cO_\cK$ be the order of conductor $c$, and let $\rcf{c}$ be the ring class field of $K$ of conductor $c$. Let $\fraka$ be a fractional ideal of $\cO_c$, and let $a\in \wh\cK^\x$ with $a\wh\cK\cap \cO_c=\fraka$. To the ideal $\fraka$ and the finite idele $a$, we associate a $\C$-pair $(A_\fraka,\eta_a)$ of complex CM elliptic curves with $\Gamma_1(Np^\infty)$-level structure as follows. Define $A_\fraka$ to be the complex elliptic curve $\C/\fraka^{-1}$. %To describe the level structure $\eta_a$,
For each prime $q\divides pN$, let $\frakq$ be the prime of $\OK$ above $q$ with $\frakq\divides \frakN\frakp$,
and let $a_\frakq\in\Q_q$ be the $\frakq$-component of $a$. We then
have $(\Z_q\ot_\Z\fraka^{-1})\cap \Q_qe_\frakq=\Z_q a_\frakq^{-1} c e_\frakq$ and the exact sequence
\[A_\fraka[\frakq^\infty]=\bbmu_{q^\infty}\ot a_\frakq^{-1} c e_\frakq\hookto A_\fraka[q^\infty]\surjto \Q_q/\Z_q\ot a_{\frakqbar}^{-1} e_{\frakqbar},\]
and we define $\eta_a=(\eta^\setp_a,\eta_{a,p}):\bbmu_N\oplus\bbmu_{p^\infty}\isoto A_\fraka[\frakN]\oplus A_\fraka[\frakp^\infty]\hookto A_\fraka[N]\oplus A_\fraka[p^\infty]$ to be the embedding determined by the isomorphism $\bbmu_{q^n}\isoto A_\fraka[\frakq^n]$ sending
\[\zeta_{q^{n}}^j\mapsto \begin{cases}j/q^{n}\ot a_\frakq^{-1} q^{\Ord_q(c)}e_\frakq&\text{ if }q\divides N^+p,\\
 j/q^n\ot a_\frakq^{-1}&\text{ if }q\divides N^-.\end{cases}\]
Denote by $\cV$ the valuation ring $\iota_p^{-1}(\cO_{\Cp})\cap \cK^\ab$. It follows from the theory of complex multiplication 
\cite[18.6, 21.1]{shimura_ABVCM} combined with the criterion of Serre--Tate \cite{serre_tate} that 
$(A_\fraka,\eta_a)$ actually descends to a discrete valuation ring $\cV_0$ inside $\cV$. Thus $[(A_\fraka,\eta_a)]$ is defined over $\cV_0$ and belongs to $\Ig(N)(\cV_0)$.
We call $[(A_\fraka,\eta_a)]\in \Ig(N)(\cV)$ the CM point attached to $(\fraka,a)$.

If $\fraka$ is a prime-to-$\frakN\frakp$ integral ideal of $\cO_c$, we write $(A_\fraka,\eta_\fraka)$ for the triple $(A_\fraka,\eta_a)$ with $\frakq$-component $a_\frakq=1$ for every $\frakq\divides\frakN\frakp$. If $\fraka=\cO_c$, we write $(A_c,\eta_c)$ for $(A_{\cO_c},\eta_{\cO_c})$. In this case, we see immediately from the construction that $A_\fraka=A_c/A_c[\fraka]$ and the isogeny $\lam_\fraka:A_c\to A_\fraka$ induced by the quotient map $\C/\cO_c\to\C/\fraka^{-1}$ yields $\eta_\fraka=\lam_\fraka\circ \eta_c$.

\subsection{CM points (II)}\label{subsec:CMII} We give an explicit complex uniformization of the CM points introduced above.
Consider the embedding $K\hookto M_2(\Q)$ given by\[a\CMP+b\mapsto \pMX{a(\CMP+\ol{\CMP})+b}{-a\CMP\ol{\CMP}}{a}{b}.\]
For each $g\in\GL_2(\Qhat)$, denote by $[(\CMP,g)]$ the image of $(\CMP,g)$ in $\prolim_n Y_1(Np^n)(\C)=\Ig(N)(\C)$.
Shimura's reciprocity law for CM points (\cf\cite[Cor.\,4.20]{hida_yellowbook}) implies that $[(\CMP,g)]\in \Ig(N)(K^\ab)$ and
\beq\label{E:ShimuraCM}\rec_\cK(a)[(\CMP,g)]=[(\CMP,\ol{a}g)]\quad(a\in\wh\cK^\x),\eeq
where $\rec_\cK\colon\cK^\x\bksl \wh\cK^\x\to \Gal(\cK^{\rm ab}/\cK)$ is the \emph{geometrically normalized} reciprocity law map.

Let $\ctame=\ctame^+\ctame^-$ be a positive integer prime to $p$ and fix a decomposition $\ctame^+\OK=\frakC\ol{\frakC}$. Define $\varsigma^{(\infty)}=(\varsigma_q)\in\GL_2(\Qhat)$ by $\varsigma_q=\pMX{1}{0}{0}{1}$ if $q\ndivides \ctame^+N^+p$, and
\begin{align*}
\varsigma_q=&(\ol{\CMP}-\CMP)^{-1}\pMX{\ol{\CMP}}{\CMP}{1}{1}\in\GL_2(K_\frakq)=\GL_2(\Q_q)\text{ if $q=\frakq\frakqbar$ with $\frakq\divides\frakC\frakN\frakp$}.
%\varsigma_q=&\pMX{0}{-1}{-1}{0}\text{ if }q\divides N^-.
\end{align*}
Let $c=\ctame p^n$ with $n\geq 0$. We define $\gamma_c=\prod_q\gamma_{c,q}\in \GL_2(\Qhat)$ by $\gamma_{c,q}=1$ if $q\ndivides cNp$,
\begin{align*}
\gamma_{c,q}=&\pMX{q^{\Ord_q(c)}}{1}{0}{1}\text{ if $q=\frakq\frakqbar$ with $\frakq\divides\frakC\frakN\frakp$},\\
\gamma_{c,q}=&\pDII{1}{q^{\Ord_q(c)-\Ord_q(N)}}\text{ if $q\divides c^-N^-$}.
\end{align*}
Let $\xi_c:=\varsigma^{(\infty)}\gamma_c\in\GL_2(\Qhat)$ be the product.
An elementary computation shows that $\cO_c=\rmp_\CMP(\bfL_{\xi_c})$ and that for $q=\frakq\frakqbar$ with $\frakq\divides \frakC\frakN\frakp$, we have
\beq\label{E:levelstr} \varsigma_{q}\pDII{a}{b}=(ae_{\frakqbar}+be_\frakq)\varsigma_{q}\quad (a,b\in\Q^\x_q),\eeq
and
\beq\label{E:levelstr2}\xi_{c,q}':\Z_qe_1\oplus \Z_q e_2\iso \Z_q\ot_\Z\cO_c,\quad \rmp_\CMP(e_2\xi'_{c,q})=q^{\Ord_q(c)}e_{\frakq},\eeq
so we have $[(\CMP,\xi_{c})]=[(A_c,\eta_c)]$. Define
\[x_{c}:=[(A_c,\eta_c)]=[(\CMP,\xi_{c})]\in\Ig(N)(\C). \]
In general, if $a\in\wh\cK^{(cp)\x}$ and $\fraka=a\wh\cO_c\cap \cK$ is a fractional ideal of $\cO_c$, we let
\[\sg_\fraka:=\rec_K(a^{-1})|_{\cK_c(\frakp^\infty)}\in \Gal(\cK_c(\frakp^\infty)/\cK),\]
where $\cK_c(\frakp^\infty)$ is the compositum of $\rcf{c}$ and the ray class field of $K$ of conductor $\frakp^\infty$.
Thus $\sg_\fraka$ is the image of $\fraka$ under the classical Artin map. We have
\beq\label{E:GaloisAction}x_\fraka:=[(A_\fraka,\eta_a)]=[(\CMP,\ol{a}^{-1}\xi_c)]=x_c^{\sg_\fraka}\in \Ig(N)(\cK_c(\frakp^\infty)).\eeq
Here the first equality can be verified by noting that the main involution induces the complex conjugation on $\A_K^\x$ and using \eqref{E:levelstr}, and the second equality follows from Shimura's reciprocity law for CM points \eqref{E:ShimuraCM}.
\subsection{CM periods}\label{SS:CMperiods}
Let $\wh\Q_p^{\rm ur}$ be the \padic completion of the maximal unramified extension $\Qp^{\rm ur}$ of $\Qp$,
and let $\cW$ be the ring of integers of $\wh\Q_p^{\rm ur}$. If $\fraka$ is a prime-to-$\frakp\frakN$ fractional ideal of $\cO_c$ with $p\ndivides c$,
then $(A_\fraka,\eta_\fraka)$ has a model defined over $\cV^{\rm ur}:=\cW\cap \cK^{\rm ab}$. In the sequel, we shall still denote this model by $(A_\fraka,\eta_\fraka)$ and simply write $A$ for $A_{\OK}$.

Fix a \Neron differential $\om_A$ of $A$ over $\cV^{\rm ur}$. There exists a unique prime-to-$p$ isogeny $\lam_\fraka:A_\fraka\to A$
inducing the identity map on both the complex Lie algebras $\C=\Lie A_\fraka(\C)\to \C=\Lie A(\C)$ via the complex uniformizations and
on the $\frakp$-divisible groups $\bbmu_{p^\infty}=A_\fraka[\frakp^\infty]\to \bbmu_{p^\infty}=A[\frakp^\infty]$ via the level structures at $p$.
Letting $\om_{A_\fraka}:=\lam_\fraka^*\om_A$ be the pull-back of $\om_A$, we see that there exists
a pair $(\Omega_\cK,\Omega_p)\in \C^\x\xx \cW^\x$ %, depending on the choice of $\om_A$,
such that \[\Omega_\cK\cdot 2\pii dw=\Omega_p\cdot  \wh\om(\eta_{\fraka,p})=\om_{A_\fraka},\]
where $w$ is the standard complex coordinate of $\C/\fraka^{-1}=A_\fraka(\C)$. The pair $(\Omega_\cK,\Omega_p)$ are called the complex and \padic periods of $\cK$. Note that the ratio $\Omega_\cK/\Omega_p$ does not depend on the choice of \Neron differential $\om_A$.

\section{Anticyclotomic \padic $L$-functions}

In this section, we review the anticyclotomic \padic $L$-functions that were originally constructed
in \cite{Miljan:1}, \cite{BDP} and \cite{hsieh} from various points of view. Our purpose is to extend their
interpolation formulae %of $\sL_{\frakp,\brch}(f)$
to include $p$-ramified characters and to prove the nonvanishing of these \padic $L$-functions,
so we find it more convenient to adopt the approach of \cite{Miljan:1}, based on the use of Serre--Tate coordinates.
\subsection{$t$-expansion of \padic modular forms}
Let $\bfx=[(A_0,\eta)]$ be a point in the Igusa tower $\Ig(N)(\Fpbar)$ and let $\wh S_\bfx\hookto \Ig(N)_{/\cW}$ be the local deformation space of $\bfx$ over $\cW$. The $p^\infty$-level structure $\eta_p$ determines a point $P_\bfx\in T_p(A_0^{\rm t})$, where $A_0^\rmt$ is the dual abelian variety of $A_0$ and $T_p(A_0^{\rm t})=\prolim_n A_0^{\rm t}[p^n](\Fpbar)$ is the \padic Tate module of $A_0^{\rm t}$.
Let $\lam_{\can}\colon A_0\iso A_0^{\rm t}$ be the canonical principal polarization.

For each deformation $\cA_{/R}$ over a local Artinian ring $(R,\frakm_R)$, let
$q_\cA\colon T_p(A_0)\xx T_p(A_0^{\rm t})\to 1+\frakm_R$ be the Serre--Tate bilinear form attached to $\cA_{/R}$ (see \cite{katz_ST}).
The \emph{canonical Serre--Tate coordinate} $t\colon\wh S_\bfx\to\wh\bbG_m$ is defined by \[t(\cA):=q_\cA(\lam_{\can}^{-1}(P_\bfx),P_\bfx)\]
and yields an identification $\cO_{\wh S_\bfx}=\cW\powerseries{t-1}$.
\def\diff{{\mathrm d}}

Let $f\in V_p(N,\cW)$ be a $p$-adic modular form over $\cW$.
The $t$-expansion $f(t)$ of $f$ around $\bfx$ is defined by \[f(t)=f|_{\wh S_\bfx}\in \cW\powerseries{t-1},\]
and we let $\diff f$ be the \padic measure on $\Zp$ such that
\[\int_{\Zp}t^x \diff f(x)=f(t).\]
Moreover, if $\phi:\Zp\to\cO_{\Cp}$ is any continuous function, we define $f\ot\phi(t)\in\cO_{\Cp}\powerseries{t-1}$ by
\[f\ot\phi(t)=\int_{\Zp}\phi(x)t^x \diff f=\sum_{n\geq 0}\int_{\Zp}\phi(x){x\choose n} \diff f(x)\cdot (t-1)^n.\]

\begin{lm}\label{L:1}
If $\phi:\Zp/p^n\Zp\to\cO_{\Cp}$, then
\[f\ot\phi(t)=p^{-n}\sum_{u\in \Z/p^n\Z}\sum_{\zeta\in\mu_{p^n}}\zeta^{-u}\phi(u)f(t\zeta).\]
%In particular, if $\phi:(\Z/p^n\Z)^\x\to\cO_{\Cp}^\x$ is a primitive Dirichlet character, then
%\[f\ot\phi(t)=p^{-n}\frakg(\phi)\sum_{\zeta\in\mu_{p^n}}\phi^{-1}(b)f(t\zeta).\]
If $\phi:\Zp\to\Zp$ is $z\mapsto z^k$, then
\[f\ot\phi(t)=\left[t\frac{d}{dt}\right]^k(f).\]
\end{lm}
\begin{proof} This is well-known. For example, see \cite[\S 3.5 (5)]{hidaBlue}.
%For each $w\in\mu_{p^\infty}$,
%\begin{align*}f\ot\phi(w)=&\int_{\Zp}\phi(x)w^x df(x)=\int_{\Zp}\sum_{\zeta\in \mu_{p^n}}\wh\phi(\zeta)\zeta^x w^x \diff f(x)\\
%=&\sum_{\zeta\in\mu_{p^n}}\wh\phi(\zeta)f(w\zeta ).
%\end{align*}
%Since the group $\mu_{p^{\infty}}$ is dense in $\Spec\,\cO_{\Cp}\powerseries{t-1}(\Cp)$, we find that
%\[f\ot\phi(t)=\sum_{\zeta\in\mu_{p^n}}\wh\phi(\zeta)f(t\zeta)=\frac{1}{p^n}\sum_{b\in\Z/p^n\Z}\sum_{\zeta\in\mu_{p^n}}\zeta^{-b}\phi(b)f(t\zeta).\]
\end{proof}

\subsection{Serre--Tate coordinates of CM points}
Suppose that $c$ is a positive integer with $p\nmid c$.
Let $\fraka$ be a prime-to-$c\frakN p$ integral ideal of $\cO_c$, and let $a\in\wh\cK^{(c\frakN p)\x}$ be such that $\fraka=a\wh\cO_c\cap K$. Define $\rmN(\fraka)$ by\[\begin{aligned}\rmN(\fraka):=&\;\text{degree of the $\Q$-isogeny }\C/\OK\to\C/\fraka^{-1}\\
=&\;c^{-1}\#(\cO_c/\fraka)=c^{-1}\abs{a}_{\AK}^{-1}.\end{aligned}\] Let $x_\fraka=[(A_\fraka,\eta_\fraka)]\in\Ig(N)(\cV)$ be the CM point attached to $\fraka$ and let $t$ be the canonical Serre--Tate coordinate of $\bfx_\fraka:=x_\fraka\ot_\cV\Fpbar$.
We will use the following notation: for each $z\in\Qp$, set
\beq\label{E:notation}\bfn(z):=\pMX{1}{z}{0}{1}\in\GL_2(\Qp)\subset\GL_2(\Qhat).\eeq
Put
\[x_\fraka*\bfn(z):=[(\CMP,\ol{a}^{-1}\xi_c\bfn(z))]\in\Ig(N)(\cV).\]
\begin{lm}\label{L:STcoordinate}Let $u\in\Zp$. We have $(x_\fraka*\bfn(up^{-n}))\ot\Fpbar=\bfx_\fraka$, and \[t(x_\fraka*\bfn(up^{-n}))=\zeta_{p^n}^{-u\rmN(\fraka)^{-1}\deltaK^{-1}}.\]
\end{lm}
\begin{proof}Let $(\cA,\eta_\cA)_{/\cV_0}$ be a model of the CM elliptic curve $(A_\fraka,\eta_\fraka)$ over a discrete valuation ring $\cV_0\subset\cV$. Let $\ol{\cA}=\cA\ot\Fpbar$.
Recall that $e_\frakp$ and $e_{\frakpbar}$ are idempotents in $\cK\ot_\Q\cK\hookto\Qp\ot_\Q\cK$ corresponding to $\frakp$ and $\frakpbar$ respectively. In fact \[e_\frakp=\frac{\CMP\otimes 1-1\otimes\ol{\CMP}}{\CMP-\ol{\CMP}}\,,\quad e_{\bar\frakp}=\frac{\CMP\otimes 1-1\otimes \CMP}{\CMP-\ol{\CMP}},\] so we have  \beq\label{E:idempotent}\rmp_\CMP(e_1\varsigma'_p)=e_{\frakpbar}\,,\quad \rmp_\CMP(e_2\varsigma_p')=e_\frakp.\eeq The complex uniformization $\al\colon(\C/\fraka^{-1},\eta_\fraka)\iso (\cA,\eta_{\cA})_{/\C}$ yields the identifications \[\al\colon\Zp e_{\frakpbar}\oplus \Zp e_\frakp=\Zp\ot_\Z\fraka^{-1}\iso T_p(\cA)\]
and
\[\al\colon\bbmu_{p^\infty}\ot\Zp e_\frakp\iso \wh \cA[p^\infty],\quad \al\colon \Zp e_{\frakpbar}\iso T_p(\ol{\cA}).\]
Here $\wh \cA$ is the formal group attached to $\cA_{/\cV_0}$. Let $\bfe_A\colon T_p(\cA)\xx T_p(\cA^\rmt)\to\Zp$ be the Weil pairing.
Let $\lam_{\CMP}\colon\cA\to \cA^\rmt$ be the prime-to-$p$ polarization induced by the Riemann form $\pair{z}{w}_{\CMP}=(\Im\CMP)^{-1}\Im(z\ol{w})$ on $\C/\fraka^{-1}$. The complex uniformization $\al^\rmt=\al\circ\lam_{\CMP}\colon\C\to \cA^{\rmt}(\C)$ induces $\al^\rmt\colon\C/\rmN(\fraka)\fraka^{-1}\iso \cA^\rmt(\C)$ and $\al^\rmt\colon\Zp e_{\frakpbar}\oplus\Zp e_\frakp\iso T_p(\cA^\rmt)$ with $\al^\rmt:\Zp e_{\frakpbar}\iso T_p(\ol{\cA}^\rmt)$. By \cite[Theorem~1, page~237]{mumford_abv} and \eqref{E:idempotent}, we have
\[\bfe_\cA(a\al(e_{\frakpbar})+b\al(e_\frakp),c\al^\rmt(e_{\frakpbar})+d\al^\rmt(e_\frakp))=-(ad-bc)(\CMP-\ol{\CMP})^{-1}=-(ad-bc)\deltaK^{-1}.\]
(Note the sign $-1$.) The canonical polarization $\lam_\can:\cA\iso \cA^\rmt$ is given by $\al(z)\mapsto \al^\rmt(\rmN(\fraka)z)$.

Let $y$ be the complex point $(\CMP,\ol{a}^{-1}\xi_c\bfn(up^{-n}))$ and let $(\cB,\eta_\cB)_{/\cV_0}$ be a model of $(A_y,\eta_y)$ over $\cV_0$
(enlarging $\cV_0$ if necessary), so $[(\cB,\eta_\cB)]=[y]$. The period lattice $L_y$ of $\cB(\C)$ is given by \[L_y=\rmp_\CMP(L'),\quad L'=(\Zhat e_1\oplus \Zhat e_2)\bfn(-up^{-n})\xi_c' a^{-1}\cap V.\]
By a direct computation and \eqref{E:idempotent}, we find that
%\[\Zp\ot_\Z L_y=\Zp(e_{\frakpbar}-\frac{u}{p^n} e_\frakp)\oplus\Zp e_\frakp,\]
\begin{align*}\Zp\ot_\Z L_y&=\rmp_\CMP(\Zp e_1\oplus \Zp e_2\pMX{1}{-up^{-n}}{0}{1}\gamma_{c,p}'\varsigma_p')\\
&=\rmp_\CMP((\Zp (e_1-up^{-n}e_2)\oplus\Zp e_2)\varsigma_p')\\
&=\Zp (e_{\bar\frakp}-up^{-n}e_\frakp)\oplus\Zp e_\frakp,
\end{align*}
so the complex uniformization $\beta\colon(\C/L_y,\eta_y)\iso (\cB,\eta_{\cB})_{/\C}$ induces the identification
\[\beta\colon\Zp(e_{\frakpbar}-\frac{u}{p^n} e_\frakp)\oplus\Zp e_{\frakpbar}\iso T_p(\cB).\]

With the above preparations, we see that over $\C$ there are natural isomorphisms
\[\cA/\cA[\frakp]\iso \cB/\cB[\frakp]\iso \C/\fraka^{-1}\frakp_c^{-1}\quad(\frakp_c=\frakp\cap\cO_c)\]
induced by the inclusions of $L_y$ and $\fraka^{-1}$ in $\fraka^{-1}\frakp_c^{-1}$, which
extend uniquely to an isomorphism $\cA/\cA[\frakp]\iso \cB/\cB[\frakp]$ over $\cV_0$ (\cite[Prop. 2.7]{faltings_chai:DAV}).
By construction, $\cA[\frakp]$ and $\cB[\frakp]$ are connected components of $\cA[p]$ and $\cB[p]$, so we get the isomorphism $(\ol{\cA},\eta_{\ol{\cA}})^{\sg_p}\iso(\ol{\cB},\eta_{\ol{\cB}})^{\sg_p}$, where $(-)^{\sg_p}$ denotes the conjugate of the $p$-th power Frobenius $\sg_p$, and hence $(\ol{\cA},\eta_{\ol{\cA}})\iso (\ol{\cB},\eta_{\ol{\cB}})$. This shows that $[y]\ot\Fpbar=\bfx_\fraka$. To compute the value $t(\cB)$, we note that $P_{\bfx_\fraka}=\al^\rmt(e_{\frakpbar})$ and that the Weil pairing of $\ol{\cA}$ induces $E_\cB\colon\wh\cB[p^\infty]\xx T_p(\ol{\cA}^\rmt)\to \wh\bbG_m$ so that $E_\cB(\beta(p^{-n}e_\frakp),\al^\rmt(e_{\frakpbar}))=\zeta_{p^n}^{\deltaK^{-1}}$ (\cite[page 150]{katz_ST}). For a sufficiently large integer $m$, we have
\begin{align*}
t(\cB)&=q_\cB(\lam_{\can}^{-1}(P_{\bfx_\fraka}),P_{\bfx_\fraka})=q_\cB(\rmN(\fraka)^{-1}\al(e_{\frakpbar}),\al^\rmt(e_{\frakpbar}))\\
&=E_\cB(``p^m"\al(p^{-m}e_{\frakpbar}),\al^\rmt(e_{\frakpbar}))^{\rmN(\fraka)^{-1}},
\end{align*}
where $``p^m":\ol{\cA}[p^m](\Fpbar)\to \wh\cA$ is the Drinfeld lift map. To compute the lift, from the diagram \beq\label{E:diagram}\xymatrix{ 0\ar[r]&\bbmu_{p^\infty}\ot \Zp e_\frakp\ar[r]\ar[d]^{\wr}&\Zp\ot \cK/L_y\ar[r]\ar[d]^{\wr}_\beta&\Qp/\Zp\ot \Zp e_\frakp\ar[d]^{\wr}_\al\ar[r]&0\\
0\ar[r]&\wh\cB[p^\infty]\ar[r]&\cB[p^\infty]\ar[r]&\Qp/\Zp\ot T_p(\ol{\cA})\ar[r] &0
}\eeq
we can see that the $p^m$-torsion point $\al(p^{-m}e_{\frakpbar})\in p^{-m}\Zp/\Zp\ot T_p(\ol{\cA})=\cA[p^m](\Fpbar)$ has a lift $\beta(p^{-m}e_{\frakpbar}-up^{-m-n}e_\frakp)\in\cB[p^\infty]$, so the Drinfeld lift $``p^m"\al(p^{-m}e_{\frakpbar})$ is given by $\beta(-up^{-n}e_\frakp)\in\wh\cB[p^\infty]$. Hence, we obtain
\[
t(\cB)=E_\cB(\beta(p^{-n}e_\frakp),\al^\rmt(e_{\frakpbar}))^{-u\rmN(\fraka)^{-1}}=\zeta_{p^n}^{-u\rmN(\fraka)^{-1}\deltaK^{-1}}.\qedhere\]
\end{proof}

\begin{prop}\label{P:1}Let $f\in V_p(N,\cW)$ be a \padic modular form with $t$-expansion $f(t)$ around $x_\fraka\ot\Fpbar$. Put
\[f_\fraka(t):=f(t^{\rmN(\fraka)^{-1}\deltaK^{-1}}).\]
If $n$ is a positive integer and $\phi\colon(\Z/p^n\Z)^\x\to \cO_{\Cp}^\x$ is a primitive Dirichlet character, then
\[f_\fraka\ot \phi(x_\fraka)=p^{-n}\frakg(\phi)\sum_{u\in(\Z/p^n\Z)^\x}\phi^{-1}(u)\cdot f(x_\fraka*\bfn(up^{-n})).\]
%If $\phi=\bfone$ is the trivial character on $\Zp^\x$, then
%\[f_\fraka\ot \bfone(x_\fraka)=f(x_\fraka)-p^{-1}\sum_{u\in\Z/p\Z} f(x_\fraka*\bfn(up^{-1})).\]
\end{prop}
\begin{proof}This follows from \lmref{L:1} combined with \lmref{L:STcoordinate}.%If $\phi\not =\bfone$, then by \lmref{L:1} we have
%\[f_\fraka\ot \phi(t)|_{t=1}=p^{-n}\frakg(\phi)\phi(-1)\sum_{u\in (\Z/p^n\Z)^\x}\phi^{-1}(u)\cdot f(\zeta_{p^n}^{u\rmN(\fraka)^{-1}}),\]
%and the assertion follows from \lmref{L:STcoordinate} immediately. The case $\b\qedhere
\end{proof}

\subsection{Anticyclotomic \padic $L$-functions}
\subsubsection*{Hecke characters and \padic Galois characters}
A Hecke character $\chi:\cK^\x\bksl \A_\cK^\x\to\C^\x$ is called a Hecke character of infinity type $(m,n)$ if $\chi_\infty(z)=z^m\zbar^n$, 
and is called anticyclotomic if $\chi$ is trivial on $\A^\x$.

For each prime $\frakq$ of $\OK$, we let $\chi_\frakq\colon\cK_\frakq\to\C^\x$ denote the $\frakq$-component of $\chi$,
and if $\chi$ has conductor $\frakc$ and $\fraka$ is any fractional ideal prime to $\frakc$,
we write $\chi(\fraka)$ for $\chi(a)$, where $a$ is an idele with $a\wh\cO_\cK\cap K=\fraka$ and
$a_\frakq=1$ for all $\frakq\divides\frakc$.

\begin{defn}The \padic avatar $\wh\chi: \cK^\x\bksl \wh\cK^\x \to\Cp^\x$ of a Hecke character $\chi$ of infinity type $(m,n)$ is defined by
\[\wh\chi(z)=i_p\circ i_\infty^{-1}(\chi(z))z_\frakp^mz_{\frakpbar}^n\]%\text{ for }z\in\wh\cK^\x.\]
for $z\in\wh\cK^\x$.
\end{defn}

Via the reciprocity law map $\rec_\cK$, each \padic Galois character $\rho\colon G_\cK:=\Gal(\Qbar/\cK)\to \Cp^\x$
will be implicitly regarded as a \padic character $\rho\colon\cK^\x\bksl \wh \cK^\x\to\Cp^\x$.
We say that a \padic Galois character $\rho$ is \emph{locally algebraic} if $\rho=\wh\rho_\A$ is the \padic avatar of some Hecke character $\rho_\A$.
%, in which case we say that $\rho_\A$ is the Hecke character associated to $\rho$.
A locally algebraic character $\rho$ is called of infinity type $(m,n)$ if the associated Hecke character $\rho_\A$ is of infinity type $(m,n)$, and the conductor of $\rho$ is defined to be the conductor of $\rho_\A$. Note that if $\rho_\A$ is unramified at $\frakp$ and of infinity type $(m,n)$, then $\rho$ is crystalline at $\frakp$ as $\rho|_{G_{\cK_\frakp}}$ is an unramified twist of the $m$-th power of the \padic cyclotomic character. \subsubsection*{Modular forms}
In the remainder of this article, we fix $f\in S^{\rm new}_{2r}(\Gamma_0(N))$ to be an elliptic newform (\ie normalized eigenform for all Hecke operators) of weight $2r$ and level $N_f\divides N$. Let
\[f(q)=\sum_{n>0}\bfa_n(f) q^n\]
be the $q$-expansion of $f$ at the infinity cusp. Let $F$ be a finite extension of $\Qp$ containing the Hecke field of $f$, \ie the field generated by $\stt{\bfa_n(f)}_n$ over $\Q$. Let $\vp_f$ be the automorphic form attached to $f$, \ie $\vp_f:\A^\x\GL_2(\Q)\bksl\GL_2(\A)\to\C$ is the function satisfying
\[\begin{aligned}\vp_f(g_\infty u)=J(g_\infty,i)^{-2r}f(g_\infty i),\quad\text{ for }g_\infty\in \GL_2(\R)^+, u\in U_1(N_f),
\end{aligned}\]
and let $\pi=\ot'\pi_q$ be the irreducible cuspidal automorphic representation on
$\GL_2(\A)$ generated by $\vp_f$. Note that $\pi$ has trivial central character. Define the automorphic form $\vp^\flat_f$ by \beq\label{E:pdeprived}\vp_f^\flat(g)=\vp_f(g)-\bfa_p(f)p^{-r}\vp_f(g\gamma_p)+p^{-1}\vp_f(g\gamma_p^2),\eeq
where $\gamma_p=\pDII{1}{p}\in\GL_2(\Qp)\hookto \GL_2(\Qhat)$. Define the complex function $\bff^\flat\colon\bbH\xx\GL_2(\Qhat)\to\C$ by
\beq\label{E:adelic}\begin{aligned}\bff^\flat(\tau,g_f)&=\vp_f^\flat((g_\infty,g_f))J(g_\infty,i)^{2r}\abs{\det g_f}_{\A_f}^{r},\\
&\quad(g_\infty\in\GL_2(\R)^+,\,g_\infty i=\tau).\end{aligned}\eeq
Then there is a unique geometric modular form $f^\flat$ of weight $2r$ and level $\Gamma_0(Np^2)$ defined over $\cO_F$ such that:
\begin{itemize}\item $f^\flat(A_x,\eta_x,2\pii dw)=\bff^\flat(x)$ for $x\in \bbH\xx\GL_2(\Qhat)$,
\item with Fourier expansion \[f^\flat({\rm Tate}(q),\eta_{\can},\om_{\can})=\sum_{p\ndivides n}\bfa_n(f) q^n.\]
\end{itemize}
 The \padic avatar $\wh f^\flat\in V_p(N,\cO_F)$ of $f^\flat$ introduced in \subsecref{S:GME} is a \padic modular form of weight $2r$.

\subsubsection*{Explicit Waldspurger formula} We recall a result on the explicit calculation of toric period integrals in \cite{hsieh}. Let $c=\ctame p^n$ with $p\ndivides\ctame$ and $n\geq 0$. Put
\[\Pic\cO_c:=\cK^\x\bksl\wh\cK^\x/\wh\cO_c^\x.\]
If $a\in \wh\cK^{(c\frakp\frakN)\x}$ and $\fraka=a\wh\cK\cap\cO_c$ is the corresponding fractional ideal of $\cO_c$, we shall write $[a]=[\fraka]$ for its class in $\Pic\cO_c$. Let $\chi\colon\cK^\x\bksl \AK^\x/\wh\cO_c^\x\to\C^\x$ be an anticyclotomic Hecke character, and set
\[A(\chi)=\stt{\textrm{primes $q\divides D_K$ such that $\chi_q$ is unramified and $q\divides N_f$.}}.\]
We assume the following Heegner hypothesis:
\beqcd{Heeg'}
\text{ $N_f^-$ is a square-free product of primes ramified in $\cK$,}
\eeqcd
and that $(f,\chi)$ satisfies the condition
\beqcd{ST}
\bfa_q(f)\chi(\frakq)=-1\text{ for every }q\in A(\chi)\quad (q\OK=\frakq^2) .
\eeqcd
\begin{defn}\label{D:E3}Define the $\chi$-toric period by
\begin{align*}P_\chi(f^\flat):=\;&\sum_{[a]\in\Pic\cO_c}\vp^\flat_f((\varsigma_\infty,a\xi_c))\chi(a)\quad(\varsigma_\infty:=\pMX{\Im\CMP}{\Re\CMP}{0}{1})\\
=\;&(c\Im\CMP)^r\cdot\sum_{[a]\in \Pic\cO_{c}}\bff^\flat(\CMP,a\xi_{c})\cdot \chi\Abs_{\AK}^{-r}(a)\quad(\text{by }\eqref{E:adelic}).\end{align*}
\end{defn}
Let $\pi_K$ be the automorphic representation of $\GL_2(\A_K)$ obtained by the base change of $\pi$ to $K$, and let $L(s,\pi_K\ot\chi)$ be the automorphic $L$-function on $\GL_2(\A_K)$ attached to $\pi_K$ twisted by $\chi\circ\det$
\footnote{See \cite[Thm.\,20.6]{jacquet} for the existence of the quadratic base change, and  \cite[\S 11]{LNM114} for the definition of $L$-functions on $\GL(2)$.}. If $\chi$ has infinity type $(r+m,-r-m)$ with $m\geq 0$, define the algebraic central value $L^\alg(\onehalf,\pi_\cK\ot\chi)$ by
\[L^\alg(\onehalf,\pi_\cK\ot\chi)=\frac{\Gamma(2r+m)\Gamma(m+1)}{(4\pi)^{2r+2m+1}(\Im\CMP)^{2r+2m}}\cdot \frac{L(\onehalf,\pi_\cK\ot\chi)}{\Omega_\cK^{4(r+m)}},\]
and the \padic multiplier $e_\frakp(f,\chi)$ by
\[e_\frakp(f,\chi)=\begin{cases}
(1-\bfa_p(f)p^{-r}\chi_{\frakpbar}(p)+\chi_{\frakpbar}(p^2)p^{-1})^2&\text{ if }p\ndivides c,\\
\vep(\onehalf,\chi_\frakp)^{-2} &\text{ if }p\divides c.\end{cases}\]

\begin{prop}\label{T:1}Suppose that
\begin{enumerate}\item[{\rm (a)}] $\chi$ has infinity type $(r,-r)$ and $(c,N^+)=1$, \item[{\rm (b)}] \eqref{Heeg'} and \eqref{ST} hold for $(f,\chi)$.
\item[{\rm (c)}] The conductor of $\chi$ is $c\cO_K$.
\end{enumerate} Then we have \[
\begin{aligned}\left(\frac{P_\chi(f^\flat)}{\Omega_\cK^{2r}}\right)^2&=L^\alg(\onehalf,\pi_\cK\ot\chi)\cdot e_\frakp(f,\chi)\cdot\vep(\onehalf,\chi_\frakp)^{2}\cdot 2^{\#A(\chi)+3}u_\cK^2\sqrt{D_\cK}\cdot c(\Im\CMP)^{2r}\cdot\chi^{-1}(\frakN)\vep(f),\end{aligned}\]
where $u_\cK:=\#(\cO_\cK^\x)/2$ and $\vep(f):=\prod_{q}\vep(\onehalf,\pi_q)$ is the global root number of $f$.
\end{prop}
\begin{proof}We will follow the notations in \cite{hsieh}. Let $W_\chi^\flat=W_{\chi,p}^\flat\prod_{v\not =p} W_{\chi,v}\colon\GL_2(\A)\to\C$ be the Whittaker function defined in \cite[\S 3.6]{hsieh}, and let $\vp_\chi^\flat\colon\GL_2(\Q)\bksl\GL_2(\A)\to\C$ be the associated automorphic form given by
\[\vp_\chi(g)=\sum_{\al\in\Q}W_\chi^\flat(\pDII{\al}{1}g).\]
Let $\varsigma=(\varsigma_\infty,\varsigma^{(\infty)})\in\GL_2(\A)$ with $\varsigma^{(\infty)}$ as in \subsecref{subsec:CMII}, and define the toric period integral
\[P_\chi(\pi(\varsigma)\vp_\chi)=\int_{\cK^\x\A^\x\bksl\AK^\x}\vp_\chi(t\varsigma)\chi(t)dt,\]
where $dt$ is the Tamagawa measure on $\AK^\x/\A^\x$. Under the assumption (b), the explicit Waldspurger formula in \cite[Theorem 3.14]{hsieh} implies that
\beq\label{E:5}P_\chi(\pi(\varsigma)\vp_\chi)^2=\abs{D_\cK}^{-\onehalf}\cdot\frac{\Gamma(2r)}{(4\pi)^{2r+1}}\cdot e_\frakp(f,\chi)\cdot L(\onehalf,\pi_\cK\ot\chi)\cdot C'_\pi(\chi)N(\pi,\chi)^2,\eeq
where $N(\pi,\chi)=\prod_{q\divides c^-}L(1,\tau_{\cK_q/\Q_q})$ and $C'_\pi(\chi)$ is the constant
\begin{align*}C'_\pi(\chi)&=2^{\#(A(\chi))+3}(c^-)^{-1}\cdot\prod_{\frakq\divides \frakC\frakN,\,\frakq\not =\frakp}\vep(\onehalf,\pi_q\ot\chi_{\frakqbar})\\
&=2^{\#(A(\chi))+3}(c^-)^{-1}\cdot\prod_{\frakq\divides\frakN}\vep(\onehalf,\pi_q)\chi_\frakq^{-1}(\frakN)\prod_{\frakq\divides \frakC,\frakq\not=\frakp}\vep(\onehalf,\chi_{\frakq})^{-2}
\end{align*}
In the last equality, we used the formulae
\[\vep(s,\pi_q\ot\chi_{\frakqbar})=\vep(s,\pi_q\ot\chi_{\frakq}^{-1})=\begin{cases}\vep(s,\pi_q)\chi_\frakq^{-1}(\frakN)&\text{ if }q\divides N^+,\\
\vep(1-s,\chi_{\frakq})^{-2}&\text{ if }q\divides c.
\end{cases}\]

On the other hand, under assumption (a) one can verify that
\[\sum_{[u]\in\wh\cO_\cK^\x/\wh\cO_c^\x}\vp^\flat_f(g(\varsigma_\infty,u\xi_c))\chi(u)=\vp_\chi(g\ol{\bfc}\varsigma)\cdot \prod_{\frakq\divides\frakC}\frakg(\chi_\frakq^{-1})\cdot c^-\prod_{q|c^-}(1+1/q)\]
by comparing the Whittaker functions of the automorphic forms $\varphi^\flat_f$ and $\varphi_\chi$ on both sides, where $\bfc=(\bfc_\frakq)_\frakq\in\wh\cK^\x$ is the idele with $\bfc_\frakq=q^{\Ord_q(c)}$ if $\frakq\divides \frakC$ and $\bfc_\frakq=1$ if $q\ndivides \frakC$.
From this equation, we obtain
\begin{align*}
P_\chi(\pi(\varsigma)\vp_\chi)&=\frac{2}{\sqrt{D_\cK}u_\cK}\sum_{[a]\in\Pic\cO_\cK}\chi(a)\vp_\chi(a\varsigma)\\
&=\frac{2N(\pi,\chi)}{c^-\sqrt{D_\cK}u_\cK}\cdot \prod_{\frakq|\frakC}\frakg(\chi_\frakq^{-1})^{-1}\cdot\sum_{[a]\in\Pic\cO_\cK}\chi(a\bfc^{-1})\sum_{[u]\in\wh\cO_\cK^\x/\wh\cO_c^\x}\vp^\flat_f(\varsigma_\infty,a u\xi_c)\chi(u)\\
&=\frac{2N(\pi,\chi)}{c^-\sqrt{D_\cK}u_\cK}\cdot \prod_{\frakq|\frakC}\frakg(\chi_\frakq^{-1})^{-1}\chi_\frakq^{-1}(q^{\Ord_q{c}})\cdot\sum_{[a]\in\Pic\cO_c}
\vp^\flat_f(\varsigma_\infty,au\xi_c)\chi(u).
\end{align*}
We thus find 
\beq\label{E:4}P_\chi(\pi(\varsigma)\vp_\chi)=\frac{2N(\pi,\chi)}{c\sqrt{D_\cK}u_\cK}\cdot \prod_{\frakq|\frakC}\vep(1,\chi_\frakq)^{-1}\cdot P_\chi(f^\flat).\eeq
It is clear that the theorem follows from \eqref{E:5} and \eqref{E:4}.
\end{proof}

\subsubsection*{Analytic construction of the \padic $L$-function}
Let $\rcf{p^\infty}=\cup_n\rcf{p^n}$ be the ring class field of conductor $p^\infty$,
and let $\wtd\Gamma:=\Gal(\rcf{p^\infty}/\cK)$. Then the Galois group $\Gamma^-_K$ %={\rm Gal}(K_\infty/K)$
of the anticyclotomic $\Zp$-extension is the maximal free quotient of $\wtd\Gamma$.
Denote by $\cC(\wtd\Gamma,\cO_{\Cp})$ the space of continuous $\cO_{\Cp}$-valued functions on $\wtd\Gamma$,
and let $\frakX_{p^\infty}\subset \cC(\wtd\Gamma,\cO_{\Cp})$ be the set of locally algebraic \padic characters $\rho:\wtd\Gamma\to\cO_{\Cp}^\x$.

Let $\rec_\frakp:\Qp^\x=\cK_\frakp^\x\to\Gal(\cK^\ab/\cK)\to \wtd\Gamma$ be the local reciprocity law map.
For $\rho\in\frakX_{p^\infty}$, %identified with a Hecke character of $\AK^\x$ via the reciprocity law map,
we define $\rho_\frakp:\Qp^\x\to \Cp^\x$ by
\[\rho_\frakp(\beta)=\rho(\rec_\frakp(\beta)),\] and for $\rho\in \cC(\wtd\Gamma,\cO_{\Cp})$, we define $\rho|[\fraka]:\Zp^\x\to\cO_{\Cp}$ by \[\rho|[\fraka](x)=\rho(\rec_\frakp(x)\sg_\fraka^{-1})=\rho(\rec_\frakp(x)\rec_\cK(a)).\]

For each $a\in\wh\cK^{(\ctame \frakN p)\x}$ with associated fractional ideal $\fraka\subset\cO_{\ctame}$,
let $(A_\fraka,\eta_\fraka)$ be the CM elliptic curve with level structure introduced in \subsecref{subsec:CMpts}.
Let $t_\fraka$ be the canonical Serre--Tate coordinate of $\wh f^\flat$ around $\bfx_\fraka=[(A_\fraka,\eta_\fraka)]\ot_\cW\Fpbar$, and set
\beq\label{E:2}\wh f^\flat_{\fraka}(t_\fraka):=\wh f^\flat(t_\fraka^{\rmN(\fraka)^{-1}\sqrt{-D_K}^{-1}})\in\cW\powerseries{t-1}\quad(\rmN(\fraka)=\abs{a}_{\AK}^{-1}\ctame^{-1}).\eeq

\begin{defn}[Analytic anticyclotomic \padic $L$-functions]\label{D:padicL}Let $\brch$ be an anticyclotomic Hecke character of infinity type $(r,-r)$,
 and let $\ctame\OK$ be the prime-to-$p$ part of the conductor of $\brch$. Define the \padic measure $\sL_{\frakp,\brch}(f)$ on $\wtd\Gamma$ by
\[\sL_{\frakp,\brch}(f)(\rho)=\sum_{[\fraka]\in \Pic\cO_{\ctame}}\brch(\fraka)\rmN(\fraka)^{-r}\cdot \left(\wh f^\flat_{\fraka}\ot\brch_\frakp\rho|[\fraka]\right)(A_\fraka,\eta_\fraka).\]
We shall also view $\sL_{\frakp,\brch}(f)$ as an element in the semi-local ring $\cW\powerseries{\wtd\Gamma}$. \end{defn}
The $p$-adic measure $\sL_{\frakp,\brch}(f)$ satisfies the following interpolation formula at characters of infinity type $(m,-m)$ with $m\geq 0$. In what follows, we assume \eqref{Heeg'}, \eqref{ST} for $(f,\psi)$ and that $(\ctame,pN^+)=1$.
\begin{prop}\label{P:interpolationformula}
 If $\wh\phi\in\frakX_{p^\infty}$ is the \padic avatar of a Hecke character $\phi$ of infinity type $(m,-m)$ with $m\geq 0$ and $p$-power conductor, then
\[\begin{aligned}\left(\frac{\sL_{\frakp,\brch}(f)(\wh\phi)}{\Omega_p^{2r+2m}}\right)^2&=L^\alg(\onehalf,\pi_\cK\ot\brch\phi)\cdot e_\frakp(f,\brch\phi)\cdot \phi(\frakN^{-1})\cdot 2^{\#A(\brch)+3}\ctame \vep(f)\cdot u_\cK^2\sqrt{D_\cK}.\end{aligned}\]
\end{prop}
\begin{proof}
\def\tchi{\chi}
Suppose that $m=0$. Then $\wh\phi=\phi$ is a finite order character, and $\tchi:=\brch\phi$ is an anticyclotomic Hecke character of infinity type $(r,-r)$. Let $c\cO_\cK$ be the conductor of $\tchi$ (so $c=\ctame p^n$). Suppose that $n>0$.
By \defref{D:padicL} and \propref{P:1}, we have
\beq\label{E:formulaL}\begin{aligned}
\sL_{\frakp,\brch}(f)(\phi)&=\ctame^r\sum_{[a]\in \Pic\cO_{\ctame}}\left(\wh f^\flat_\fraka\ot\tchi_\frakp\right)(x_\fraka)\tchi\Abs_{\A_\cK}^r(a)\\
&=p^{-n}\frakg(\tchi_\frakp)\ctame^r\sum_{[a]\in \Pic\cO_{\ctame}}\tchi\Abs_{\A_\cK}^{-r}(\ol{a}^{-1})\sum_{u\in (\Zp/p^n\Zp)^\x}\wh f^\flat(x_\fraka*\bfn(up^{-n}))\tchi_\frakp(u^{-1}).
\end{aligned}\eeq
For $z\in\Qp^\x$, we use $z_{\frakpbar}$ (resp. $z_\frakp$) to denote the finite idele in $\wh\cK^\x$ with $z$ at $\frakpbar$ (resp. $\frakp$) and $1$ at all the other places. Since $f^\flat$ is of weight $2r$ and level $\Gamma_0(Np^2)$, a direct calculation shows that
\begin{align*}\wh f^\flat(x_\fraka*\bfn(up^{-n}))&=\bff^\flat(\CMP,\ol{a}^{-1}\xi_{\ctame}\bfn(up^{-n}))\cdot \frac{\Omega_p^{2r}}{\Omega_\cK^{2r}}\\
&=\bff^\flat(\CMP,\ol{a}^{-1}u_{\frakpbar}p_{\frakpbar}^{-n}\xi_{\ctame p^n}))\cdot \frac{\Omega_p^{2r}}{\Omega_\cK^{2r}}\quad(u\in\Zp^\x),\end{align*}
where $(\Omega_\cK,\Omega_p)$ are the periods defined in \subsecref{SS:CMperiods}. Note that we used \eqref{E:levelstr} in the last equation. We thus find
\begin{align*}
\frac{\sL_{\frakp,\brch}(f)(\phi)}{\Omega_p^{2r}}&=\frac{p^{-n}\frakg(\tchi_\frakp)\ctame^r}{\Omega_\cK^{2r}}\cdot\sum_{[a]\in \Pic\cO_{\ctame}}\sum_{u\in \wh\cO_{\ctame}^\x/\wh\cO_{c}^\x}\tchi\Abs_{\AK}^{-r}(\ol{a}^{-1}u_{\frakpbar})\bff^\flat(\CMP,\ol{a}^{-1}u_{\frakpbar}p_{\frakpbar}^{-n}\cdot \xi_{c})\\
&=\frac{p^{-n}\frakg(\tchi_\frakp)\ctame^r}{\Omega_\cK^{2r}}\cdot\sum_{[a]\in \Pic\cO_{c}}\tchi\Abs_{\A_\cK}^{-r}(a)\bff^\flat(\CMP,ap_{\frakpbar}^{-n}\cdot \xi_{c})\\
&=\frac{\vep(1,\tchi_{\frakpbar})\tchi_{\frakpbar}(-1)c^r}{\Omega_\cK^{2r}}\cdot\sum_{[a]\in \Pic\cO_{c}}\tchi\Abs_{\A_\cK}^{-r}(a)\bff^\flat(\CMP,a\xi_{c}) \quad(\text{by  \eqref{E:formulaepsilon}}).
\end{align*}
Therefore, according to \defref{D:E3} we obtain
\[\frac{\sL_{\frakp,\brch}(f)(\phi)}{\Omega_p^{2r}}=\frac{\vep(\onehalf,\tchi_{\frakp})^{-1}}{\Omega_\cK^{2r}}\cdot P_{\tchi}(f^\flat)\cdot p^{-\frac{n}{2}}(\Im\CMP)^{-r}.\]
The proposition for the case $m=0$ and $n>0$ now follows from \thmref{T:1}.
If $n=0$, \ie $\chi_\frakp=\bfone$ is the trivial character on $\Zp^\x$, then one can use \eqref{E:pdeprived} and the fact that $\vp_f$ is a Hecke eigenform to show that $\wh f^\flat_\fraka\ot\tchi_\frakp(x_\fraka)=\wh f^\flat_\fraka(x_\fraka)$, so \eqref{E:formulaL} is still valid, and as above the proposition also follows in this case.

For general $m\geq 0$, comparing the interpolation formulas for $\sL_{\frakp,\brch}(f)$ and for the \padic $L$-function $\sL_{\frakp}(\pi,\psi)$ constructed in \cite[Thm.\,A]{hsieh} at $p$-ramified finite order characters $(m=0)$, we find that $\sL_{\frakp,\brch}(f)=u\cdot \sL_{\frakp}(\pi,\brch)$ with $u=2^{\#A(\brch)+3}\ctame \vep(f)\sqrt{D_\cK}$, and hence the general interpolation formulae of $\sL_{\frakp,\psi}(f)$ can be deduced from those of  $\sL_{\frakp}(\pi,\psi)$ in \emph{loc.cit.}. We omit the details.
\end{proof}
We now prove the nonvanishing of the \padic $L$-function $\sL_{\frakp,\brch}(f)$.
\begin{thm}\label{T:nonvanishing}Suppose $(N_f,D_\cK)=1$. For all but finitely many $\phi\in \frakX_{p^\infty}$, we have $\sL_{\frakp,\brch}(f)(\phi)\not =0$.
\end{thm}
\begin{proof}Since $f$ has conductor prime to $D_\cK$, $f$ can not be a CM form arising from $\cK$, and hence the $\ell$-adic representation $\rho_{f,\ell}$ is irreducible when restricted to $G_K$ for every prime $\ell$. Therefore, it is well-known that there exist infinitely many primes $\ell$ such that:
\begin{itemize}\item $\ell$ is prime to $pND_\cK\prod_{q|c^-}(1+q)$,
\item the residual Galois representation $\bar\rho_{f,\ell}|_{G_K)}$ is absolutely irreducible.
\end{itemize}
By \cite[Theorem\,C]{hsieh}, the central $L$-values $\stt{L^{\alg}(1/2,\pi_\cK\ot\brch\phi)}$ are non-zero modulo $\ell$ for all but finitely many finite order characters $\phi\in\frakX_{p^\infty}$. (Note that the roles of $p$ and $\ell$ have been switched here.) In particular, this implies that $\sL_{\frakp,\brch}(f)$ does not vanish identically, and hence the theorem follows from \padic Weierstrass preparation theorem.
\end{proof}

% !TEX root = HeegnerCycles.tex
%ML, May 20

\def\frakpbar{{\ol{\frakp}}}
\def\loc{{\rm loc}}
\def\rmD{\mathrm D}
\def\dR{{\rm dR}}
\def\cris{{\rm cris}}
\def\DdR{\bD_\dR}
\def\rmt{{\rm t}}
\def\Iw{{\rm Iw}}
\def\pcyc{\vep_{\rm cyc}}
\def\rmNS{{\rm NS}}
\def\OKo{\cO}
\def\Initial{A}
\def\Hclass{z}
\section{Generalized Heegner cycles}

\subsection{Definitions}\label{sec:def-Heegner}
We continue to let $f\in S_{2r}^{\rm new}(\Gamma_0(N))$ be a newform of weight $2r$ and level $N$. We assume the (strong) Heegner condition
\beqcd{Heeg} \text{$N$ is a product of primes split in $\cK$}.\eeqcd  Thus \eqref{Heeg'} and \eqref{ST} will automatically hold. Let $K=\Q(\sqrt{-D_K})$ be the imaginary quadratic field of discriminant $-D_K$.
If $r>1$, we further assume that
\beqcd{can}\text{either $D_K>3$ is odd, or $8\mid D_K$}.\eeqcd
This assumption ensures the existence of canonical elliptic curves in the sense of Gross
(see \cite[Thm.~0.1]{yang2004}). We shall fix a canonical elliptic curve $A$ with CM by $\cO_K$, which is characterized by the following properties:
\begin{itemize}
\item $\Initial$ is equipped with CM by $[\cdot]\colon\cO_K\iso\End \Initial$.
\item There is a complex uniformization $\xi\colon\C/\cO_K\iso \Initial(\C)$.
\item $\Initial$ is a $\Q$-curve defined over $H_K^+$, where $H_K^+=\Q(j(\cO_K))$ is the real subfield of the Hilbert class field $H_K$ of $K$.
\item The conductor of $\Initial$ is only divisible by prime factors of $D_K$.
\end{itemize}

For each positive integer $c$, let $\sC_c:=\xi(c^{-1}\cO_{c}/\cO_K)\subset \Initial$ be a cyclic subgroup of order $c$. The elliptic curve $A/\sC_c$ is defined over the real subfield $\Q(j(\cO_c))$ of the ring class field
$\rcf{c}$ of conductor $c$. Let $\vphi_c\colon\Initial_{/\rcf{c}}\to A_c{}_{/\rcf{c}}$ be the isogeny given by the natural quotient map. Then $A/\sC_c$ is equipped with the complex uniformization $A_c\iso \C/\cO_{c}$ such that $\varphi_c\colon\C/\cO_K\to\C/\cO_{c}$ is given by $z\mapsto cz$. Thus we see that the elliptic curve $A_c$ introduced in\subsecref{subsec:CMpts} descends to the elliptic curve $\Initial/\sC_c$, still denoted by $A_c$ in the sequel.

For any ideal $\fraka$ of $\cO_c$, in this section we always assume that $\fraka\cO_K$ is prime to $cD_K\frakp\frakN$. Let $\fraka$ be an ideal of $\cO_{c}$ and recall that $\sg_\fraka\in\Gal(K^\fraka/K)$ is the image of $\fraka$ under the Artin map, where $K^\fraka$ is the maximal abelian $\fraka$-ramified extension of $K$.
Then, by the main theorem of complex multiplication (\cf \cite[Prop.~1.5, p.42]{de_shalit}), we have $A_\fraka=A_c^{\sg_\fraka}$ and
the isogeny $\lam_\fraka\colon A_c\to A_\fraka$ in \subsecref{subsec:CMpts} is actually defined over $\rcf{c}$ and characterized by the rule
\beq\label{E:Amult.norm}\lam_\fraka(x)=\sg_\fraka(x)\quad\text{ for all }x\in \Initial[m],\,(m,\rmN(\fraka))=1.\eeq

Define the isogeny
\[\varphi_\fraka:=\lam_\fraka\circ\varphi_c\colon\Initial{}_{/\rcf{c}}\longto A_\fraka{}_{/\rcf{c}},\]
and let $\Gamma_\fraka$ be the graph
\[\Gamma_\fraka=\stt{\varphi_\fraka(z),z)\mid z\in \Initial}\subset A_\fraka\x \Initial.\]
Let $x_\fraka=[(A_\fraka,\eta_\fraka)]\in Y_1(N)(K_c)$ be the CM point associated to $\fraka$ as in the last paragraph of \subsecref{subsec:CMpts},
and let $\sA$ be the universal elliptic curve over $Y_1(N)$. Then $x_\fraka$ determines an embedding $i_{x_\fraka}:A_\fraka\to \sA$, and we define
\[\cY_\fraka=(i_{x_\fraka}\x{\rm id})(\Gamma_\fraka)=\stt{(i_{x_\fraka}(\varphi_\fraka(z)),z)\mid z\in \Initial}\subset \sA\x \Initial.\]
Denote by $W_{2r-2}$ the Kuga--Sato variety of dimension $2r-1$ (\cf\cite[p.1056]{BDP}). Following \cite[p.1063]{BDP}, define the cycle $\Upsilon_\fraka$ in the generalized Kuga--Sato variety $X_{2r-2}:=W_{2r-2}\times \Initial^{2r-2}$
by \[\Upsilon_\fraka=\cY_\fraka^{2r-2}\subset (i_{x_\fraka}(A_\fraka)\times\Initial)^{2r-2}\hookto X_{2r-2}. \]
Let $\ep_X=\ep_W\times\ep_A$, with $\ep_{W}\in\Z[\Aut (W_{2r-2})]$ and $\ep_A\in \Z[\Aut(\Initial^{2r-2})]$ the idempotents defined in \cite[(2.1.2),\,(1.4.4)]{BDP}. The following definition is given in \cite[p.1063]{BDP}.
\begin{defn}\label{D:Heegnercycles}
The \emph{generalized Heegner cycle} $\Delta_{\varphi_\fraka}$ associated to the isogeny $\varphi_\fraka$ is
\[\Delta_{\varphi_\fraka}:=\ep_X[\Upsilon_\fraka]\in {\rm CH}^{2r-1}(X_{2r-2}/\rcf{c})_{0,\Q}.\]
\end{defn}
\subsection{Generalized Heegner classes (I)}\label{SS:GHCI}
Let $p$ be a prime with $p\ndivides 2(2r-1)!N\varphi(N)$. Let $F$ be a finite extension of $\Qp$ containing the Hecke field of $f$.
Let $V_f$ be the two-dimensional \padic representation of $G_\Q$ over $F$ attached to the newform $f$ by Deligne,
and denote by $V_f(r)$ the Tate twist $V_f\ot\e_{\rm cyc}^r$, where $\e_{\rm cyc}$ is the \padic cyclotomic character. Following \cite[\S 3.1]{BDP}, we consider the \padic Abel--Jacobi map
\begin{align*}\Phi_{\et,f}:{\rm CH}^{2r-1}(X_{2r-2}/\rcf{c})_0&\longto H^1(\rcf{c},\ep_X
H_{\et}^{4r-3}(X_{2r-2}{}_{/\Qbar},\Zp)(2r-1))\,\\
&\longto H^1(K,\ep_WH^{2r-1}(W_{2r-2}{}_{/\Qbar},\Zp)(r))\otimes\Sym^{2r-2}H^1_{\et}(A_{/\Qbar},\Zp)(r-1))\\
&\longto H^1(\rcf{c},T\ot S^{r-1}(A)),
\end{align*}
where $T$ is the Galois stable $\cO_F$-lattice in $V_f(r)$ in \cite[\S 3]{nekovar-invmath}, and $S^{r-1}(A)$ is the $G_{H_K}$-module
\[S^{r-1}(A):=\Sym^{2r-2}T_p(A)(1-r)\]
with $T_p(A)$ the \padic Tate module of $A$. For every ideal $\fraka$ of $\cO_c$,
define the \emph{generalized Heegner class} $\Hclass_{f,\fraka}$ associated to $\fraka$ by
\beq\label{E:defnGHC}
\Hclass_{f,\fraka}:=\Phi_{\et,f}(\Delta_{\varphi_\fraka})\in H^1(\rcf{c},T\ot S^{r_1}(A)).
\eeq
In the following, we shall simply write $\Hclass_{f,c}$ for $\Hclass_{f,\cO_c}$.

\subsection{Norm relations}
\begin{lm}\label{L:NS.norm}If $D\subset (A_\fraka \times A)^{2r-2}$ is a cycle of codimension $r-1$ such that $D$ is zero in the N\'eron--Severi group of ${\rm NS}(A_\fraka\times A)^{2r-2}$, then the \padic Abel--Jacobi image of $\ep_X (i_{x_\fraka})_*(D)$ in $H^1(\rcf{c},T\ot S^{r-1}(\Initial))$ is also trivial. \end{lm}
\begin{proof}
This follows from the fact that the Abel--Jacobi image of $\ep_X (i_{x_\fraka})_*(D)$ lies in the image of the map \[ H^1(\rcf{c},\ep_XH_{\et}^{4r-5}(A_\fraka{}_{/\Qbar}^{2r-2}\times\Initial_{/\Qbar}^{2r-2},\Zp))\stackrel{i_{x_\fraka}}\longto H^1(\rcf{c},\ep_X H_{\et}^{4r-3}(X_{2r-2}{}_{/\Qbar},\Zp))\]
and $\ep_XH_{\et}^{4r-5}(\bar A_\fraka{}_{/\Qbar}^{2r-2}\times\bar\Initial_{/\Qbar}^{2r-2},\Zp)=0$.
\end{proof}

We refer to $\S\ref{SS:twisted}$ for the definition of the character $\wtd\kappa_A$ appearing in the next result.

\begin{lm}\label{L:Galois}Suppose $\fraka\cO_K$ is trivial in $\Pic\cO_K$, and let $\al:=\wtd\kappa_\Initial(\fraka)\in K^\x$.
Then for every  ideal $\frakb$ of $\cO_{c}$ prime to $cND$, we have
\[
({\rm id}\x[\al])^*\Delta_{\varphi_\frakb}^{\sg_\fraka}=\Delta_{\varphi_{\fraka\frakb}}.
\]
\end{lm}
\begin{proof}
Let $\sg=\sg_\fraka\in\Gal(K^{ab}/H_K^+K)$. By definition, $\Initial^\sg=\Initial$ and $A_\fraka^{\sg}=A_{\fraka\frakb}$.
Note that for any $t\in \Initial[m]$ with $(m,\rmN(\fraka\frakb))=1$, we have $\sigma(t)=\lam_{\fraka\cO}(t)=[\al](t)$ and
\[\varphi_c^\sg\circ[\al](t)=\varphi_n^\sg(\sg(t))=\sg(\varphi_c(t))=\lam_\fraka(\varphi_n(t)).\]
This implies that $\varphi_c^\sg\circ[\al]=
\lam_\fraka\circ\varphi_c$. Therefore,  \[[\al]\circ\varphi_\frakb^\sg=\varphi_\frakb^\sg\circ[\al]=\lam_\frakb^\sg\circ\varphi_c^\sg\circ[\al]=\lam_\frakb^\sg\circ\lam_\fraka\circ\varphi_c=\varphi_{\fraka\frakb},\]
and
\[({\rm id}\x[\al])^*\Gamma_{\frakb}^{\sg_\fraka}=([\al]\x{\rm id})_*\Gamma_\frakb^{\sg_\fraka}=\Gamma_{[\al]\circ\varphi_\frakb^\sg}=\Gamma_{\varphi_{\fraka\frakb}}=\Gamma_{\fraka\frakb}.\]
The lemma thus follows immediately from $x_\frakb^{\sg_\fraka}=x_{\fraka\frakb}$.
\end{proof}

Let $G_n:=\Gal(\rcf{cp^n}/\rcf{cp^{n-1}})$, which is identified with $\ker\stt{\Pic\cO_{cp^n}\to \Pic\cO_{cp^{n-1}}}$ via the Artin isomorphism. The usual Hecke correspondence $T_q$ associated with a prime $q\ndivides N$ on the Kuga--Sato variety $W_{2r-2}$ (see \cite[$\S$4]{scholl100}) induces the Hecke correspondence $T_q\times {\rm id}$ on the generalized Kuga--Sato variety $X_{2r-2}=W_{2r-2}\times A^{2r-2}$. In what follows, we shall still write $T_q$ for $T_q\times{\rm id}$ if no confusion arises.
\begin{prop}\label{prop:norm}
Assume that $p\nmid c$. If $p=\frakp\frakpbar$ is split in $K$, then for all $n>1$ we have
\[
T_p \Hclass_{f,cp^{n-1}}=
p^{2r-2}\cdot \Hclass_{f,cp^{n-2}}+ {\rm cor}_{\rcf{cp^n}/\rcf{cp^{n-1}}}(\Hclass_{f,cp^n}),% &\textrm{if $n\geq 2$,}
%\\p^{r_1}\cdot \Hclass_{f,c}^{\sigma_\frakp}+p^{r_1}\cdot \Hclass_{f,c}^{\sigma_{\frakpbar}}+u_c\cdot{\rm cor}_{H_{cp}/H_c}(\Hclass_{f,cp})& \textrm{if $n=1$,}
\]
where $u_c=\#(\cO_c^\times)$, and $\sigma_{\frakp}$, $\sigma_{\frakpbar}\in{\rm Gal}(K_c/K)$ are the Frobenius elements of $\frakp$
and $\frakpbar$. Moreover, if $\ell\ndivides c$ is inert in $K$, then
\[T_\ell \Hclass_{f,c}={\rm cor}_{\rcf{c\ell}/\rcf{c}}(\Hclass_{f,c\ell}).\]
\end{prop}
\begin{proof}
Let $\cL\subset \cO_{cp^{n-1}}$ be a sublattice of $\cO_{cp^{n-1}}$ with index $p$, and let $A_\cL=\C/\cL$.
Let $\psi_\cL\colon A_\cL\to A_{cp^{n-1}}$ be the isogeny induced by $\cL\hookto \cO_{cp^{n-1}}$. We have two cases:

\emph{Case}(i): $\cL$ is an $\cO_{cp^n}$-ideal and the class $[\cL]$ is trivial in $\Pic\cO_{cp^{n-1}}$, so we can write $\cL=\al\fraka^{-1}$ for some integral $\cO_{cp^{n}}$-ideal $\fraka$ with $\al=\wtd\kappa_\Initial(\fraka)$. Then we have $A_\cL\iso A_\fraka$ and \[\psi_\cL\circ\varphi_\fraka=[p\al]\circ \varphi_{cp^{n-1}}.\]
Denote by $T_x$ the translation map by a torsion point $x\in A_\fraka\x \Initial$. Then we have \begin{align*}\bigsqcup_{z\in\ker\psi_\cL}T_{(z,0)}^*\Gamma_{\fraka}&=\stt{(x,y)\mid \psi_\cL(x)=\psi_\cL(\varphi_{\fraka}(y))}\\
&=\stt{(x,y)\mid \psi_\cL(x)=\varphi_{cp^{n-1}}(p\al y)}\\
&=({\rm id}\x [p\alpha])^*\psi_\cL^*\Gamma_{\cO_{cp^{n-1}}}.
\end{align*}
This implies that $p\cdot\Gamma_\fraka$ and $p\cdot({\rm id}\x [\al])^*\psi_\cL^*\Gamma_{\cO_{cp^{n-1}}}$ are equal in the N\'eron--Severi
group ${\rm NS}(A_\fraka\x \Initial)$, and hence by \lmref{L:NS.norm} we have
\[\Hclass_{f,\fraka}=({\rm id}\x [\al])^*\psi_\cL^*\Hclass_{f,cp^{n-1}}.\]
Using \lmref{L:Galois} and the projection formula $({\rm id}\x [\al])_*({\rm id}\x [\al])^*=\rmN(\fraka\cO_K)^{2r-2}$, we conclude that
\beq\label{E:1case}\psi_\cL^*\Hclass_{f,cp^{n-1}}=\Hclass_{f,cp^n}^{\sg_\fraka}.\eeq

\emph{Case}(ii): $\cL=p\cO_{cp^{n-2}}$ and $p$ is split in $K$.
Then $A_\cL\iso A_{cp^{n-2}}$, and
\[\psi_\cL\circ\varphi_{cp^{n-2}}=\varphi_{cp^{n-1}}.\]
Note that
\[\psi_\cL^*\Gamma_{\cO_{cp^{n-1}}}=\bigsqcup_{z\in{\rm ker} \psi_\cL}T_{(z,0)}^*(\Gamma_{\cO_{cp^{n-2}}}),\]
so $\psi_\cL^*\Gamma_{\cO_{cp^{n-1}}}$ and $p\cdot \Gamma_{\cO_{cp^{n-2}}}$ are equal in the N\'eron--Severi
group ${\rm NS}(A_{cp^{n-2}}\x \Initial)$. By \lmref{L:NS.norm}, we have
\beq\label{E:2case}\psi_\cL^*\Hclass_{f,cp^n}=p^{2r-2}\cdot \Hclass_{f,cp^{n-2}}.\eeq
Choose a set $\Xi$ of representatives of fractional $\cO_{cp^n}$-ideals of $\ker\stt{\Pic\cO_{cp^n}\to \Pic\cO_{cp^{n-1}}}$, and let \[\Xi^*:=\stt{\al^{-1}\fraka\subset \cO_{cp^n}\mid \fraka\in\Xi,\,\al=\widetilde{\kappa}_\Initial(\fraka)}.\]
If $p$ is split, then \[\stt{\cL\subset \cO_{cp^{n-1}}\mid [\cO_{cp^{n-1}}:\cL]=p}=\Xi^*\disjoint \stt{p\cO_{cp^{n-2}}},\]
and thus by \eqref{E:1case} and \eqref{E:2case} we see that
 \begin{align*}T_p \Hclass_{f,cp^{n-1}}=&\sum_{\substack{
L\subset \cO_{cp^{n-1}},\\ [\cO_{cp^{n-1}}:L]=p}}\psi^*_L \Hclass_{f,cp^{n-1}}=p^{2r-2}\cdot \Hclass_{f,cp^{n-2}}+\sum_{\sg\in G_n}\Hclass_{f,cp^n}^\sg.
\end{align*}
If $\ell$ is inert and $n=1$, then \[\stt{\cL\subset \cO_{c}\mid [\cO_{c}:\cL]=\ell}=\stt{\al^{-1}\fraka\subset \cO_{c}\mid \fraka \text{ ideal of }\cO_{c\ell},\,\al=\wtd\kappa_A(\fraka)},\] and hence
\[T_\ell \Hclass_{f,c}=\sum_{\sg\in \Gal(\rcf{c\ell}/\rcf{c})}\Hclass_{f,c\ell}^\sg.\]
This completes the proof.
\end{proof}

\subsection{Generalized Heegner classes (II)}\label{SS:twisted}
Let $\ctame$ be a positive integer with $(\ctame,pN)=1$, and let
$\chi\colon\Gal(\rcf{\ctame p^\infty}/K)\to\cO_F^\x$ be a locally algebraic anticyclotomic character of infinity type $(j,-j)$
with $-r< j<r$ and conductor $\ctame p^s\cO_K$. The aim of this section is to construct classes $\Hclass_{f,\chi,c}\in H^1(K_c,T\ot\chi)$ by taking the corestriction of $\Hclass_{f,c}$ for every $c$ divisible by $\ctame p^s$. However, note that the CM elliptic curve $A$ is only defined over the Hilbert class field $H_K$, so the group $\Gal(\rcf{c}/K)$ does not act on $\Hclass_{f,c}$ in general. In order to get a natural Galois action, we consider
\[ B_{/K}:={\rm Res}_{H_K/K}\Initial,\]
the abelian variety obtained by restriction of scalars. As is well-known,
$B$ is a CM abelian variety over $K$ and $M:=\Q\ot_\Z\End_KB$ is a product of CM fields over $K$ with $\dim B=[M:K]=[H_K:K]$
(see \cite[Prop.\,(1.2)]{rubin1981}).

Let $I(D_K)$ be the group of prime-to-$D_K$ fractional ideals of $K$, and let
\[
\wtd{\kappa}_\Initial: I(D_K)\longto M^\x
\]
be the CM character associated to $B$ with the following properties (\cf \cite[Lemma, p.457]{rubin1981}):
\begin{itemize}
\item $\wtd\kappa_\Initial(\al \fraka)=\pm \al\cdot \wtd\kappa_\Initial(\fraka)$ for all $\al\in K^\x$ with $\al$ prime to $D_K$ and $\fraka\in I(D_K)$.
\item For all $\fraka\in I(D_K)$ and $t\in B[m]$ with $(m,\rmN(\fraka))=1$, we have
\[
\wtd\kappa_\Initial(\fraka)(t)=\sg_\fraka(t);
\]
and if $\sg_\fraka$ is trivial on $H_K$ (or equivalently, if $\fraka$ is the norm of an ideal of $H_K$),
then $\wtd\kappa_\Initial(\fraka)\in K^\x$ and $\sg_\fraka(t)=[\wtd\kappa_\Initial(\fraka)]t$ for all $t\in\Initial[m]$.
\end{itemize}
%and We now define the cohomology classes $\Hclass_{f,c}^\chi\in H^1(\rcf{c},T\ot\chi)$ for the fixed Galois stable lattice $T\subset V_f(r)$.
 Define the $G_K$-module \[S^{r-1}(B):=\Sym^{2r-2}T_p(B)(1-r)\ot_{\Zp}\cO_F\iso\Ind_{G_{H_K}}^{G_K}S^{r-1}(A)\ot_{\Zp}\cO_F.\]
Enlarge $F$ so that $M\subset F$, and let $\kappa_A:G_K\to\cO_F^\x$ be the \padic avatar of $\wtd\kappa_A$. By the above properties of the CM character $\wtd\kappa_A$, we have
\[T_p(B)\ot_{\Qp}F=\bigoplus_{\rho\in\Hom(M,F)}{}^\rho \kappa_A, %\quad ({}^\rho\kappa_A(\sigma):=\rho(\kappa_A(\sigma))).
\]
where ${}^\rho\kappa_A(\sigma):=\rho(\kappa_A(\sigma))$. If follows that if $\kappa_A^\tau$ is the \padic character of $G_K$ defined by $\kappa^\tau(\sigma):=\kappa(\tau\sigma\tau^{-1})$, where $\tau$ is the complex conjugation, then $(\kappa_A^\tau/\kappa_A)^j$ has infinity type $(j,-j)$ and is a direct summand of $S^{r-1}(B)$ as $G_K$-modules. Therefore, there exists a finite order anticyclotomic character $\chi_t$ such that $\chi$ is realized as a direct summand of $S^{r-1}(B)\ot\chi_t$ as $G_K$-modules, and let
\beq\label{E:projector}
e_\chi:S^{r-1}(B)\ot\chi_t\longto \chi
\eeq
be the corresponding $G_K$-equivariant projection. Note that $\chi_t$ is unique up to multiplication by a character of $\Gal(H_K/K)$,
and that it has the same conductor as $\chi$. %(Note that $c_0\cO_K$ is equal to the conductor of $\chi_t$ because $\chi$ is anticyclotomic.)
In view of the decomposition
\[
T_p(B)=\bigoplus_{\rho\in\Gal(H_K/K)} T_p(\Initial^\rho)\iso\Ind_{G_{H_K}}^{G_K}(T_p(A)),
\]
we shall regard the classes $\Hclass_{f,\fraka}$ of $(\ref{E:defnGHC})$ as elements
$\Hclass_{f,\fraka}\in H^1(\rcf{c},T\ot S^{r-1}(B))$ via the natural inclusion $T_p(\Initial)\to T_p(B)$
for $c$ divisible by $\ctame p^s$ and $\fraka$ an $\cO_c$-ideal.

\begin{prop}\label{P:Galois.norm}
Let $\fraka$ be an $\cO_{c}$-ideal with $(\fraka,cND_K)=1$. Then
\[\chi_t(\sg_\fraka)\cdot ({\rm id}\ot e_\chi) \Hclass_{f,c}^{\sg_\fraka}=\chi\varepsilon_{\rm cyc}^{1-r}(\sg_\fraka)\cdot ({\rm id}\ot e_\chi) \Hclass_{f,\fraka}.
\]
\end{prop}
\begin{proof}
We write $\sg_\fraka=\rec_K(a^{-1})$, where $a\in \wh K^{(c)\x}$ is such that $a\wh\cO_{c}\cap K=\fraka$,
and let $\sg=\sg_\fraka\in\Gal(\rcf{c}/K)$. One easily verifies that
\begin{align*}({\rm id}\x\lam_{\fraka\cO_K})_*\Delta_{\varphi_\fraka}
&=\stt{(\lam_\fraka(\varphi_{c}(z)),\lam_{\fraka\cO_K}(z)\mid z\in \Initial}\\
&=\stt{(\varphi^{\sg_\fraka}_{c}(\lam_{\fraka\cO_K}(z)),\lam_{\fraka\cO_K}(z)\mid z\in \Initial}\\
&=\Delta_{c}^{\sg_\fraka}.
\end{align*}
We have the following fact:
\[
\wtd\kappa_{\Initial}(\fraka\cO_K)=(\lam_{\fraka\cO_K}^\rho)_{\rho}
\in\bigoplus_{\rho\in\Gal(H_K/K)}\Hom(\Initial^\rho,\Initial^{\rho\sg_\fraka})\subset \End(B).
\]
This can be checked, for instance, by comparing the action of both sides on the \padic Tate module of $B$
(see Eq.~\eqref{E:Amult.norm}). By the above fact, we find that
\[\Hclass_{f,c}^\sg=\Phi_{\et,f}(\Delta_{c}^{\sg_\fraka})=({\rm id}\x\lam_{\fraka\cO_K})_*\Hclass_{f,\fraka}=[\wtd \kappa_\Initial(\fraka\cO_K)]_*(\Hclass_{f,\fraka}),\]
where $[\wtd\kappa_\Initial(\fraka\cO_K)]_*$ denotes the push-forward of $\wtd\kappa_\Initial(\fraka\cO_K)$ acting on $\Sym^{2r-2}H^1_{\et}(B_{/\Qbar},\Zp)$. Note that $[\wtd\kappa_{\Initial}(\fraka\cO_K)]_*$ induces the Galois action $\sg_\fraka$ on $H^1_{\et}(B_{/\Qbar},\Zp)(\iso T_p(\Pic^0_{/\Qbar}))$  and that \[e_\chi(\sg\ot\e_{\rm cyc}^{r-1}(\sg)\ot \chi_t(\sg)t)=\chi(\sg)e_\chi t\] for every $t\in S^{r_1}(B)\ot\chi_t=\Sym^{2r-2}H_{\et}^1(B_{/\Qbar},\Zp)(r-1)\ot\chi_t$ by the definition of $e_\chi$. We thus find that
\begin{align*}({\rm id}\ot e_\chi)\Hclass_{f,c}^\sg&=e_\chi([\wtd\kappa_\Initial(\fraka\cO_K)]_*\Hclass_{f,\fraka})\\
&=\chi_t^{-1}\e_{\rm cyc}^{1-r}(\sg_\fraka)\cdot e_\chi(\sg_\fraka\ot\e_{\rm cyc}^{r-1}(\sg_\fraka)\ot\chi_t(\sg_\fraka)\cdot \Hclass_{f,\fraka})\\
&=\chi_t^{-1}\varepsilon_{\rm cyc}^{1-r}(\sg_\fraka)\chi(\sg_\fraka)\cdot ({\rm id}\ot e_\chi)\Hclass_{f,\fraka},\end{align*}
and the proposition follows.
\end{proof}

For each integer $c$ divisible by the conductor of $\chi$, put $\Hclass_{f,c}\ot \chi_t:=\Hclass_{f,c}\in H^1(\rcf{c},T\ot S^{r-1}(B)\ot\chi_t)$,
%\beq\label{E:1.norm}\Hclass_{f,c}\ot \chi_t:=\Hclass_{f,c}\in H^1(\rcf{c},T\ot S^{r_1}(B)\ot\chi_t),\eeq
and let $\Hclass_{f,\chi,c}$ be the $\chi$-component of the class $\Hclass_{f,c}$ defined by
\beq \label{E:defn.norm}\begin{aligned}
\Hclass_{f,\chi,c}&:=({\rm id}\ot e_\chi)(\Hclass_{f,c}\ot {\chi_t})\in H^1(\rcf{c},T\ot\chi).
\end{aligned}\eeq

We finish this section with the proof of two lemmas which will be used in $\S\ref{S:EulerSystem}$.
Recall that we have fixed a decomposition $N\cO_K=\frakN\ol{\frakN}$.

\begin{lm}\label{L:ES2.norm}
 Let $\tau$ be the complex conjugation. Then
\[
(\Hclass_{f,\chi,c})^\tau=w_f\cdot \chi(\sg_\frakN)\cdot (\Hclass_{f,\chi^{-1},c})^{\sg_{\ol{\frakN}}},
\]
where $w_f\in \stt{\pm 1}$ is the Atkin--Lehner eigenvalue of $f$.
\end{lm}
\begin{proof}
We begin by noting that complex conjugation does indeed act on $\Hclass_{f,c}$, since the elliptic curve $A$ is defined over the real field $H_K^+$.
Let $w_N$ be the Atkin--Lehner involution, and set $\frakN_c:=\frakN\cap \cO_c$. We have the relations $w_N(\tau(x_c))=x_{\ol{\frakN}}$ and
$w_N^*[\Gamma_{\ol{\frakN}}]=N\cdot [\Gamma_c]$ in ${\rm NS}(A_c\x\Initial)$ (\cf \cite[Lemma 20]{Shnidman}), from which we find that \[
(w_N\x{\rm id})^*\Delta_{\ol{\frakN}}=N^{r-1}\cdot \Delta_c^\tau.
\]
Combined with \lmref{P:Galois.norm}, the above equation yields the lemma.
\end{proof}
\begin{lm}\label{L:ES3.norm}Let $\ell\ndivides cND_K$ be a prime inert in $K$. Let $\ol{\lam}$ be a prime of $\Qbar$ above $\ell$,
and let $\lam_{c\ell}$ and $\lam_{c}$ be the primes of $\rcf{c\ell}$ and $\rcf{c}$ below $\ol{\lam}$.
Denote by $\rcf{\lam_{c\ell}}$ and $\rcf{\lam_c}$ be the completions of $\rcf{c\ell}$ and $\rcf{c}$ at $\lam_{c\ell}$ and $\lam_c$, respectively. Then
\[\res_{\rcf{\lam_{c\ell}},\rcf{\lam_c}}(\loc_{\lam}(\Hclass_{f,\chi,c})^{\Frob_\ell})=\loc_{\lam_{c\ell}}(\Hclass_{f,\chi,c\ell}),\]
where $\Frob_\ell\in\Gal(\Q_\ell^{\rm ur}/\Q_\ell)$ is the Frobenius element of $\ell$.
\end{lm}
\begin{proof}Since $\chi$ is anticyclotomic and $\ell$ is inert, $\chi$ is a trivial character of $G_{\rcf{\lam_c}}$,
and hence $\Frob_\ell$ acts naturally on $H^1(\rcf{\lam_c},T\ot\chi)=H^1(\rcf{\lam_c},T)$.
The natural isogeny $A_c\to A_{c\ell}$ reduces to the Frobenius map $\Frob_\ell$ modulo $\ol{\lam}$, so we find that
\[(\Frob_\ell\x 1)(\wtd\Delta_c)=\wtd\Delta_{c\ell},\]
where $\wtd\Delta_?$ denotes the reduction of $\Delta_?$ modulo $\ol{\lam}$. The lemma follows.
\end{proof}

%We used \lmref{L:Galois.norm} for the last equality.

\subsection{The $p$-adic Gross--Zagier formula of Bertolini--Darmon--Prasanna}\label{subsec:BDP}
The purpose of this section is to give a mild extension of the \padic Gross--Zagier formula in \cite[Thm.~5.13]{BDP},
which relates the Bloch--Kato logarithm of generalized Heegner classes to the values of the $p$-adic $L$-function
$\mathscr{L}_{\frakp,\psi}(f)$ at characters outside the range of interpolation.
We keep the notation as in \subsecref{SS:twisted}.

\subsubsection*{Some notation for \padic representations}Let $L$ be a finite extension of $\Qp$, and
let $V$ be a finite dimensional $F$-vector space with a continuous $F$-linear action of $G_L$.
Recall that $\bfD_{\dR,L}(V)$ denotes the filtered $(L\ot_{\Qp} F)$-module $(\bfB_\dR\ot_{\Qp} V)^{G_L}$, where $\bfB_\dR$ is Fontaine's ring of \padic periods. If $V$ is a de Rham representation (\ie $\dim_L\bfD_{{\dR},L}(V)=\dim_{F} V$), then there is a canonical isomorphism $D_{\dR,E}(V)=E\ot_{L} D_{{\dR},L}(V)$ for any finite extension $E/L$. Denote by $\pairing$ the de Rham pairing
\[
\pairing:\bfD_{\dR,L}(V)
\times\bfD_{\dR,L}(V^*(1))\longto L\ot_{\Qp} F \longto\C_p,
\]
where $V^*=\Hom_F(V,F)$. Let $\bfB_{\cris}\subset \bfB_{\dR}$ be the crystalline period ring and define $\bfD_{\cris,L}(V):=(\bfB_\cris\ot_{\Qp}V)^{G_L}$. Then $\bfD_{\cris,L}(V)$ is an $(L_0\ot_{\Qp} F)$-module equipped with the action of crystalline Frobenius $\Phi$, where $L_0$ is the maximal unramified subfield of $L$. When $L=\Qp$, we write  $\bfD_{\dR}(V)=\bfD_{\dR,\Qp}(V)$ and $\bfD_{\cris}(V)=\bfD_{\cris,\Qp}(V)$.
If $V$ is a crystalline representation (\ie $\dim_{L_0}\bfD_{\cris, L}(V)=\dim_{F}V$), then we have a canonical isomorphism $L\ot_{L_0}\bfD_{\cris,L}(V)=\bfD_{\dR,L}(V)$.

Let $H^1_e(L,V)$ be the image of the Bloch--Kato exponential map
\[
{\rm exp}_{L,V}:\frac{\bfD_{{\rm dR},L}(V)}{{\rm Fil}^0\bfD_{{\rm dR},L}(V)+\bfD_{\cris,L}(V)^{\Phi=1}}
\hookrightarrow H^1(L,V),
\]
and $H^1_\BK(L,V)\subset H^1(L,V)$ be the Bloch--Kato `finite' subspace. %(\cite[(3.7.2)]{BK}).
If $\bfD_{\cris,L}(V)^{\Phi=1}=0$, then the natural inclusion $H^1_e(L,V)\subset H_f^1(L,V)$
is an equality (see for example \cite[Cor.~3.8.4]{BK}), and we define the Bloch--Kato
logarithm map
%\begin{equation}\label{def:log}
\[
\log:=\log_{L,V}:H^1_\BK(L,V)\longto
\frac{\bfD_{{\rm dR},L}(V)}{{\rm Fil}^0\bfD_{{\rm dR},L}(V)}
=({\rm Fil}^0\bfD_{{\rm dR},L}(V^*(1)))^\vee
%\end{equation}
\]
to be the inverse of the Bloch--Kato exponential. We also let $\exp^*$ be the dual exponential map
\[\exp^*:=\exp^*_{L,V}:H^1(L,V^*(1))\longto \Fil^0\bfD_{\dR,L}(V^*(1)),\]
obtained by dualizing ${\rm exp}_{L,V}$ with respect to the de Rham and local Tate pairings (\cf\cite[\S 2.4]{LZ2}).

Recall that we assumed $p=\frakp\frakpbar$ splits in $K$, with $\frakp$ induced by the fixed embedding $\imath_p:\Qbar\to \Cp$. If $E$ is a finite extension of $K$, we denote by $E_\frakp$ the completion of $E$ at the prime induced by $\imath_p$. With a slight the abuse of notation,
we call $E_\frakp$ the $\frakp$-adic completion of $E$, and for any $G_{E}$-module $V$, we let
\[
\loc_\frakp:H^1(E,V)\longto H^1(E_\frakp,V)
\]
denote the localization map. %Later we shall write $\log_\frakp$ for $\log\circ\,\loc_\frakp$.

\subsubsection*{Some de Rham cohomology classes}By the work of Scholl \cite{scholl100}, it is known that $V_f$ can be realised as a quotient of $H^{2r-1}_{\et}(W_{2r-2}{}_{/\Qbar},\Qp)\ot_{\Qp} F$, and we get the composite quotient map
\[
H^{2r-1}_{\dR}(W_{2r-2}/F)\cong \bfD_{\dR}(H^{2r-1}_{\et}(W_{2r-2}{}_{/\Qbar},\Qp)\ot_{\Qp} F)\longto\bfD_{\dR}(V_f)
\]
by %taking the de Rham functor and
applying the comparison isomorphism \cite{tsuji}. Let $\wtd\om_f\in H^{2r-1}_{\dR}(W_{2r-2}/F)$ be the differential form attached to the newform $f$ via the rule in \cite[Cor.~2.3]{BDP}, and let $\om_f\in \bfD_{\dR}(V_f)$ be the image of $\wtd\om_f$.

Let $L=H_{K,\frakp}$ be the $\frakp$-adic completion of $H_K$. The $\cO_K$-action on $A_{/L}$ gives rises to a canonical decomposition of the de Rham cohomology group $H^1_{\dR}(A/L)=H_{\dR}^{1,0}(A/L)\oplus H_{\dR}^{0,1}(A/L)$. Recall our fixed choice of \Neron differential $\om_A\in H^{1,0}_{\dR}(A/L)$, which determines $\eta_A\in H^{0,1}_{\dR}(A/L)$ by the requirement that $\pair{\om_A}{\eta_A}=1$  (\cf\cite[page 1051]{BDP}). We shall view $\om_A$, $\eta_A$ as elements in $\bfD_{\dR,L}(H^1_{\et}(A_{/\Qbar},\Qp))$ by the comparison isomorphism,
and let
\[
\om_A^{r-1+j}\eta_A^{r-1-j}\quad (-r<j<r)
\]
be the resulting basis for $\bfD_{\dR,L}({\rm Sym}^{2r-2}H^1_{\et}(A_{/\Qbar},\Qp))$,
where $\om_A^{r-1+j}\eta_A^{r-1-j}$ is as in \cite[(1.4.6)]{BDP}.

\subsubsection*{\padic Gross--Zagier formula}
Define the generalized Heegner class $\Hclass_{f,\chi}$ attached to $(f,\chi)$ by
\beq\label{E:zchi.norm}
\begin{aligned}
\Hclass_{f,\chi}&:={\rm cor}_{\rcf{\ctame p^s}/K}(\Hclass_{f,\chi,\ctame p^s})\\
%=&({\rm id}\ot e_\chi){\rm cor}_{\rcf{c_0p^n},K}(\Hclass_{f,c_0p^n}\ot\chi_t)\\
&\;=\sum_{\sg\in\Gal(\rcf{\ctame p^s}/K)}\chi_t(\sg)\cdot ({\rm id}\ot e_\chi)\Hclass_{f,\ctame p^s}^\sg\\
&\;=\sum_{[\fraka]\in\Pic\cO_{\ctame p^s}}\chi\varepsilon_{\rm cyc}^{1-r}(\sg_\fraka)\cdot ({\rm id}\ot e_\chi)\Hclass_{f,\fraka},
%\,(\text{by  \lmref{L:Galois.norm}}).
\end{aligned}
\eeq
where $\ctame p^s\cO_K$ is the conductor of $\chi$.
\begin{rem}\label{rem:loc-f}
By \cite[Thm.~3.3.1]{nekovarCRM}, the classes $\Hclass_{f,\chi,c}$ from $\S\ref{SS:twisted}$ lie in the Bloch--Kato
Selmer group ${\rm Sel}(\rcf{c},T\otimes\chi)\subset H^1(\rcf{c},T\otimes\chi)$; in particular,
$\loc_\frakp(\Hclass_{f,\fraka})\in H^1_\BK(\rcf{c,\frakp},T\otimes S^{r-1}(A))$ and
$\loc_\frakp(\Hclass_{f,\chi})\in H^1_\BK(K_\frakp,T\ot\chi)$.
\end{rem}
%, and the latter pairing is given by \[
%\pairing:\bfD_{{\rm dR},L_c}(V_{f,A})
%\times\bfD_{{\rm dR},L_c}(V_{f,A})\longrightarrow \wh\Q_p^{\rm ur}\otimes_{\Q_p}\C_p\xrightarrow{\imath_p\otimes 1}\C_p\quad(V_{f,A}:=V_f(r)\ot S^{1-r}(A)).\]

\begin{thm}\label{thm:BDP}
Suppose $p=\frakp\frakpbar$ splits in $K$. Let $\psi$ be an anticyclotomic Hecke character of infinity type $(r,-r)$
and conductor $\ctame\cO_K$ with $(\ctame,Np)=1$. %with $(\ctame,pND_K)=1$.
If $\wh\phi\in\mathfrak{X}_{p^\infty}$ is the $p$-adic avatar of an anticyclotomic Hecke character of infinity type $(r+j,-j-r)$ with $-r<j<r$
and conductor $p^n\cO_K$ with $n>1$, then
\begin{align*}
\frac{\mathscr{L}_{\frakp,\psi}(f)(\wh\phi^{-1})}{\Omega_p^{-2j}}
&=\frac{\frakg(\phi_\frakp^{-1})\phi_\frakp(p^n)\ctame^{1-r} \wh\psi_\frakp^{-1}(p^n)}{(r-1+j)!}
\cdot\pair{\log_\frakp(\Hclass_{f,\chi})}{\omega_f\otimes\omega_{A}^{r-1+j}\eta_{A}^{r-1-j}},
\end{align*}
where $\chi:=\wh\psi^{-1}\wh\phi$ and $\log_\frakp:=\log\circ\, \loc_\frakp$.
%and
%\[
%\langle\;,\;\rangle:\frac{\bfD_{\rm dR}(V_f(r)\otimes\chi)}{{\rm Fil}^0\bfD_{\rm dR}(V_f(r)\otimes\chi)}
%\times{\rm Fil}^0\bfD_{\rm dR}(V_f(r)\otimes\chi^{-1})\longrightarrow\C_p
%\]
%is the de Rham pairing.
\end{thm}

\begin{proof}
%Let $\wh\phi\in\mathfrak{X}_{p^\infty}$ be a finite order character of conductor $p^n\cO_K$ with $n>1$.
Let $t_\fraka$ be the Serre--Tate coordinate of
$\bfx_\fraka:=[(A_\fraka,\eta_\fraka)]\ot\Fpbar$. Since the Fourier coefficients
$\bfa_n(f^\flat)$ of $f^\flat$ vanish for $n$ divisible by $p$, we have
\[
U_p\wh f^\flat(t):=\sum_{\zeta^p=1}\wh f^\flat(t\zeta)=0.
\]
This implies that the associated measure ${\rm d} \wh f_\fraka^\flat$ is supported on $\Zp^\x$,
and hence by Lemma~\ref{L:1}, that
\[
\left(\wh f^\flat_{\fraka}\ot\psi_\frakp\wh\phi^{-1}|[\fraka]\right)(t_\fraka)
=\phi(\sg_\fraka)\rmN(\fraka)^{r+j}\cdot\left(\theta^{-j-r}\wh f^\flat\ot\psi_\frakp\phi_\frakp^{-1}\right)(t_\fraka),
\]
where $\theta$ is the operator acting on $\cO_{\wh S_{\bfx_\fraka}}$ as $t_\fraka\frac{d}{dt_\fraka}$.
Put $\xi:=\psi^{-1}\phi$. By \propref{P:1}, we thus find that
\begin{align*}
\sL_{\pp,\psi}(f)(\wh\phi^{-1})
&=\sum_{[\fraka]\in\Pic\cO_{\ctame}}
\psi(\fraka){\rm N}(\fraka)^{-r}\cdot\left(\wh f^\flat_\fraka\ot\psi_\frakp\wh\phi^{-1}\vert[\fraka]\right)(x_\fraka)\\
&=\sum_{[\fraka]\in\Pic\cO_{\ctame}}
\xi^{-1}(\fraka){\rm N}(\fraka)^{j}\cdot\left(\theta^{j-r}\wh f^\flat\ot\xi_\frakp^{-1}\right)(x_\fraka)\\
&=p^{-n}\frakg(\xi_\frakp^{-1})\cdot\sum_{[\fraka]\in\Pic\cO_{\ctame}}
\xi^{-1}(\fraka){\rm N}(\fraka)^{j}
\cdot\sum_{u\in (\Z/p^n\Z)^\x}\theta^{-j-r}\wh f^\flat(x_\fraka*\bfn(up^{-n}))\xi_\frakp(u).
\end{align*}
%where $p^n$ is the conductor of $\chi_\frakp$.
%and the definition ${\rm N}(\fa)=\ctame^{-1}\vert a\vert_{\mathbb{A}_K}^{-1}$
%where as usual $a\in\wh K^\times$ is such that $a\wh K\cap\cO_{ctame}=\fraka$.
Since $\theta^{-j-r}\wh f^\flat$ is a $p$-adic modular form
of weight $-2j$, we deduce from $(\ref{E:transformpadic})$ together with $(\ref{E:GaloisAction})$ that
\begin{equation}\label{eq:trans-Gal}
\theta^{-j-r}\wh f^\flat(x_\fraka*\bfn(up^{-n}))=\theta^{-j-r}\wh f^\flat({\rm rec}_K(a^{-1}u_\frakp p_\frakp^{-n})x_{\ctame p^n})u^{2j}.
\nonumber
\end{equation}
From the relations
\[
\chi(\rec_K(a))=\wh\xi(\rec_K(a))=\xi(a)a_\frakp^{j}a_\frakpbar^{-j},\quad
\varepsilon_{\rm cyc}(a)=\vert a\vert_{\mathbf{A}_K},\quad
\varepsilon_{\rm cyc}(u_\frakp p_\frakp^n)=u,
\]
it follows that
\beq\label{eq:prev}
\begin{aligned}
\mathscr{L}_{\frakp,\psi}(f)(\wh\phi^{-1})
&=p^{-n}\frakg(\xi^{-1}_\frakp)\ctame^{-j}\cdot
\sum_{[a]\in{\rm Pic}(\cO_{\ctame p^n})}\chi\varepsilon_{\rm cyc}^{j}(\rec_K(a))\cdot\theta^{-j-r}\wh f^\flat({\rm rec}_K(a)x_{\ctame p^n})\chi_\frakp(p^n)\\
&=\frakg(\xi_\frakp^{-1})\ctame^j p^{n(j-1)}\chi_\frakp(p^n)\cdot
\sum_{\sigma\in{\rm Gal}(\rcf{\ctame p^n}/K)}\chi\varepsilon_{\rm cyc}^{j}(\sigma)\cdot\theta^{-j-r}\wh f^\flat(x_{\ctame p^n}^\sigma).
\end{aligned}
\eeq

On the other hand, if $\sg=\sg_\fraka$ with $a\in\wh K^{(p\ctame N)\x}$
and $\fraka=a\wh\cO_{\ctame p^n}^\x\cap K$, then
\[\theta^{-j-r}\wh f^\flat(x_{\ctame p^n}^\sigma)=\theta^{-j-r}\wh f^\flat(x_{\fraka})=
\theta^{-j-r} f^\flat(x_{\fraka},\wh\om(\eta_{\fraka,p})),\]
where $\wh\om(\eta_{\fraka,p})$ is the differential form induced from the $p^\infty$-level structure $\eta_{\fraka,p}$ defined in \subsecref{subsec:CMpts}. For the isogeny $\varphi_\fraka: A\to A_\fraka$, one can verify that $\deg\varphi_\fraka=\ctame p^n\abs{a}_{\AK}^{-1}$ and
\[\varphi_\fraka^*(\wh\om(\eta_{\fraka,p}))=\ctame\cdot \wh\om(\eta_{\cO_K,p})=\frac{\ctame}{\Omega_p}\cdot\om_A.\]
Thus following the calculations in Proposition~3.24, Lemma~3.23, and Lemma~3.22 of \cite{BDP},
we see that
\begin{equation}\label{eq:bdp}
\theta^{-j-r}\wh f^\flat(x_{\ctame p^n}^\sigma)
=\frac{(\ctame p^n\abs{a}_{\AK}^{-1})^{-j-r-1}}{(r_1+j)!}\cdot\left(\frac{\ctame}{\Omega_p}\right)^{2j}
\cdot\langle\log_\frakp(\Hclass_{f,\fraka}^\flat)),\omega_f\otimes\omega_A^{r-1+j}\eta_A^{r-1-j}\rangle,
\end{equation}
where
\[
\Hclass_{f,\fraka}^\flat
:=\Hclass_{f,\fraka}-\mathbf{a}_p(f)p^{2j}\cdot \Hclass_{f,\fraka\cO_{\ctame p^{n-1}}}
-p^{2j+1}\cdot \Hclass_{f,\fraka\cO_{\ctame p^{n-2}}}.
\]
Substituting $(\ref{eq:bdp})$ into $(\ref{eq:prev})$, and using that $\phi$ has the exact conductor $p^n$ ($n>1$) and $\psi_\frakp$ is unramified,
we conclude that
\begin{align*}
\frac{\sL_{\pp,\psi}(f)(\wh\phi^{-1})}{\Omega_p^{2j}}
&=\frac{\frakg(\xi_\frakp^{-1})\ctame^{1-r}p^{-nr}\chi_\frakp(p^n)}{(r-1+j)!}
\cdot\sum_{\sigma\in{\rm Gal}(\rcf{\ctame p^n}/K)}
\chi\varepsilon_{\rm cyc}^{1-r}(\sg_\fraka)\cdot\langle{\rm log}_\frakp(\Hclass_{f,\fraka}^\flat),\omega_f\otimes\omega_A^{r-1+j}\eta_A^{r-1-j}\rangle\nonumber\\
&=\frac{\frakg(\phi_\frakp^{-1})\phi_\frakp(p^n)\ctame^{1-r} \wh\psi_\frakp^{-1}(p^n)}{(r-1+j)!}\sum_{[\fraka]\in \Pic\cO_c}\chi\varepsilon_{\rm cyc}^{1-r}(\sg_\fraka)\cdot\langle{\rm log}_\frakp(\Hclass_{f,\fraka})),\om_f\otimes\om_A^{r-1+j}\eta_A^{r-1-j}\rangle\\
&=\frac{\frakg(\phi_\frakp^{-1})\phi_\frakp(p^n)\ctame^{1-r} \wh\psi_\frakp^{-1}(p^n)}{(r-1+j)!}
\cdot\langle\log_\frakp(\Hclass_{f,\chi}),\omega_f\otimes\omega_A^{r-1+j}\eta_A^{r-1-j}\rangle
\end{align*}
as was to be shown.
\end{proof}

% !TEX root = HeegnerCycles.tex
% ML, May 19

\def\ctame{c_o}

\section{Explicit reciprocity law}

%Using a variant of the Perrin-Riou regulator map, we then deduce from the $p$-adic
%Gross-Zagier formula of Bertolini--Darmon--Prasanna \cite{BDP} (as extended in \S\ref{subsec:BDP}) a
%new construction of the $p$-adic $L$-functions $\mathscr{L}_{\frakp,\psi}(f)$ in terms of these classes.
%We assume throughout the following that the prime $p$ splits in $K$.

\subsection{The Perrin-Riou big logarithm}\label{subsec:PR-map}

\def\frakpbar{{\ol{\frakp}}}
\def\loc{{\rm loc}}
\def\rmD{\mathrm D}
\def\dR{{\rm dR}}
\def\ur{{\rm ur}}

In this section we deduce from the main result of \cite{LZ2} the construction of a variant
of the Perrin-Riou logarithm map for certain relative height one Lubin--Tate extensions.

For any commutative compact \padic Lie group $G$ and any complete discretely valued extension $E$ of $\bQ_p$, %with ring of integers $\cO=\cO_E$,
we let $\Lambda_{\cO_E}(G):=\prolim_n \cO_E[G/G^{p^n}]$, $\Lambda_E(G):=\Lambda_{\cO_E}(G)\ot_{\cO_E} E$,
and $\cH_{E}(G)$ be the ring of tempered \padic distributions on $G$ valued in $E$.
If $L$ is a finite extension of $\bQ_p$ and
$G$ is the Galois group of a $p$-adic Lie extension of $L_\infty=\cup_n L_n$ of $L$ with $L_n/L$ finite and Galois,
we define
\[
H^i_{\rm Iw}(L_\infty,V):=\biggl(\varprojlim_n H^i(L_n,T)\biggr)\otimes_{\bZ_p}\bQ_p,
\]
where $T$ is any $G_L$-stable lattice in $V$ (this is independent of the choice of $T$).

In the following, we let $L$ be a finite unramified extension of $\bQ_p$ with ring of integers $\cO_L$,
and let $\wh F^\ur$ denote the composite of $\wh\Q_p^\ur$ with a finite extension $F$ of $\Qp$.
We also let $t\in\bfB_{\rm dR}$ be Fontaine's $p$-adic analogue of $2\pi i$ associated with the compatible
system $\stt{\imath_p(\zeta_{p^n})}_{n=1,2,\ldots}$ of $p$-power roots of unity.
%In addition, we let $\zeta=(\zeta_n)_{n\geq 0}$ be a fixed choice of a compatible system of $p$-power roots of unity.

\begin{thm}
\label{thm:LZ}
Let $V$ be a crystalline $F$-representation of $G_L$ with non-negative Hodge--Tate weights, and assume that $V$ has
no quotient isomorphic to the trivial representation. Let $\mathfrak{F}$ be a relative height one Lubin--Tate formal group over $\cO_L/\Zp$, and
let $\Gamma:={\rm Gal}(L(\mathfrak{F}_{p^\infty})/L)\iso\Zp^\x$.
If $V^{G_{L(\mathfrak{F}_{p^\infty})}}=0$, there exists a $\Lambda_{\Zp}(\Gamma)$-linear map
\[
\mathcal{L}_{V}:H^1_{\rm Iw}(L(\mathfrak{F}_{p^\infty}),V)\longto
\mathcal{H}_{\wh F^\ur}(\Gamma)\otimes_L\bfD_{\cris,L}(V)
\]
with the following interpolation property: for any
$\mathbf{z}\in H^1_{\rm Iw}(L(\mathfrak{F}_{p^\infty}),V)$ and any locally algebraic character
$\chi:\Gamma\rightarrow\overline{\Q}_p^\times$ of Hodge--Tate weight $j$ and conductor $p^n$, we have
\begin{align*}\label{eq:interp}
\mathcal{L}_{V}^{}(\mathbf{z})(\chi)
&=\varepsilon(\chi^{-1})
\cdot\frac{\Phi^nP(\chi^{-1},\Phi)}{P(\chi,p^{-1}\Phi^{-1})}\cdot
\left\{
\begin{array}{lr}
\frac{(-1)^{-j-1}}{(-j-1)!}\cdot{\rm log}_{L,V(\chi^{-1})}(\mathbf{z}^{\chi^{-1}})\otimes t^{-j}
&\textrm{if $j<0$,}
\\
j!\cdot{\rm exp}^*_{L,V(\chi^{-1})^*(1)}(\mathbf{z}^{\chi^{-1}})\otimes t^{-j}
&\textrm{if $j\geq 0$},\nonumber
\end{array}\right.
\end{align*}
where
\begin{itemize}
%\item{} $\chi_o:=\chi\varepsilon_{\rm cyc}^{-j}$.
\item{}  $\varepsilon(\chi^{-1})$ and $P(\chi^{\pm},X)$ are the epsilon-factor and the $L$-factor for Galois characters $\chi$ and $\chi^\pm$,
respectively (see \cite[\S{2.8}]{LZ2} for the definitions).
\item{} $\Phi$ is the crystalline Frobenius operator on $\Qbar_p\otimes_L\bfD_{\cris,L}(V)$ acting trivially on the first factor.
%\item{}
%$P(\chi^{\pm},X)=\left\{
%\begin{array}{ll}
%1-p^{\mp j}\chi_{o}^{\mp}(\Frob_p)X &\textrm{if $\chi_o:=\chi\varepsilon_{\rm cyc}^{-j}$ is unramified,}\\
%1&\textrm{otherwise.}
%\end{array}\right.$
\item{} $\mathbf{z}^{\chi^{-1}}\in H^1(L,V(\chi^{-1}))$ is the specialisation of $\mathbf{z}$ at $\chi^{-1}$.
\end{itemize}
\end{thm}

\begin{proof}
Let $K_\infty\subset L\cdot\bQ_p^{\rm ab}$ be a $p$-adic Lie extension of $F$ containing $\wh F^\ur\cdot L(\mathfrak{F}_{p^\infty})$,
and set $G:={\rm Gal}(K_\infty/L)$. By \cite[Thm.~4.7]{LZ2} there exists a
$\Lambda_{\cO_L}(G)$-linear map
\[
\mathcal{L}_{V}^G:H^1_{\rm Iw}(K_\infty,V)\longto
\mathcal{H}_{\wh F^\ur}(G)\ot_L\bfD_{\cris,L}(V)
\]
satisfying the above interpolation formula for all continuous characters $\chi$ of $G$
(see [\emph{loc.cit.}, Thm.~4.15]).
Let $\mathcal{J}$ be the kernel of the natural projection $\mathcal{H}_{\wh F^\ur}(G)
\rightarrow\mathcal{H}_{\wh F^\ur}(\Gamma)$. The corestriction map
\[
H^1_{\rm Iw}(K_\infty,V)/\mathcal{J}\longto H^1_{\rm Iw}(L(\mathfrak{F}_{p^\infty}),V)
\]
is injective, and its cokernel is $H^2_{\rm Iw}(K_\infty,V)[\mathcal{J}]$, which vanishes
if $V^{G_{L(\mathfrak{F}_{p^\infty})}}=0$. Thus quotienting $\mathcal{L}_V^G$ by $\mathcal{J}$ we obtain a map
\[
H^1_{\rm Iw}(L(\mathfrak{F}_{p^{\infty}}),V)\cong H^1_{\rm Iw}(K_\infty,V)/\mathcal{J}
\longrightarrow\mathcal{H}_{\wh F^\ur}(\Gamma)\otimes_L\bfD_{\cris,L}(V)
\]
with the desired properties.
\end{proof}

\subsection{Iwasawa cohomology classes}\label{subsec:Iw}

Keep the notations from $\S\ref{sec:def-Heegner}$, and for any positive integer $c$,
let $\Sigma=\Sigma_c$ be a finite set of places of $K$ containing the primes above $pNc$.
Recall the Heegner classes $\Hclass_{f,\mathfrak{a}}\in H^1(K_c,T\otimes S^{r_1}(A))$ of (\ref{E:defnGHC})
attached to every integral $\cO_c$-ideal $\mathfrak{a}$.

In this section we further assume that $p=\frakp\frakpbar$ splits in $K$ and that the newform $f$ is ordinary at $p$, \ie the $p$-th Fourier coefficient $\bfa_p(f)\in \cO_F^\x$. The latter assumption will be crucial to construct,
out of the classes $\Hclass_{f,c p^n}=\Hclass_{f,\cO_{c p^n}}$ for varying $n$,
elements in the Iwasawa cohomololy groups
\[
H^1_{\rm Iw}(K_{cp^\infty},T):=\varprojlim_n H^1({\rm Gal}(K^\Sigma/K_{c p^n}),T),
\]
where %$G_{K_{\ctame p^n},\Sigma}={\rm Gal}(K^\Sigma/K_{\ctame p^n})$ with
$K^\Sigma$ is the maximal extension of $K$ unramified outside $\Sigma$.

\begin{defn}
Let $\alpha$ be the $p$-adic unit root of $X^2-\mathbf{a}_p(f)X+p^{2r-1}$. The \emph{$\alpha$-stabilized Heegner class}
$\Hclass_{f,\fraka,\alpha}\in H^1(K_c,T\ot S^{r-1}(A))$ is given by
\[
\Hclass_{f,\mathfrak{a},\alpha}:=\left\{
\begin{array}{ll}
\Hclass_{f,\mathfrak{a}}-\frac{p^{2r-2}}{\alpha}\cdot \Hclass_{f,\mathfrak{a}\cO_{c/p}}&\textrm{if $p\divides c$,}\\
\frac{1}{u_c}\left(1-\frac{p^{r-1}}{\alpha}\sg_\frakp\right)
\left(1-\frac{p^{r-1}}{\alpha}\sg_\frakpbar\right)
\cdot \Hclass_{f,\fraka}
&\textrm{if $p\nmid c$,}
\end{array}
\right.
\]
where $u_c=\#\cO_c^\times$ and $\sg_{\pp}, \sg_{\frakpbar}\in{\rm Gal}(K_c/K)$ are the Frobenius elements of $\pp$ and $\frakpbar$.
\end{defn}
\begin{lem}\label{lem:compat}
For all $c\geq 1$, we have
\[
{\rm cor}_{K_{cp}/K_c}(\Hclass_{f,cp,\alpha})=\alpha\cdot \Hclass_{f,c,\alpha}.
\]
\end{lem}

\begin{proof}
This follows from a straightforward computation using Proposition~\ref{prop:norm}.
\end{proof}

Now let $\Hclass^o_{f,\fraka}$ denote
the image of $\Hclass_{f,\fraka}$ under the natural map
\[
{\rm id}\ot e^o: H^1(K_c,T\otimes S^{r-1}(B))\longto H^1(K_c,T),
\]
where $e^o=e_{\bfone}$ is the projection \eqref{E:projector} attached to the trivial character
(so $\chi=\chi_t=\bfone$). Similarly as before, we shall simply write $\Hclass^o_{f,c,\alpha}$ for $\Hclass^o_{f,\cO_c,\alpha}$.
In view of Lemma~\ref{lem:compat}, the classes $\alpha^{-n}\cdot \Hclass^o_{f,c p^n,\alpha}$ are compatible under corestricion,
thus defining the Iwasawa cohomology class
\beq\label{E:bfz.E}
\mathbf{z}_{f,c,\alpha}:=\varprojlim_n\alpha^{-n}\cdot \Hclass^o_{f,c p^n,\alpha}\in H^1_{\rm Iw}(\rcf{c p^\infty},T).
\eeq
%in $H^1_{\rm Iw}(\rcf{c p^\infty},T)$.
%To be precise, we have  \[\wtd\bfz_{f,c,\alpha}=\sum_{\sg\in \wtd\Gamma_c/\Gamma_c}\bfz_{f,c,\alpha}^\sg \ot \sg\in H^1(K,T\ot\Lam_{\cO}(\wtd\Gamma_{c})).\]

For any character $\chi$ of $\Gal(\rcf{cp^\infty}/\rcf{c})$ we may consider the twist of
$\mathbf{z}_{f,c,\alpha}$ in $H^1(\rcf{c},T\otimes\chi)$. The next lemma compares the resulting classes,
for characters $\chi$ of \emph{finite} order, to the classes $z_{f,\chi,c}$ of $\S\ref{SS:twisted}$.

\begin{lem}\label{lem:finite-tw}Suppose that $p\ndivides c$.
Let $\chi:\Gal(\rcf{cp^\infty}/\rcf{c})\rightarrow\cO_{\Cp}^\times$ be a nontrivial
finite order character of conductor $cp^n$,
and let $\bfz_{f,c,\alpha}^\chi$ be the image of $\bfz_{f,c,\alpha}$ under the $\chi$-specialization map
\[
H^1_{\Iw}(\rcf{cp^\infty},T)\longto H^1(\rcf{c},T\otimes\chi).
\]
Then
\[
\bfz_{f,c ,\alpha}^\chi=\al^{-n}\cdot\Hclass_{f,\chi,c}.
\]
\end{lem}
\begin{proof}
Directly from the definition of $\bfz_{f,c,\alpha}$, by
\cite[Lemma~2.4.3]{Rubin-ES} we see that
\[\bfz_{f,c ,\alpha}^\chi=\al^{-n}\sum_{\sg\in\Gal(\rcf{c p^n}/\rcf{c})}\chi(\sg)(\Hclass_{f,c p^n,\alpha}^o)^\sg,\]
and since $\chi$ is nontrivial, we may replace $\Hclass_{f,c p^n,\alpha}^o$ by
$\Hclass_{f,c p^n}^o$ in this equation. By \propref{P:Galois.norm} %and \eqref{E:zchi.norm}.
(noting that $e_\chi$ can be taken to be $e^o$ with $\chi_t=\chi$), the result follows
from the definition $(\ref{E:defn.norm})$ of $z_{f,\chi,c}$.
\end{proof}

\subsection{Explicit reciprocity law for generalized Heegner cycles}

We now specialize the local machinery of $\S\ref{subsec:PR-map}$ to the global setting in $\S\ref{subsec:Iw}$. In
particular, we assume that $p=\pp\frakpbar$ splits in $K$ and that the newform $f\in S_{2r}^{\rm new}(\Gamma_0(N))$
is ordinary at $p$.

Let $\psi$ be an anticyclotomic Hecke character of infinity type $(r,-r)$ and conductor $\ctame\cO_K$ with $p\nmid\ctame$.
Recall that the \padic avatar $\wh\psi$ is a \padic character of $\Gal(\rcf{\ctame p^\infty}/K)$ valued in some finite extension
$\Qp$ which by the hypothesis on the conductor is crystalline at the primes above $p$. Let $F$ be a finite extension of $\Q_p$ containing the Fourier coefficients of $f$
and the values of $\wh\psi$, and let $V_f\cong F^2$ be the Galois representation associated to $f$.  %satisfying ${\rm det}V_f=\varepsilon_{\rm cyc}^{1-2r}$.
We assume throughout that $p\nmid N$, so that $V_f\vert_{G_{\Q_p}}$ is crystalline.

%\begin{lem}\label{lem:eta}
%For every root $\alpha$ of $X^2-\mathbf{a}_p(f)X+p^{2r-1}$,
%there exists a class $\omega_{f,\alpha}\in\bfD_{\rm cris}(V_f)$ such that
%$\Phi\omega_{f,\alpha}=\alpha'\cdot\omega_{f,\alpha}$, where $\alpha'=p^{2r-1}/\alpha$.
%\end{lem}

%\begin{proof}
%The crystalline Frobenius $\Phi$ on $\bfD_{\rm cris}(V_f)$
%and the $U_p$-operator on de Rham cohomology of $X_1(N)$ %(with nontrivial coefficients in $r>1$)
%are related by $\Phi=p^{2r-1}V=p^{2r-1}U_p^{-1}$ (see \cite[Lemma~2.10]{DR1}, for example).
%Hence the class of the differential
%\[
%\omega_{f,\alpha}:=\omega_f-\frac{p^{2r-1}}{\alpha}\omega_{Vf}
%\]
%is a $\Phi$-eigenvector with eigenvalue $p^{2r-1}/\alpha$, as desired.
%\end{proof}
By $p$-ordinarity, there is an exact sequence of $G_{\Q_p}$-modules
\[
0\longrightarrow\mathscr{F}^+V_f\longrightarrow V_f\longrightarrow\mathscr{F}^-V_f\longrightarrow 0
\]
with $\mathscr{F}^{\pm}V_f\cong F$ and with the $G_{\Q_p}$-action on $\mathscr{F}^+V_f$ being unramified
(see \cite[Thm.~2.1.4]{wiles88}). Let $T\subset V_f(r)$ be a $G_{\bQ}$-stable lattice as in $\S\ref{SS:GHCI}$, and set
$\mathscr{F}^+T:=\mathscr{F}^+V_f(r)\cap T$. %and $\mathscr{F}^-T:=T/\mathscr{F}^+T$.
Let
\[V:=V_f(r)\otimes\wh\psi^{-1}|_{G_{K_\frakp}},\quad\wh\psi_\frakp:=\wh\psi|_{G_{K_\frakp}}.\]
The dual representation $V^*$ is $\Hom_F(V,F)=V_f(r-1)\ot\wh\psi_\frakp$. Define\[
\mathscr{F}^{\pm}V:=\mathscr{F}^{\pm}V_f(r)\otimes\wh\psi_\frakp^{-1},\quad\quad
\mathscr{F}^{\mp}V^*:={\rm Hom}_F(\mathscr{F}^{\pm}V,F).
\]

We next introduce an element $\om_{f,\psi}\in\bfD_{\cris,L}(\sF^-V^*)$. Recall that $A$ is the canonical CM elliptic curve
over the Hilbert class field $H_K$ fixed in $\S\ref{sec:def-Heegner}$.  Let $\kappa_A:G_{H_K}\to \Aut T_\frakp(A)\cong\Zp^\x$ be the character
describing the Galois action on the $\frakp$-adic Tate module of $A$.
%, which is the $p$-adic avatar of the CM character $\wtd\kappa_A$ introduced in \subsecref{SS:twisted}.
Thus $H^1_{\et}(A_{/\Qbar},\Qp)\cong\kappa_A^{-1}\oplus\kappa_A\varepsilon^{-1}_{\rm cyc}$ as $G_{H_K}$-modules. Recall that $t\in\bfB_{\rm dR}$
denotes Fontaine's $p$-adic analogue of $2\pi i$ and set
\[
t_A:=\Omega_p t,
\]
where $\Omega_p$ is the \padic CM period defined in \subsecref{SS:CMperiods}.
Then $t_A$ generates $\bfD_{{\rm cris},F}(\kappa_A^{-1})$, and
according to the discussion in \cite[\S{II.4.3}]{de_shalit} we have
\begin{equation}\label{eq:t}
\om_A=t_A,\quad\quad\eta_A=t_A^{-1}t.
\end{equation}
On the other hand, note that the character $\wh\psi_\frakp\e_{\rm cyc}^{-r}$ is trivial on the inertia group,
and $\bfD_{\rm cris}(\wh\psi_\frakp(-r))=F\om_\psi$ is a one-dimensional $F$-vector space with generator $\om_\psi$.
%Let $\alpha$ be the $p$-adic unit root of $X^2-\mathbf{a}_p(f)X+p^{2r-1}$ and
%let $\omega_{f,\alpha}\in\bfD_{\rm cris}(V_f)$ be as in Lemma~\ref{lem:eta}.
Define the class
\begin{align*}
\omega_{f,\psi}&:=\omega_{f}\ot t^{1-2r}\otimes\om_\psi\in \bfD_{\rm cris}(V^*).
%&:=\omega_{f,\alpha}\otimes t\otimes\omega_A^{2r}
%=\omega_{f,\alpha}\otimes t^{1-2r}\cdot\Omega_p^{-2r}
%\in\bfD_{\rm cris}(V^*(1)).
\end{align*}
%(Note that $V^*\cong V_f(r_1)\otimes\kappa_\frakp^r=V_f(-1)\otimes\kappa_A^{2r}$.)
With a slight abuse of notation, we shall still denote by $\omega_{f,\psi}$ its image
under the natural projection $\bfD_{\rm cris}(V^*)\rightarrow\bfD_{\rm cris}(\mathscr{F}^-V^*)$,
which is nonzero by weak-admissibility (\cite[\S{3.3}]{fontaine}).
Moreover, since the periods of unramified characters
lie in $F^{\rm ur}:=\wh\Q_p^{\rm ur}F\subset\bfB_{\rm cris}$,
there exists a non-zero element $\Omega_\psi\in\wh F^{\rm ur}$ such that,
for all $x\in\bfD_{\rm cris}(\mathscr{F}^+V)$, we have
\beq\label{E:periodpsi}
\pair{x}{\om_{f,\psi}}=\pair{x}{\om_{f}\ot t^{1-2r}}\Omega_\psi,
\eeq
and the action of the crystalline Frobenius $\Phi$ is given by
\beq\label{E:Fcriseigenvalue}\pair{x}{\Phi\om_{f,\psi}}
%=\al'\cdot p^{1-2r}\cdot \e_{\rm cyc}^r/\wh\psi_\frakp(\sg_p)\cdot\pair{x}{\om_{f,\psi}}=
=\al^{-1}\wh\psi_\frakp(p)\cdot \pair{x}{\om_{f,\psi}}.
\eeq

Let $L_\infty/L$ denote the $\frakp$-adic completion of $\rcf{\ctame p^\infty}/\rcf{\ctame}$ and let $\Gamma:=\Gal(L_\infty/L)$. Let $h_p$ be the order of $\frakp$ in ${\rm Pic}(\cO_{\ctame})$, and write $\frakp^{h_p}=(\pi)$ with $\pi\in\cO_{\ctame}$.
Then $L$ is the unramified extension of $\Qp$ of degree $h_p$. By local class field theory, $L_\infty$ is contained
in the extension $L(\mathfrak{F}_{p^\infty})$ obtained by adjoining to $L$ the torsion points of the relative height one Lubin--Tate
formal group $\frakF$ attached to the uniformizer $\pi/\bar\pi$ (see \cite[Prop.~37]{Shnidman} for details). Note that the element $\rec_\frakp(\pi/\ol{\pi})$ fixes $L(\frakF_{p^\infty})$ and acts on $\sF^+V$ by a multiplication by $(\frac{\pi}{\ol{\pi}p^{h_p}})^r\alpha^{h_p}$, which is not $1$ by Ramanujan's conjecture for $f$ \cite{deligne355}, \cite{deligne-weil-II}.
This implies that $(\sF^+V)^{G_{L(\frakF_{p^\infty})}}=0$, %The character
%\[
%\kappa:=\kappa^2_A\varepsilon_{\rm cyc}=\kappa_A\kappa_A^{-c}
%\]
%factors through
%$\Gamma={\rm Gal}(H_{\ctame p^\infty}/H_{\ctame})$, so $\kappa_\frakp:=\kappa\vert_{G_L}$ factors through
%$\Gamma_\frakp:={\rm Gal}(L_\infty/L)$.
and hence we may consider the big logarithm map
$\mathcal{L}_{\mathscr{F}^+V}$ of \thmref{thm:LZ} over the extension $L_\infty/L$.
\begin{lem}\label{lem:growth}
The composition of
$\mathcal{L}_{\sF^+V}$ with the natural pairing
\[
\langle-,\omega_{f,\psi}\rangle:\mathcal{H}_{\wh F^{\rm ur}}(\Gamma)\otimes\bfD_{\rm cris}(\mathscr{F}^+V)
\times\bfD_{\rm cris}(\mathscr{F}^-V^*)\longrightarrow\mathcal{H}_{\wh F^{\rm ur}}(\Gamma)
\]
has image contained in the Iwasawa algebra $\Lambda_{\wh F^{\rm ur}}(\Gamma)$.
\end{lem}
\begin{proof}
This follows easily from the Frobenius eigenvalue formula \eqref{E:Fcriseigenvalue} and \cite[Prop.~4.8]{LZ2}.
\end{proof}

In what follows, we make the identification $\Gal(\rcf{\ctame p^\infty}/\rcf{\ctame})\iso\Gamma=\Gal(L_\infty/L)$ via the restriction map. Let $\rho:\Gamma\to\cW^\x$ be a continuous character, where $\cW$ is the ring of the integers in $\wh F^{\rm ur}$. For every $\bfz\in H^1_{\Iw}(\rcf{\ctame p^\infty},T)$, denote by $\bfz\ot\rho\in H^1_{\Iw}(\rcf{\ctame p^\infty},T\ot\rho)$ the \emph{$\rho$-twist} of $\bfz$. By definition, for any $\chi:\Gamma \to\cW^\x$, we have
\[(\bfz\ot \rho)^\chi=\bfz^{\rho\chi}\in H^1(\rcf{\ctame},T\ot \rho\chi).\]

As shown in \cite[Prop.~2.4.2]{LZ-Coleman},
there is an isomorphism $H^1_{\rm Iw}(K_{\ctame p^\infty},T)\simeq H^1(K_{\ctame},T\otimes\Lam_{\cO_F}(\Gamma))$.
Thus letting $\wtd\Gamma_{\ctame}:=\Gal(\rcf{\ctame p^\infty}/K)$
we may view $\bfz_{f,\ctame,\alpha}$ as an element in $ H^1(\rcf{\ctame},\Lam_{\cO_F}(\wtd\Gamma_{\ctame}))$
via %the composite map
\[H^1_{\Iw}(\rcf{\ctame p^\infty},T)\simeq H^1(\rcf{\ctame},\Lam_{\cO_F}(\Gamma))\longto H^1(\rcf{\ctame},\Lam_{\cO_F}(\wtd\Gamma_{\ctame})),\] and define  \beq\label{E:corbfz.E}\bfz_{f}:={\rm cor}_{\rcf{\ctame}/K}(\bfz_{f,\ctame,\alpha})\in  H^1(K,T\ot \Lam_{\cO_F}(\wtd\Gamma_{\ctame})). \eeq

Similarly as in $\S\ref{subsec:BDP}$ (see Remark~\ref{rem:loc-f}), %by \cite[Thm.~3.3.1]{nekovarCRM}
the Heegner classes $z_{f,\mathfrak{a}}^o$
lie in the Bloch--Kato Selmer group ${\rm Sel}(K_c,T)\subset H^1(K_c,T)$; in particular, ${\rm loc}_\pp(z^o_{f,\mathfrak{a}})\in H^1_f(K_{c,\pp},T)$.
On the other hand, by \cite[Lem.~9.6.3]{nekovar310} and [\emph{loc.cit.}, Prop.~12.5.9.2] the Bloch--Kato finite subspace $H^1_f(K_{c,\pp},T)$
is identified with the image of the natural map $H^1(K_{c,\pp},\mathscr{F}^+T)\rightarrow H^1(K_{c,\pp},T)$, and hence
${\rm loc}_{\pp}(\bfz_{f,\ctame,\alpha})$ naturally defines a class in $H^1_{\rm Iw}(L_\infty,\mathscr{F}^+T)$.

\begin{defn}[Algebraic anticyclotomic \padic $L$-functions]\label{defn:AlgL}Let $\wh\psi:\wtd\Gamma_{\ctame}\to\cO_F^\x$ be as before. Set
\begin{align*}\cL_{\frakp}^*(\bfz_f\ot\wh\psi^{-1}):=&\;{\rm cor}_{\rcf{\ctame}/K}(\cL_{\sF^+V}(\loc_\frakp(\bfz_{f,\ctame,\alpha}\ot \wh\psi^{-1}))\\
=&\;\sum_{\sg\in\wtd\Gamma_{\ctame}/\Gamma_{\ctame}}\cL_{\sF^+V}(\loc_\frakp(\bfz_{f,\ctame,\alpha}^\sg\ot \wh\psi^{-1}))\wh\psi(\sg^{-1})\in\bfD_{\cris}(\sF^+V)\ot \Lam_{\wh F^{\rm ur}}(\wtd\Gamma_{\ctame}),\end{align*}
and letting $\res_{\rcf{p^\infty}}:\wtd\Gamma_{\ctame}\to\wtd\Gamma=\Gal(\rcf{p^\infty}/K)$ be the restriction map, define
\[\cL_{\frakp,\psi}(\bfz_f):=\res_{\rcf{p^\infty}}(\cL_\frakp^*(\bfz_f\ot\wh\psi^{-1}))\in\bfD_{\cris}(\sF^+V)\ot\Lam_{\wh F^{\rm ur}}(\wtd\Gamma).\]
\end{defn}

\begin{thm}\label{thm:walds}
Suppose $p=\frakp\frakpbar$ splits in $K$. Let $f\in S_{2r}^{\rm new}(\Gamma_0(N))$ with $p\nmid N$
be a $p$-ordinary newform, and let $\psi$ be an anticyclotomic Hecke character of infinity type $(r,-r)$ and conductor $\ctame\cO_K$ with $p\nmid\ctame$.
Then
\[
\left\langle\mathcal{L}_{\frakp,\psi}(\bfz_f),
\omega_{f}\otimes t^{1-2r}\right\rangle=(-\ctame^{r-1})\cdot \mathscr{L}_{\frakp,\psi}(f)\cdot \sg_{-1,\frakp}\in \Lam_{\wh F^{\rm ur}}(\wtd\Gamma),
\]
where $\sg_{-1,\frakp}:=\rec_\frakp(-1)|_{\rcf{p^\infty}}\in \wtd\Gamma$ is an element of order two.
%where $\mathbf{z}:=\mathbf{z}_{f,\ctame,\alpha}$.
\end{thm}

\begin{proof}
Let $\wh\phi:\wtd\Gamma\to \Cp^\x$ be the $p$-adic avatar of a Hecke character $\phi$
of infinity type $(r,-r)$ and conductor $p^n$, for any $n>1$, and set $\chi:=\wh{\psi}^{-1}\wh\phi$, which is
a finite order character. Applying \lmref{lem:finite-tw},
we find that $\bfz_f^\chi=\alpha^{-n}\cdot\Hclass_{f,\chi}$, where $\bfz_f^\chi$ denotes the $\chi$-specialization of $\bfz_f$.
By Theorem~\ref{thm:BDP} (with $j=0$), we thus obtain
\beq\label{eq:bdp}
\begin{aligned}
\mathscr{L}_{\frakp,\psi}(f)(\wh\phi^{-1})
&=\frac{\mathfrak{g}(\phi_\frakp^{-1})\phi_{\frakp}(p^n)\ctame^{1-r}\wh\psi_\frakp^{-1}(p^n)}{(r-1)!}
\cdot\langle{\rm log}_\frakp(\Hclass_{f,\chi}),\omega_f\otimes t^{1-r}\rangle
\\&=\alpha^n\cdot\frac{\mathfrak{g}(\phi_\frakp^{-1}){\phi_{\frakp}}(p^n)\ctame^{1-r}\wh\psi_\frakp^{-1}(p^n)}{(r-1)!}
\cdot\langle{\rm log}_\frakp(\mathbf{z}_f^{\chi})\otimes t^{r},\omega_f\otimes t^{1-2r}\rangle.
\end{aligned}
\eeq
 On the other hand, a straightforward calculation reveals that
the $\varepsilon$-factor for the \padic Galois character $\wh\phi_\frakp$ defined in \cite[\S{2.8}]{LZ2} agrees with Tate's $\varepsilon$-factor for $\phi_\frakp$, \ie $\varepsilon(\wh\phi_\frakp)=\e(0,\phi_\frakp)=\frakg(\phi_\frakp^{-1})\phi_\frakp(-p^{n})$. Therefore, by Theorem~\ref{thm:LZ} combined with \eqref{E:periodpsi} and \eqref{E:Fcriseigenvalue}, we find that
\beq\label{eq:LZ}\begin{aligned}
&\left\langle\mathcal{L}_{\frakp,\psi}(\bfz_f),\omega_{f}\ot t^{1-2r}\right\rangle(\wh\phi^{-1})\\
=&
\left\langle\mathcal{L}_{\frakp,\psi}(\bfz_f),\omega_{f,\psi}\right\rangle(\wh\phi^{-1})\cdot \Omega_\psi^{-1}\\
=&\;\frakg(\phi_\frakp^{-1})\phi_\frakp(-p^n)
\cdot\alpha^n\wh\psi_\frakp^{-1}(p^n)
\cdot\frac{(-1)^{r-1}}{(r-1)!}
\cdot\langle{\rm log}_{\frakp}(\mathbf{z}_f^{\chi})\otimes t^r,\omega_{f,\psi}\rangle\Omega_\psi^{-1}\\
=&-\wh\phi(\sg_{-1,\frakp})\alpha^n\cdot\frac{\frakg(\phi_\frakp^{-1})\phi_\frakp(p^n)\wh\psi_\frakp^{-1}(p^n)}{(r-1)!}
\cdot\langle{\rm log}_{\frakp}(\mathbf{z}_f^{\chi})\otimes t^r,\omega_{f}\ot t^{1-2r}\rangle.
\end{aligned}\eeq

Since $\psi$ has conductor prime to $p$, we have $\frakg(\phi_\pp^{-1})=\frakg(\chi_\frakp^{-1})$ in formula
$(\ref{eq:LZ})$. Comparing $(\ref{eq:bdp})$ and $(\ref{eq:LZ})$, we see that both sides of the
desired equality agree when evaluated at $\wh\phi^{-1}$. Since the set of all such characters $\wh\phi$ (for varying $n>1$)
is Zariski-dense in the space of continuous \padic characters of $\wtd\Gamma$, and both sides of the desired equality are elements in the Iwasawa algebra $\Lam_{\wh F^\ur}(\wtd\Gamma)$, the result
follows from the $p$-adic Weierstrass preparation theorem.
\end{proof}

We are now ready to prove the ``explicit reciprocity law" relating the image of generalized Heegner classes under the dual exponential map
to the central values of the Rankin $L$-series $L(f,\chi,s)$ associated with $f$ and the theta series of an
anticyclotomic locally algebraic Galois character $\chi$ of conductor $c\cO_K$. Recall that $L(f,\chi,s)$ is defined by the analytic continuation of the Dirichlet series
\[L(f,\chi,s)=\zeta(2s+1-2r)\sum_{\fraka}\frac{\bfa_{\rmN(\fraka)}(f)\chi(\sg_\fraka)}{\rmN(\fraka)^s}\quad(\Re(s)>r+\onehalf),\]
where $\fraka$ runs over ideals of $\cO_K$ with $(\fraka,c\cO_K)=1$. In terms of automorphic $L$-functions, we have
\[L(f,\chi,s)=L(s+\frac{1}{2}-r,\pi_K\ot\chi_\A),\]
where $\pi_K$ is the base change of the automorphic representation $\pi$ generated by $f$,
and $\chi_\A$ is the Hecke character of $K^\x$ associated to $\chi$. Also, recall from $(\ref{eq:t})$
the relation $\omega_A\eta_A=t$.

\begin{cor}\label{cor:ERL}
With notations and assumptions as in Theorem~\ref{thm:walds},
let $\chi:\Gal(\rcf{\ctame p^\infty}/K)\to\cO_F^\x$ be a locally algebraic \padic character of infinity type $(j,-j)$ with $j\geq r$
and conductor $\ctame p^n\cO_K$. Then
\begin{align*}\label{eq:ERL}
\pair{\exp^*(\loc_\frakp(\mathbf{z}_f^{\chi^{-1}}))}{\om_{f}\ot \om_A^{-j-r+1}\eta_A^{j-r+1}}^2
&=c_{f,K}\cdot e_\frakp'(f,\chi)^2\cdot\left(p^{2r-1}/\alpha^2\right)^n\cdot\chi^{-1}\psi(\mathfrak{N})\cdot\frac{L^{\rm alg}(f,\chi,r)}{\Gamma(j-r+1)^2},
\end{align*}
where $c_{f,K}=8u_K^2\sqrt{D_K}\ctame^{2r-1}\e(f)$,
\begin{align*}
e'_\frakp(f,\chi)&=
\begin{cases}
\left(1-\alpha^{-1}\chi(\sg_\frakp)p^{r-j-1}\right)\left(1-\alpha^{-1}\chi(\sg_{\frakpbar})p^{r-j-1}\right)&\textrm{if $n=0$},\\
1&\textrm{if $n>0$;}
\end{cases}\\
%\intertext{ and }
L^{\rm alg}(f,\chi,r)&=\frac{\Gamma(j-r+1)\Gamma(j+r)}{(4\pi)^{2j+1}(\Im\vartheta)^{2j}}
\cdot L(f,\chi,r).
\end{align*}
%In particular,
%\[
%\loc_\frakp(\mathbf{z}^{\wh\chi^{-1}}_{o}_o)=0
%\quad\Longleftrightarrow\quad
%L(\frac{1}{2},\pi_K\otimes\chi)=0.
%\]
\end{cor}

\begin{proof}
Choose an anticyclotomic Hecke character $\psi$ of infinity type $(r,-r)$ and conductor $\ctame$ such that the character $\wh\phi=\chi\wh\psi^{-1}$ is of infinity type $(j-r,r-j)$ and conductor $p^n$.
Assume first that $n>0$. By \thmref{thm:walds} and \thmref{thm:LZ}, we then see that
\beq\label{eq:alg}\begin{aligned}
&\quad\left\langle\mathcal{L}_{\frakp,\psi}(\bfz_f),\omega_{f}\otimes t^{1-2r}\right\rangle
(\wh\phi)\\
&=\frakg(\phi_\frakp)\phi_\frakp(-p^n)
\cdot\alpha^n\wh\psi_\frakp^{-1}(p^n)
\cdot(j-r)!\cdot\langle{\rm exp}^*(\loc_{\frakp}(\mathbf{z}_f^{\chi^{-1}}))\otimes t^{r-j},\omega_{f}\otimes t^{1-2r}\rangle\\
&=\pm\alpha^n\cdot\varepsilon(0,\phi_\frakp^{-1}\psi_\frakp^{-1})p^{-nr}\cdot(j-r)!
\cdot\langle{\rm exp}^*(\loc_{\frakp}(\mathbf{z}_f^{\chi^{-1}})),\omega_{f}\otimes t^{1-r-j}\rangle\\
&=\pm\al^n\cdot\e(0,\phi_\frakp^{-1}\psi_\frakp^{-1})p^{-nr}\cdot\Gamma(j-r+1)\cdot\pair{\exp^*(\loc_\frakp(
\mathbf{z}_f^{\chi^{-1}}))}{\om_{f}\ot t^{1-r}\ot\om_A^{-j}\eta_A^{j}}\Omega_p^{2j}.
\end{aligned}
\eeq
%(Note that $\om_f\ot t^{-r_1}\in\bfD_{\rm cris}(V_f(r-1))=\bfD_{\rm cris}(V_f(r)^*)$,
%and $\om_A^{-j}\eta_A^j\in\bfD_{\rm cris}(\kappa_A\kappa_A^{-c})=\bfD_{\rm cris}(\kappa)$.)

On the other hand, by the interpolation formula in Proposition~\ref{P:interpolationformula} (with $m=j-r$),
we have
\begin{equation}\label{eq:an}
\left(\frac{\sL_{\frakp,\psi}(f)(\wh\phi)}{\Omega_p^{2j}}\right)^2=L^{\rm alg}(f,\chi,r)\cdot \e(\onehalf,\psi_\frakp\phi_\frakp)^{-2}\cdot \phi(\mathfrak{N}^{-1})\cdot  2^3u_K^2\sqrt{D_K}\ctame\varepsilon(f)
\end{equation}
where
\[
\e(\onehalf,\psi_\frakp\phi_\frakp)^{-2}=\e(\onehalf,\psi_\frakp^{-1}\phi_\frakp^{-1})^2
=\e(0,\psi_\frakp^{-1}\phi_\frakp^{-1})^2p^{-n}.
\]
Combining $(\ref{eq:alg})$ and $(\ref{eq:an})$ with the equality in
Theorem~\ref{thm:walds}, we find that
\[
\pair{\exp^*(\loc_\frakp(\mathbf{z}_f^{\chi^{-1}}))}{\om_f\ot t^{1-r}\ot\om_A^{-j}\eta_A^{j}}^2
=\frac{L^{\rm alg}(f,\chi,r)}{\Gamma(j-r+1)^2}\cdot\left(p^{2r-1}/\alpha^2\right)^n\cdot
 \phi(\mathfrak{N}^{-1})\cdot  2^3u_K^2\sqrt{D_K}\ctame^{2r-1}\varepsilon(f).
\]
This proves the result when $n>0$; the case $n=0$ is similar, and is left to the reader.
\end{proof}

%\begin{rem}
%The equality in Corollary~\ref{cor:ERL} may be equivalently written as
%\[
%\pair{\exp^*(\loc_\frakp(\mathbf{z}_{o}^{\wh\chi^{-1}}))}{\om_{f,\alpha}\ot t\ot\om_A^{-r-j}\eta_A^{-2r-j}}^2
%=\frac{L^{\rm alg}(\frac{1}{2},\pi_K\ot\chi^{-1})}{\Gamma(j-r+1)^2}\cdot\left(p^{2r-1}/\alpha^2\right)^n\cdot c_{f,\phi}
%\]
%which is an  equality between algebraic numbers.
%\end{rem} 
% !TEX root = HeegnerCycles.tex

\newcommand{\euler}[2]{\bfz^{#2}_{#1}}
\def\cH{\mathcal H}
\def\Chi{\chi}
\def\rmd{\,{\rm d}}
\def\Sg{\Sigma}
\def\cH{H}
\def\fin{f}
\def\sing{s}
\def\cY{Y}
\def\pd{\partial}
\def\Chi{\chi}
\def\Chibar{\chi^{-1}}
\def\cS{S}
\def\Bone{B_2}
\def\Btwo{B_1}
\def\Bthree{B_3}
\def\Bfour{B_4}
\def\Esys{\bfc}

\section{The arithmetic applications}
\label{sec:arith-applic}

In this section, we state our main arithmetic applications in this paper,
whose proof will be based on the results of the preceding sections combined with Kolyvagin's method of Euler systems. The details of the Euler system argument will be given in $\S{7}$.

\subsection{Setup and running hypotheses}\label{subsec:setup}

Let $f\in S^{\rm new}_{2r}(\Gamma_0(N))$ be a newform, and let $F/\Q_p$ be a finite extension with the ring of integers $\cO=\cO_F$
containing the Fourier coefficients of $f$. Let
\[
\rho_f:G_\Q\longrightarrow \GL_F(V_f)\simeq\GL_2(F)
\]
be the \padic Galois representation attached to $f$, and set
$\rho_f^*:=\rho\ot\e_{\rm cyc}^r$ and $V:=V_f(r)$. Let $\chi\colon\Gal(\rcf{\ctame p^\infty}/K)\to\cO^\x$ be a locally algebraic character of infinity type $(j,-j)$
and conductor $c\cO_K$ and set $V_{f,\chi}:=V\vert_{G_K}\otimes\chi$. Recall that the \emph{Bloch--Kato Selmer group} of $V_{f,\chi}$ is defined by
\[
{\rm Sel}(\cK,V_{f,\chi}):={\rm ker}\left\{H^1(\cK,V_{f,\chi})\longto\prod_{v}
\frac{H^1(\cK_v,V_{f,\chi})}{H^1_\BK(\cK_v,V_{f,\chi})}\right\},
\]
where
\[
H^1_\BK(\cK_v,V_{f,\chi})=
\left\{
\begin{array}{ll}
{\rm ker}\left(H^1(\cK_v,V_{f,\chi})\longto H^1(\cK_v^{\rm ur},V_{f,\chi})\right) &\textrm{if $p\nmid v$},\\
{\rm ker}\left(H^1(\cK_v,V_{f,\chi})
\longto H^1(\cK_v,V_{f,\chi}\otimes\mathbf{B}_{\rm cris})\right) &\textrm{if $p\divides v$.}
\end{array}
\right.
\]
We summarize the running hypotheses in this section.
\begin{hypH}\noindent
\begin{itemize}
\item[(a)]$p\ndivides 2(2r-1)!N\varphi(N)$;
\item[(b)] \eqref{Heeg} and \eqref{can} in \secref{sec:def-Heegner};
\item[(c)] $(\ctame,N)=1$;
\item[(d)]$p\cO_K=\frakp\frakpbar$ is split in $K$.
\end{itemize}
\end{hypH}
Let $\ep(V_{f,\chi})=\pm 1$ be the sign in the functional equation for $L(f,\chi,s)$. To calculate the sign, we note that $\ep(V_{f,\chi})=\prod_v\e(\frac{1}{2},\pi_{K_v}\ot\chi_v)$ is a product of local root numbers over places $v$ of $\Q$. By the formulae \cite[(9),\,(12)]{schmidt_localnewforms}, we see that $\e(\onehalf,\pi_{K_v}\ot\chi_v)=1$ for all finite place $v$ under the hypothesis \eqref{Heeg}. On the other hand, since $\pi_\infty$ is the unitary discrete series of weight $2r-1$, we have
\[\e(\onehalf,\pi_{K_\infty}\ot\chi_\infty)=\e(\onehalf,\mu^{r-\onehalf+j})\e(\onehalf,\mu^{\onehalf-r+j})=(\sqrt{-1})^{\abs{2r-1+2j}+\abs{1-2r+2j}},\]
where $\mu:\C^\x\to\C^\x$ is the character $z\mapsto z/\ol{z}$ (\cite[(3.2.5)]{tate77Corvallis}). Therefore, we find that
\beq\label{E:sign}\ep(V_{f,\chi})=(\sqrt{-1})^{\abs{2r-1+2j}+\abs{1-2r+2j}}=-1\iff -r<j<r. \eeq

%\beq\label{E:Hyp1}N, D_K, \text{ and $\ctame$ are pairwise coprime to each other}; \eeq
%\beq\label{E:Hyp2}p\ndivides 2(2r-1)!ND_K\ctame\text{ and $p=\frakp\frakpbar$ splits in $K$}.\eeq

%\begin{conj}[Bloch--Kato]\label{conj:BK}
%${\rm ord}_{s=r}L(f,\chi,s)={\rm dim}_F{\rm Sel}(K,V_{f,\chi})$.
%\end{conj}

\subsection{Nonvanishing of generalized Heegner cycles}

Recall from \subsecref{SS:twisted} the construction of the generalized Heegner classes
$z_{f,\chi}\in H^1(K,T\ot\chi)$
in \eqref{E:zchi.norm}.

\begin{thm}\label{thm:main1} Suppose that $\ep(V_{f,\chi})=-1$. The following two statements hold.
\begin{mylist}
\item If $z_{f,\chi}\neq 0$, then ${\rm Sel}(K,V_{f,\chi})=F\cdot z_{f,\chi}$.
\item The classes $\Hclass_{f,\chi\phi}$ are nonzero in $H^1(K,V_{f,\chi\phi})$
for all but finitely many finite order characters $\phi\colon\Gal(\rcf{p^\infty}/K)\to\mu_{p^\infty}$.
\end{mylist}
\end{thm}

\begin{proof} The first part is a restatement of Theorem~\ref{T:NVHeegner}. The second part follows immediately from Theorem~\ref{thm:BDP}
and the nonvanishing of the \padic $L$-function in Theorem~\ref{T:nonvanishing}.
\end{proof}

\subsection{Vanishing of Selmer groups}\label{subsec:BK}
Assume further that $f$ is ordinary at $p$ in this subsection.
\begin{thm}\label{thm:ETNC}If $L(f,\chi,r)\neq 0$, then
${\rm Sel}(K,V_{f,\chi})=\stt{0}$.
%\item ${\rm Sel}(K,V_{f,\chi\phi})=\stt{0}$ for all but finitely many finite order characters $\phi:\Gal(\rcf{p^\infty}/K) \to\mu_{p^\infty}$.
%\end{mylist}
\end{thm}

\begin{proof}
The nonvanishing of the central value $L(f,\chi,r)$ implies that $\ep(V_{f,\chi})=+1$, and hence $\chi$ has infinity type $(j,-j)$ with $j\geq r$ or $j\leq -r$ by \eqref{E:sign}. Let $\chi^\tau(g):=\chi(\tau g\tau)$, where $\tau$ is the complex conjugation. Then clearly
$L(f,\chi^\tau,r)=L(f,\chi,r)$ and the action of $\tau$ induces an isomorphism $\Sel(K,V_{f,\chi})\iso \Sel(K,V_{f,\chi^\tau})$,
so we may assume that $j\geq r$.
One then immediately checks that $V_{f,\chi}\vert_{G_{K_\frakp}}$  has positive Hodge--Tate
weights\footnote{Here our convention is that $p$-adic cyclotomic character has Hodge--Tate weight $+1$.}, %(in fact, $\stt{r+j,1-r+j}$),
while the Hodge--Tate weights of $V_{f,\chi}\vert_{G_{K_{\frakpbar}}}$ are all $\leq 0$. By \cite[Thm.~4.1(ii)]{BK} we thus have
\begin{equation}\label{def:Hif}
H^1_\BK(\cK_v,V_{f,\chi})=\left\{
\begin{array}{ll}
H^1(\cK_v,V_{f,\chi})&\textrm{if $v=\frakp$},\\
\{0\}&\textrm{if $v={\ol{\pp}}$}.
\end{array}
\right.
\end{equation}
%in other words, we have
%\[
%{\rm Sel}(K,V_{f,\chi})=H^1_{\emptyset,0}(K,V_{f,\chi}).
%\]
Let $\bfz_f^\chi\in H^1(K,T\otimes\chi)$ be the $\chi$-specialization of the Iwasawa cohomology class
$\bfz_f$ defined in \eqref{E:corbfz.E}. By Corollary~\ref{cor:ERL}, the nonvanishing of $L(f,\chi,r)$ implies that
${\rm loc}_\frakp(\mathbf{z}_{f}^{\chi^{-1}})\neq 0$ (note that the factor $e'_\frakp(f,\chi)$ never vanishes).
The result thus follows from \thmref{T:finitenessSelmer}.
\end{proof}

Combined with the nonvanishing of the \padic $L$-function in Theorem~\ref{T:nonvanishing},
the results of \thmref{thm:main1} and \thmref{thm:ETNC} allow us to immediately obtain the following analogue of the
\emph{growth number conjecture} in \cite{mazur-icm83} on the asymptotic behavior of the ranks of Selmer groups over ring class fields.
\begin{thm}\label{T:Mazur}There exists a non-negative integer $e$ such that the formula
\[\dim_F\Sel(\rcf{p^n},V_{f,\chi})=(1-\ep(V_{f,\chi}))\cdot [\rcf{p^n}:K]+e\]
holds for all sufficiently large $n$.
\end{thm}
\subsection{The parity conjecture}
In combination with Nekov{\'a}{\v{r}}'s results on the parity of a \padic family of Galois representations \cite{nekovar-parity3}, our results imply the following parity conjecture for $V_{f,\chi}$. We heartily thank Ben Howard for drawing this application to our attention.
\begin{thm}Suppose that $f$ is ordinary at $p$. Then we have
\[
{\rm ord}_{s=r}L(f,\chi,s)\equiv{\rm dim}_F{\rm Sel}(K,V_{f,\chi})\quad({\rm mod}\;2).
\]
\end{thm}

\begin{proof}
Let $K^-_\infty/K$ be the anticyclotomic $\Zp$-extension and let $\Gamma_K^-:=\Gal(K^-_\infty/K)$.
%Let $\chi:G_K\to\cO^\x$ be a locally algebraic anticyclotomic character of infinity type $(j,-j)$ and of conductor $c\cO_K$ with $(c,N)=1$.
Let $\Lambda:=\cO\powerseries{\Gamma_K^-}$ and let $\cX:G_K\to \Lambda^\times$ be the universal deformation of $\chi$ defined by $g\mapsto \chi(g)g|_{K_\infty^-}$. Recall that $\tau\in G_\Q\smallsetminus G_K$ is the complex conjugation. Let $\Ind_K^\Q\cX:=\Lam e_\frakp\oplus \Lam e_\frakpbar$ be the $G_\Q$-module defined by
\begin{align*}g(ae_\frakp+be_\frakpbar)&=\cX(g)ae_\frakp+\cX^\tau(g)be_\frakpbar\,\text{ for }g\in G_K,\\
\tau(ae_\frakp+be_\frakpbar)&=be_\frakp+ae_\frakpbar.\end{align*}
Let $\cT:=T\otimes_{\cO}\Ind_K^\Q\cX$, which is a self-dual left $\Lam[G_\Q]$-module equipped with a skew-symmetric paring defined in \cite[Example~(5.3.4)]{nekovar-parity3}, and define the $\Lam[G_{\Qp}]$-submodule $\cT_p^+\subset\cT$ by
\[\cT_p^+:=\begin{cases}\sF^+ T\ot\Lam&\text{ if }-r<j<r,\\
T\ot\Lam e_\frakp&\text{ if }j\geq r\text{ or }j\leq -r.
\end{cases}
\]
Then $(\cT,\cT_p^+)$ satisfies \cite[(5.1.2)~(1)--(4)]{nekovar-parity3}. Moreover, one verifies that for any finite order character $\phi:\Gamma_K^-\to\mu_{p^\infty}$, the specialization $\cT_\phi=T\ot\Ind_K^\Q\chi\phi$ together with the corresponding subspace $\cT^+_{p,\phi}$ also satisfy conditions (5)--(8) in \emph{loc.cit.}\footnote{As explained in \cite[Example~(5.3.4)(5)]{nekovar-parity3}, %(see also \cite{nekovar-parity-erratum}),
this follows from properties [\emph{loc.cit.},(2)-(3)] for $\cT_\phi$, whose verification is immediate. Indeed,
$(\cT_\phi,\cT^+_{p,\phi})$ satisfies the Panchishkin condition of \cite[Def.~(3.3.1)]{nekovar-parity3} by construction,
and $\cT_\phi$ is pure of weight $1$ at all finite places, since Ramanujan's conjecture holds for $f$; and 
anticyclotomic Hecke characters are pure of weight $0$.}.

Let $F(\phi)$ be the field generated over $F$ by the values of $\phi$, let $\cO(\phi)$
be the ring of integers of $F(\phi)$, and put $\cV_\phi:=\cT_\phi\ot_{\cO(\phi)} F(\phi)$.
Let $\ep(\cV_\phi)\in\stt{\pm 1}$ be the sign of the Weil--Deligne representation attached to $\cV_\phi$. Under Hypothesis~(H), it is well-known that $\ep(\cV_\phi)=\ep(V_{f,\chi})$ is independent of $\phi$,
and as already noted we have
\[
\ep(V_{f,\chi})=
\begin{cases}-1&\text{ if }-r<j<r,\\
+1&\text{ if }j\geq r\text{ or }j\leq -r.
\end{cases}
\]
Now choose a Hecke character $\psi$ of infinity type $(r,-r)$ and conductor $c_o\cO_K$ such that $\chi\wh\psi^{-1}$ is of $p$-power conductor.
By Theorem~\ref{T:nonvanishing}, we can choose $\phi$ sufficiently wildly ramified such that
\begin{equation}\label{eq:nonzero}
\sL_{\frakp,\psi}(f)(\chi\wh\psi^{-1}\phi)\not =0.
\end{equation}
Thus Proposition~\ref{P:interpolationformula} and Theorem~\ref{thm:ETNC} imply that $\dim_{F(\phi)}\Sel(K,V_{f,\chi\phi})=0$ if $\ep(V_{f,\chi})=+1$,
while Theorems~\ref{thm:walds} and \ref{thm:main1} imply that $\dim_{F(\phi)}\Sel(K,V_{f,\chi\phi})=1$ if $\ep(V_{f,\chi})=-1$.
On the other hand, by Shapiro's lemma we can verify that
\[
\Sel(K,V_{f,\chi\phi})\simeq\Sel(\Q,\cV_\phi).
\]
Therefore, by \cite[Cor.~(5.3.2)]{nekovar-parity3} (see also \cite{nekovar-parity-erratum}),
we conclude that
\[\dim_F\Sel(K,V_{f,\chi})\con \dim_{F(\phi)}\Sel(K,V_{f,\chi\phi})\con \ep(V_{f,\chi})\pmod{2},\]
and the parity conjecture for $V_{f,\chi}$ follows.
\end{proof}

%We can write $\cT=T_\Lam e_\frakp\oplus T_\Lam e_\frakpbar$($T_\Lam:=T\ot\Lam$) such that
%\begin{align*}\rho_\cT(g)(x\ot e_\frakp+y\ot e_\frakp)=&\rho_f(g)x \ot \Psi(g)e_\frakp+\rho_f(g)y\ot \Psi(g^\tau)e_\frakpbar,\\
%\rho_\cT(g\tau )(x\ot e_\frakp+y\ot e_\frakp)=&\rho_f(g)x \ot \Psi(g^\tau)e_\frakpbar+\rho_f(g)y\ot\Psi(g) e_\frakp\end{align*}
%for $g\in G_K$ and $\tau$ the complex conjugation.  Nekovar's theorem is applied to a triple $(\cT,\pairing_\cT,\cT^+)$

%Let $\pairing_T: T\x T\to\cO(1)$ be the canonical paring on $\cT$, which extends to $T_\Lam\x T_\Lam\to \Lam(1)$ by linearity .

%Define the skew-symmetric pairing $\pairing_{\cT}$ by \[\pair{ae_\frakp+be_\frakpbar}{ce_\frakp+de_\frakpbar}_\cT=\pair{a}{d}_T-\pair{b}{c}_T.\]
%Then this pairing $\pairing_{\cT}:\cT\x\cT\to\Lam(1)$ is Galois equivariant ($\Psi(g^\tau)=\Psi(g)^{-1}$).

%We define $\cT^+:=T_\Lam e_\frakp$ and $\cT^-=T_\Lam e_\frakpbar$.

%\begin{appendix}
%!TEX root = HeegnerCycles.tex
%Last revised on May25, ML
	
\section{Kolyvagin's method for generalized Heegner cycles}\label{S:EulerSystem}

\def\Sg{\Sigma}
\def\cH{H}
\def\fin{f}
\def\sing{s}
\def\cY{Y}
\def\pd{\partial}
\def\Chi{\chi}
\def\Chibar{\chi^{-1}}
\def\cS{S}
\def\Bone{B_2}
\def\Btwo{B_1}
\def\Bthree{B_3}
\def\Bfour{B_4}
\def\Esys{\bfc}

We keep the setup and Hypothesis (H) introduced in $\S\ref{sec:arith-applic}$, except that we do \emph{not} assume that $p$ is split in $K$. In particular,
$f\in S_{2r}^{\rm new}(\Gamma_0(N))$ is a newform of level $N$ prime to $p$, and $\chi:\Gal(\rcf{\ctame p^\infty}/K)\to\cO^\x$
is a locally algebraic anticyclotomic Galois character of infinity type $(j,-j)$ and conductor $c\cO_K$. Write $c=\ctame p^s$ with $p\ndivides \ctame$. The aim of this section it to develop a suitable extension of Kolyvagin's method of Euler systems
for the Galois representation $V\otimes\chi$. We largely follow Nekov{\'a}{\v{r}}'s approach \cite{nekovar-invmath}.

\subsection{Notation}

For each positive integer $n$, let $\rcf{n}$ be the ring class field of $K$ of conductor $n$, and let $\Delta_n=\Gal(\rcf{n}/K)$. If $A$ is a $G_{\rcf{n}}$-module unramified outside $pN\ctame $, let $H^1(\rcf{n},A)$ denote the cohomology group $H^1({\rm Gal}(K^{\Sigma_n}/\rcf{n}),A)$,
where $K^{\Sigma_n}$ is the maximal extension of $K$ unramified outside the prime factors of $pN\ctame n$.

Recall that $T$ is the $G_\Q$-stable $\cO$-lattice of the self-dual Galois representation $V$ as in
\cite[\S{3}]{nekovar-invmath}. %, and that the Heegner classes $\Hclass_{f,n}$ defined by $(\ref{E:defnGHC})$ belong to $H^1(\rcf{n},T)$.
By \cite[Prop.\,3.1(2)]{nekovar-invmath}, there is a $G_\Q$-equivariant $\cO$-linear perfect pairing
\beq\label{E:pairing.E}\pairing:T\xx T \longto\cO(1),\eeq
and for any local field $L$, let $\pairing_L:H^1(L,T)\xx H^1(L,T)\to \cO$ denote the local Tate pairing induced by $\pairing$.
Let $\uf$ be a uniformizer of $\cO$ and let $\bbF=\cO/(\uf)$ be the residue field.
If $M$ is a positive integer, we abbreviate
\[
\cO_M:=\cO/\uf^M\cO,\quad\quad T_M:=T/\uf^M\cO.
\]
We let $\ell$ always denote a rational prime \emph{inert} in $K$, and let $\lam$ be the prime of $\cO_K$ above $\ell$,
$K_\lam$ be the completion of $K$ at $\lam$, and $\Frob_\ell$ be the Frobenius element of $\lam$ in $G_K$.
If $A$ is a discrete $\cO[G_K]$-module, we denote by $A^\vee$ the \pont dual of $A$.
Let $H^1_{\fin}(K_\lam,A)$ and $H^1_{\sing}(K_\lam,A):=H^1(K_\lam,A)/H_{\fin}(K_\lam,A)$ be the finite part and the
singular quotient of $H^1(K_\lam,A)$, respectively. Denote by $\loc_\ell:H^1(K,A)\to H^1(K_\lam,A)$ the localization map at $\ell$ and by
\[
\pd_\ell\colon H^1(K,A)\longto H^1_{\sing}(K_\lam,A)
\]
the composition of $\loc_\ell$ with the quotient map $H^1(K_\lam,A)\to H^1_{\sing}(K_\lam,A)$.

\subsection{Kolyvagin's anticyclotomic Euler systems}\label{SS:Kolyvagin}

Denote by $\sK$ the set of square-free products of primes $\ell$ inert in $K$ with $\ell\ndivides 2pN\ctame$.
Let $\tau$ denote the complex conjugation, and let $w_f\in\stt{\pm 1}$ be the Atkin--Lehner eigenvalue of $f$.
\begin{defn}\label{D:EulerSystem.E}
An anticyclotomic Euler system attached to $(T,\Chi)$ is a collection $\stt{\Esys_n}_{n\in \sK}$
of classes $\Esys_n\in H^1(\rcf{c n},T\ot\Chibar)$ such that for every $n=m\ell\in\sK$ we have:
\begin{mylist}
\item[(E1)] $\cores_{\rcf{nc},\rcf{mc}}(\Esys_{n})=\bfa_\ell(f)\cdot \Esys_{m}$;
\item[(E2)] $\loc_\ell(\Esys_n)=\res_{\rcf{mc,\lam},\rcf{nc,\lam}}(\loc_\ell(\Esys_{m})^{\Frob_\ell})$; %where $\Frob_\ell$ is the Frobenius at $\ell$;
\item[(E3)] if $\chi^2=1$, then $\Esys_n^\tau=w_f\cdot \Chi(\sg_{\ol{\frakN}})\cdot \Esys_n^{\sg_{\ol{\frakN}}}$.
\end{mylist}
\end{defn}
We briefly recall the construction of derivative classes attached to an anticyclotomic Euler system $\Esys=\stt{\Esys_n}_{n\in \sK}$.
First we make an auxiliary choice of a positive integer $\nu$ such that $p^{\nu}$ annihilates:
\begin{itemize}\item[(i)] the kernel and cokernel of the map $\res_{K,\rcf{n}}:H^1(K,\Epr\ot\Chibar)\to H^1(\rcf{n},\Epr\ot\Chibar)^{\Delta_n}$ for all positive integers $n$ and $M$;
\item[(ii)] the local cohomology groups $H^1(K_v,\Epr\ot\Chibar)$ for all $v\divides \ctame N$. \end{itemize}The existence of such $\nu$ follows from \cite[Prop.\,6.3,\,Cor.\,6.4,\,Lem.\,10.1]{nekovar-invmath}.
Define the constant
\beq\label{E:B3}
\Btwo=\min\stt{\Ord_\uf(x-1)\mid x\cdot I_2\in \rho_{f}^*\ot\Chibar(G_K),\; x\in\Zp^\x }.
\eeq
A rational prime $\ell$ is called an \emph{$M$-admissible Kolyvagin prime} if
\begin{itemize}\item $\ell\ndivides 2cNp$ is inert in $K$;
\item $ \bfa_\ell(f)\con \ell+1\con 0\pmod{\uf^M}$;
\item $\uf^{M+\Btwo+1}\ndivides \ell+1\pm \bfa_\ell(f)\ell^{1-r}$.\end{itemize}
Let $\sK_M$ be the set of square-free products of $M$-admissible primes, and
 for each $n\in\sK_M$ let $G_n$ denote the Galois group $\Gal(\rcf{nc}/\rcf{c})\subset \Delta_{cn}$.
 For each $\ell\divides n$, the group $G_\ell$ is cyclic of order $\ell+1$,
and we have a canonical decomposition $G_n=\prod_{\ell \divides n}G_\ell$.
Fixing a generator $\sg_\ell$ for each $G_\ell$, Kolyvagin's derivative operators are defined by
\[
D_\ell=\sum_{i=1}^{\ell}i\sg_\ell^i\in\Z[G_\ell]
\]
and
\[
D_n:=\prod_{\ell\divides n} D_\ell\in \Z[G_n]\subset \cO[\Delta_{nc}].
\]
Then for each $n\in\sK_M$ there is a unique $\cD_M(n)\in H^1(\rcf{c},\Epr\ot\Chibar)$ such that
\[\res_{\rcf{c},\rcf{nc}}(\cD_M(n))=p^{3\nu} D_n\Esys_n,\]
and the \emph{derivative class} $\Hg{n}{\Chibar}$ is defined by
\[\Hg{n}{\Chibar}:=\cores_{\rcf{c}/K}(\cD_M(n))\in H^1(K,\Epr\ot\Chibar).\]
%The collection $\stt{\Hg{n}{\Chibar}}_{n\in\sK_M}$ for all $M$ is called the derivative classes attached to $\Esys$.

We next introduce Euler systems with local conditions at $p$. Let $\cF\subset H^1(K\ot\Qp,V\ot\Chibar)$ be an $F$-vector subspace and let $\cF^*\subset H^1(K\ot \Qp,V\ot\Chi)$ be the orthogonal complement of $\cF$ under the local Tate pairing. We assume that $\cF^*=\cF$ if $\chi^2=1$.
Let $\cF_T\subset H^1(K\ot\Qp,T\ot\Chibar)$ be the inverse image of $\cF$ under the natural map $H^1(K\ot\Qp,T\ot\Chibar)\to H^1(K\ot\Q_p,V\ot\Chibar)$
and let $\cF_M\subset H^1(K\ot\Qp,T_M\ot\Chibar)$ be the image of $\cF_T$ under the reduction map $H^1(K\ot\Qp,T\ot\Chibar)\to H^1(K\ot\Qp, T_M\ot\Chibar)$.
 For each positive integer $n$, let $\Sel_\cF^{(n)}(K,\Epr\ot\Chibar)$ be the $n$-imprimitive Selmer group defined by
\begin{align*}
\Sel_\cF^{(n)}(K,\Epr\ot\Chibar)
&:=\stt{s\in H^1(K,\Epr\ot\Chibar)\mid
\begin{array}{ll}
\loc_v(s) \in H^1_{\fin}(K_v,\Epr\ot\Chibar)&\text{ for }v\nmid pn,\\
\loc_p(s)\in \cF_M&\text{ if }p\nmid n
\end{array}
}.
\end{align*}
%such that the localization $\loc_v(s)$ lies in the subgroup $H^1_{\fin}(K_v,\Epr\ot\Chibar)$, for each place $v\nmid pn$ of $K$, and $\loc_p(s)\in \cF_M$ if $p\nmid n$.
Note that if $p\divides n$, then $\Sel_\cF^{(n)}(K,\Epr\ot\Chibar)$ does not depend on the choice of $\cF$.
When $n=1$ we shall simply write $\Sel_\cF(K,\Epr\ot\Chibar)$ for $\Sel_\cF^{(n)}(K,\Epr\ot\Chibar)$.
We let
\[
\Sel_\cF(K,V/T\ot\Chibar):=\dirlim_M \Sel_\cF(K,\Epr\ot\Chibar)
\]
and define $\Sel_{\cF^*}(K,\Epr\ot\Chi)$ in a similar way.
%By definition, for $n\in\sK_M$, we have $\loc_v(\Hg{n}{\Chibar})=0$ for $v\divides \ctame N$ and $\loc_v(\Hg{n}{\Chibar})=0$ in $H^1(K^{ur},T_M)$.

Let
\[\Esys_K:=\cores_{\rcf{c}/K}(\Esys_1)\in H^1(K,T\ot\Chibar).\]
By \cite[Prop.\,10.2 (2)(3)]{nekovar-invmath}, the derivative classes
$\Hg{n}{\Chibar}$ satisfy
\beqcd{K1}
\Hg{n}{\Chibar}\in\Sel_\cF^{(np)}(K,\Epr\ot\Chibar),
\eeqcd
and by definition we see that
\[
\Hg{1}{\Chibar}=p^{3\nu}\Esys_K\pmod{\uf^M}.
\]

If $\ell$ is an $M$-admissible prime, then $G_{K_\lam}$ acts trivially on $\Epr\ot\Chibar$, and there are isomorphisms
\begin{align*}\al_\ell\colon&H^1_{\fin}(K_\lam,\Epr\ot\Chibar)=H^1(K_\lam^{\rm ur}/K_\lam,\Epr)\overset{\sim}\longto\Epr,
%\quad \al_\ell(s)=s(\Frob_\lam),
\\
\beta_\ell\colon&H^1_{\sing}(K_\lam,\Epr\ot\Chibar)=H^1(K_\lam^{\rm ur},\Epr)\overset{\sim}\longto\Epr,
%\quad \beta_\ell(s)=s(\gamma_\ell),
\end{align*}
given by evaluation of cocycles at ${\rm Frob}_\ell$ and $\gamma_\ell$, respectively,
where $\gamma_\ell$ is a generator of the pro-$p$ part of the tame inertia group of $K_\lam$. Define the \emph{finite-to-singular} map
\[\vp_\ell:=\beta_\ell^{-1}\circ\al_\ell: H^1_{\fin}(K_\lam,\Epr\ot\Chibar)\overset{\sim}\longto H^1_{\sing}(K_\lam,\Epr\ot\Chibar).\]
Then it is proved in \cite[Prop. 10.2]{nekovar-invmath} that for every $M$-admissible prime $\ell\divides n$, we have the relations
\beqcd{K2}\left(\frac{(-1)^{r-1}\ep_n \bfa_\ell(f)\ell^{1-r}}{\uf^M}-\frac{\ell+1}{\uf^M}\right)\vp_\ell(\loc_\ell(\Hg{n/\ell}{\Chibar}))=\left(\frac{\ell+1}{\uf^M}\ep_n-\frac{\bfa_\ell(f)\ell^{1-r}}{\uf^M}\right)\pd_{\ell}(\Hg{n}{\Chibar});\eeqcd
\beqcd{K3}\Hg{n}{\Chibar}^\tau=\ep_n\cdot\Hg{n}{\Chibar}\text{ if }\chi^2=1,\eeqcd
where $\ep_n=\chi(\sg_\frakN)\cdot w_f\cdot (-1)^{\omega(n)}\in\stt{\pm 1}$ with $\omega(n)$ 
the number of prime divisors of $n$.

\begin{defn}\label{D:KolyvaginSystem}
Let ${\rm ES}(T,\Chi,\cF)$ be the space of anticyclotomic Euler systems with local condition $\cF$,
consisting of anticyclotomic Euler systems $\Esys=\stt{\Esys_n}_{n\in\sK}$ satisfying,
in addition to (E1--3) in \defref{D:EulerSystem.E}, the conditions:
\begin{itemize}
\item[(E4)]  $\Esys_K\in \Sel_\cF(K,T\ot\Chibar)$ and $\Esys_K^\tau\in\Sel_{\cF^*}(K,T\ot\Chi)$
($\Leftrightarrow \loc_p(\Esys_K)\in\cF_T$ and $\loc_p(\Esys_K^\tau)\in\cF^*_T$);
\item[(E5)] for every $M$ and $n\in\sK_M$, we have $\Hg{n}{\Chibar}\in\Sel_\cF^{(n)}(K,\Epr\ot\Chibar)$
($\Leftrightarrow\loc_p(\Hg{n}{\Chibar})\in \cF_{M}$).
\end{itemize}
\end{defn}
The following is one of the key technical results in this paper.

\begin{thm}\label{T:Descent.E}
If  $\Esys\in {\rm ES}(T,\Chi,\cF)$ is an Euler system with local condition $\cF$ with \[\Esys_K\not =0\text{ in }H^1(K,V\ot\Chibar), \]
then $\Sel_{\cF^*}(K,V\ot\Chi)=F\cdot \Esys_K^\tau$.
\end{thm}

In the next two sections we shall give the applications of this result to the Euler system constructed in this paper,
postponing the proof of Theorem~\ref{T:Descent.E} to $\S\ref{SS:Descent.E}$.

% !TEX root = HeegnerCycles.tex
%ML, May 20

%\newcommand{\euler}[2]{\bfz^{#2}_{#1}}
\def\cH{\mathcal H}
\def\Chi{\chi}
\def\rmd{\,{\rm d}}
\def\cK{\mathcal K}

\subsection{Euler system for generalized Heegner cycles (I)}
Suppose $p=\frakp\ol{\frakp}$ splits in $K$, and for this section assume that $\chi$ has infinity type $(j,-j)$ with
\[-r<j<r.\]
We consider the $\Chibar$-component $z_{f,\chi^{-1},n}$ of the generalized Heegner classes $z_{f,n}$, as defined in (\ref{E:defn.norm}).
%\subsecref{SS:GHCI} and \subsecref{SS:twisted}.
\begin{prop}\label{P:EulerS}If $n=m \ell$ with $\ell$ inert in $K$ and $c\divides m$, then:
\begin{mylist}
\item $\cores_{\rcf{n},\rcf{m}}(\Hclass_{f,\Chibar,n})=\bfa_\ell(f)\cdot \Hclass_{f,\Chibar,m}$.
\item $\loc_\ell(\Hclass_{f,\Chibar,n})=\res_{\rcf{m,\lam},\rcf{n,\lam}}(\loc_\ell((\Hclass_{f,\Chibar,m})^{\Frob_\ell})$.
\item $(\Hclass_{f,\Chibar,n})^\tau= w_f\cdot \Chi(\sg_{\ol{\frakN}})(\Hclass_{f,\Chi,n})^{\sg_{\ol{\frakN}} }$.
\end{mylist}
\end{prop}
\begin{proof}
These properties follow from  \propref{prop:norm}, \lmref{L:ES2.norm}, and \lmref{L:ES3.norm}, respectively.
\end{proof}

\begin{lm}\label{L:0.f}
Suppose $p>2r-1$. Let $w$ be a place of $\rcf{c}$ above $p$, and let $\rcf{c,w}$ be the completion of $\rcf{c}$ at $w$.
If $L/\rcf{c,w}$ is a finite extension and $L'/L$ is a finite unramified extension, then the corestriction map
\[\cores_{L'/L}: H_f^1(L',T_M\ot\chi)\longto H_f^1(L,T_M\ot\chi)\]
is surjective, and the restriction map
\[\res_{L,L'}:H^1(L,T_M\ot\Chibar)/H^1_\BK(L,T_M\ot\Chibar)\longto H^1(L',T_M\ot\Chibar)/H^1_\BK(L',T_M\ot\Chibar)\]is injective.
\end{lm}
\begin{proof}
By local Tate duality, it suffices to establish the first claim.
Since $p>2r-1$, the Bloch--Kato group $H^1_f(L,T\ot\chi)$ for the crystalline representation $T\ot\chi$ admits a description
in terms of Fontaine--Laffaille modules (see \cite[Lem.~4.5(c)]{BK}).
Thus let $D$ be the Fontaine--Laffaille $\cO_{L}$-module attached to $T\ot\chi$ as a $G_{L}$-module.
Then $D\ot_{\cO_{L}}\cO_{L'}$ is the Fontaine--Laffaille module of $T\ot\chi$ regarded as a $G_{L'}$-module, and by \emph{loc.cit.}
we have the commutative diagram
\[\xymatrix{ D^0\ot_{\cO_L}\cO_{L'}\ar[d]^{1\ot\Tr_{L'/L}}\ar[r]^{f_0-1}&
D^0\ot_{\cO_L}\cO_{L'}\ar[r]\ar[d]^{1\ot\Tr_{L'/L}}& H^1_f(L
,T\ot\chi)\ar[d]^{\cores_{L'/L}}\ar[r]&0\\
D^0\ar[r]^{f_0-1}& D^0\ar[r] &H^1_f(L,T\ot\chi)\ar[r]&0,
}\]
where $f_0$ is the usual Frobenius map.
The surjectivity of $\cores_{L'/L}$ thus follows from the surjectivity of the trace map $\Tr_{L'/L}:\cO_{L'}\to \cO_L$.
\end{proof}

For each $n\in\sK$ define
\[
\bfc_n^\heeg:=\Hclass_{f,\Chibar,nc}.
\]
Set $\bfc^\heeg:=\stt{\bfc_n^\heeg}_{n\in\sK}$ and let $\cF_{\rm BK}:=H^1_f(K\ot\Qp,V\ot\Chibar)$
be given by the usual Bloch--Kato finite subspaces.

\begin{prop}\label{prop:Heeg-ES}
We have $\bfc^\heeg_K=\Hclass_{f,\Chibar}$,
and $\bfc^\heeg\in {\rm ES}(T,\chi^{-1},\cF_{\rm BK})$
is an Euler system with local condition $\cF_{\rm BK}$.
\end{prop}

\begin{proof}
The first claim is clear. On the other hand,
it follows from \propref{P:EulerS} that $\bfc^\heeg$ satisfies conditions (E1--3) in \defref{D:EulerSystem.E}.
To see that $\bfc^\heeg$ also satisfies conditions (E4) and (E5) in \defref{D:KolyvaginSystem}, we note that $\loc_p(\Hclass_{f,\Chibar,nc})\in
H^1_\BK(\rcf{nc},T\ot\Chibar)$ by \cite{niziol1997}. Since the action of complex conjugation induces an isomorphism $H^1_f(K\ot\Q_q,T\ot\Chibar)\iso H^1_f(K\ot\Q_q,T\ot\Chi)$ for every prime $q$, we see that $(\bfc^\heeg)^\tau$ satisfies (E4). Therefore, we have
$\bfc^\heeg_K=\Hclass_{f,\chi^{-1}}\in \Sel_{\cF_{\rm BK}}(K,T\ot\Chibar)$ and
\[
\loc_w(\res_{\rcf{c},\rcf{cn}}(\cD_M(n))=\loc_w(p^{3\nu}D_n\Hclass_{f,cn,\Chibar})\in H^1_\BK(\rcf{cn,w},T_M\ot\Chibar)
\]
for each place $w\divides p$. By \lmref{L:0.f}, this implies that
$\loc_w(\cD_M(n))\in H^1_\BK(\rcf{c,w},T_M\ot\Chibar)$,
and hence $\loc_p(\kappa_{\Chibar}(n))\in\cF_M$, as was to be shown.
\end{proof}

\begin{thm}\label{T:NVHeegner}
If $\Hclass_{f,\Chi}\not =0\in H^1(K,V\ot\Chi)$, then
\[
\Sel(K,V\ot\Chi)=\Sel_{\cF_{\rm BK}}(K,V\ot\Chi)=F\cdot \Hclass_{f,\Chi}.
\]
\end{thm}
\begin{proof}
Note that by \propref{P:EulerS}(3) we have the equivalence
\[
\Hclass_{f,\Chi}=0\iff \Hclass_{f,\Chibar}=0.
\]
Thus \propref{prop:Heeg-ES} combined with \thmref{T:Descent.E} yields the result.
\end{proof}

\subsection{Euler system for generalized Heegner cycles (II)}

As in the preceding section, we assume that $p=\frakp\frakpbar$ splits in $K$, but suppose now that
$\chi$ has infinity type $(j,-j)$ with
\[j\geq r.\]
In addition, in this section we assume that $f$ is ordinary at $p$.

Let $\bfz_f^{\Chi}$ be the $\Chi$-specialization of the
Iwasawa cohomology class $\bfz_f$ defined in \eqref{E:corbfz.E}.
For every place $v$ of $K$ above $p$, let $\cL_v\subset H^1(K_v,V\ot \Chi)$
be the subspace spanned by $\loc_{v}(\bfz_f^{\Chi})$. Then
\[
\cL_{T,v}:=\cL_v\cap H^1(K_v,T\ot\Chi)=\cO \uf^{-a_v}\loc_v(\bfz_f^{\Chi})+H^1(K_v,T)_{\rm tor}
\]
for some $a_v\in\Z_{\geq 0}$, where $H^1(-)_{\rm tor}$ denotes the torsion subgroup of $H^1(-)$.
Let $\cL_v^*\subset H^1(K_v,V\ot \Chibar)$ be the orthogonal complement of $\cL_v$,
and set $\cL^*:=\cL_\frakp^*\oplus \cL_{\frakpbar}^*$.
We will choose the integer $\nu$ in \subsecref{SS:Kolyvagin} large enough so that
$p^\nu H^1(K_v,T)_{\rm tor}=\{0\}$ for each $v\divides p$.

Consider the Iwasawa cohomology classes $\bfz_{f,n}:=\bfz_{f,n,\alpha}$ from \eqref{E:bfz.E},
and for each $n\in\sK$ define
\[
\bfc_n^{\heeg,\dagger}:=\bfz_{f,cn}^{\Chibar}\in H^1(\rcf{cn},T\ot\Chibar)
\]
to be the specialization of $\bfz_{f,cn}$ at $\chi^{-1}$. 
Set $\bfc^{\heeg,\dagger}:=\stt{\bfc_n^{\heeg,\dagger}}_{n\in\sK}$.

\begin{prop}\label{P:HeegES2}
The collection $\bfc^{\heeg,\dagger}\in{\rm ES}(T,\chi,\cL^*)$ is an Euler system
for the local condition $\cL^*$ with $\bfc^{\heeg,\dagger}_K=\bfz_f^{\Chibar}$.
\end{prop}

\begin{proof}
We begin by noting that for inert primes $\ell$ with $n=m\ell\in\sK$, we have
\begin{mylist}
\item $\cores_{\rcf{nc},\rcf{mc}}(\bfz_{f,n})=\bfa_\ell(f)\cdot \bfz_{f,m}$;
\item $\loc_\ell(\bfz_{f,n})=\res_{\rcf{mc,\lam},\rcf{nc,\lam}}(\loc_\ell(\bfz_{f,m})^{\Frob_\ell})$;
\item $\bfz_{f,n}^\tau=w_f\cdot \sg_{\ol{\frakN}}\cdot \bfz_{f,n}^{\sg_{\ol{\frakN}}}$,
\end{mylist}
since by Lemma~\ref{lem:finite-tw} and \propref{P:EulerS} these relations hold after specialization at every finite order ramified character.
Specializing the same relations to $\Chibar$, we thus find that conditions (E1--3) are satisfied by $\bfc^{\heeg,\dagger}$. The validity of (E4) for $\bfc^{\heeg,\dagger}$ and its image under $\tau$ follows from the fact that if $v$ and $\bar v$ are the two places of $K$ above $p$, then $\loc_v(\bfc^{\heeg,\dagger}_K)=\loc_v(\bfz_f^{\Chibar})$
belongs to $H^1(K_v,\sF^+T\ot\Chibar)$ and the action of complex conjugation sends $H^1(K_v,\sF^+T\ot\Chibar)$ to $H^1(K_{\bar v},\sF^+T\ot\Chi)$.
We now proceed to verify condition (E5) for $\bfc^{\heeg,\dagger}$.
For any finite extension $L/K_v$, let
\[
\pairing_{L}:H^1(L,T_M\ot\Chibar)\x H^1(L,T_M\ot\Chi)\longrightarrow \cO/\uf^M\cO
\]
be the canonical pairing. By \cite[Prop.~1.4.3]{Rubin-ES},
it suffices to show that $\pair{\loc_v(\Hg{n}{\Chibar})}{\cL_{T,v}}_{K_v}=0$, \ie
\beq\label{E:3.Euler}\pair{\loc_v(\Hg{n}{\Chibar})}{ \uf^{-a_v}\loc_v(\bfz_f^{\Chi})+x}_{K_v}\con 0\pmod{\uf^M},
\;\text{ for all }x\in H^1(K_v,T)_{\rm tor}.\eeq

Let $v$ be a place of $K$ above $p$ and let $w/w_0$ be places of $\rcf{nc}/\rcf{c}$ above $v$.
Let $\cK$ and $\cN$ be the completion of $\rcf{c}$ and $\rcf{nc}$ at $w_0$ and $w$, respectively,
and note that $\cN/\cK$ is an unramified extension. Set
\[
\cK_\infty:=K_\infty\cK,\quad\quad \cN_\infty:=K_\infty\cN.
\]
Let $\Psi_{v}$ be a set of representatives of $\Delta_c/\Delta_{c,w_0}$, where $\Delta_{c,w_0}:=\Gal(\cK/K_v)$ is the decomposition group of $v$,
and let $\Delta_c={\rm Gal}(K_c/K)$ as always. By \lmref{L:0.f}, there exists $\bfy_{\cN,\sg}\in H^1_{\Iw}(\cN_\infty,T_M)$ such that
 \[\cores_{\cN/\cK}(\bfy_{\cN,\sg})\con\loc_{w_0}(\res_{\rcf{\ctame},\rcf{c}}(\uf^{-a_v}\sg\bfz_{f,\ctame}))\pmod{\uf^M}.\]

It is easy to see that $\Hg{n}{\Chibar}$ is divisible by $p^\nu$, and so
$\pair{\Hg{n}{\Chibar}}{H^1(K_v,T)_{\rm tor}}=0$. On the other hand, we compute
\begin{align*}
\pair{\Hg{n}{\Chibar}}{\loc_v(\uf^{-a_v}\bfz_f^\Chi)}_{K_v}
&=\sum_{\sg\in\Delta_{\ctame},\;\rho\in\Psi_v}\pair{\loc_{w_0}(\rho\cD_M(n))}{\loc_{w_0}(\res_{\rcf{\ctame},\rcf{c}}(\uf^{-a_v}\sg\bfz_{f,\ctame}^\chi))}_{\cK}\\
&=\sum_{\sg\in\Delta_{\ctame},\;\rho\in\Psi_v}\pair{\loc_{w_0}(\rho \cD_M(n))}{\cores_{\cN/\cK}(\bfy_{\cN,\sg}^{\Chi})}_\cK\\
&=\sum_{\sg\in\Delta_{\ctame},\;\rho\in\Psi_v}p^{3\nu}\pair{\loc_w(\rho D_n\bfz^{\Chibar}_{f,cn})}{\bfy_{\cN,\sg}^{\Chi}}_{\cN}.
\end{align*}
Thus to verify \eqref{E:3.Euler} it remains to show that $\pair{\loc_w(\rho D_n\bfz^{\Chibar}_{f,cn})}{\bfy_{\cN,\sg}^{\Chi}}_{\cN}
\equiv 0\pmod{\varpi^M}$. Consider Perrin-Riou's $\Lam$-adic local pairing (\cite[3.6.1]{PR115}):
\[\pairing_{\cN_\infty}\colon \HIw^1(\cN_\infty,T_M)\,\x\, \HIw^1(\cN_\infty,T_M)\longrightarrow \Lam_\cO(\Gamma)\ot\cO/\uf^M.\]
Recall that for every $x=\prolim_m x_m$ and $y=\prolim_m y_m$ in $\HIw^1(\cN_\infty,T_M)$, the pairing is defined by
\[\pair{x}{y}_{\cN_\infty}=\prolim_m \sum_{\sigma\in\Gal(\cN_m/\cN)}\pair{x_m}{\sigma y_m}_{\cN_m}\sigma,\]
and it enjoys the interpolation property: if $\chi:\Gamma\to\cO^\x$ is any $p$-adic character, then
\[\pair{x}{y}_{\cN_\infty}(\chi)=\pair{x^\chi}{y^{\Chibar}}_{\cN}.\]
Since for any finite order character $\phi$ of $\Gamma$ and any $\rho\in\Delta_{\ctame}$,
the classes $\bfy_{\cN,\sg}^\phi$ and $\loc_w(\rho \bfz_{f,\ctame}^\phi)$ belong to $H^1_\BK(\cN,T\ot\phi)$,
we see that $\pair{\loc_w(\rho D_n\bfz_{f,cn})}{\bfy_{\cN,\sg}}_{\cN}=0$, and hence
\begin{align*}\pair{\loc_w(\rho D_n\bfz^{\Chibar}_{f,cn})}{\bfy_{\cN,\sg}^{\Chi}}_{\cN}
&=\pair{\loc_w(\rho D_n\bfz_{f,cn})^{\Chibar}}{\bfy_{\cN,\sg}^{\Chi}}_{\cN}\\
&=\pair{\loc_w(\rho D_n\bfz_{f,cn})}{\bfy_{\cN,\sg}}_{\cN}(\Chi)\con 0\pmod{\uf^M}.\end{align*}
This completes the proof.
\end{proof}

\begin{thm}\label{T:finitenessSelmer}
If $\loc_\frakp(\bfz_f^{\Chibar})\not =0$, then ${\rm Sel}(K,V\ot\Chi)=\stt{0}$.
\end{thm}
\begin{proof}
To every choice of subspaces $\cF_v\subset H^1(K_v,V\ot\chi)$ for every prime $v\divides p$,
we associate the generalized Selmer group
\begin{align*}
H^1_{\cF_\frakp,\cF_{\frakpbar}}(K,V\ot\Chi)
&:=\left\{s\in H^1(K,V\ot\Chi)\mid
\begin{array}{lr}
\loc_\frakq(s)\in H^1_f(K_\frakq,V\ot{\Chibar})&\text{ for }\frakq\ndivides p\\
\loc_v(s)\in\cF_v&\text{ for }v\divides p
\end{array}
\right\}.
\end{align*}
The nonvanishing hypothesis implies that $\loc_{\frakpbar}(\bfz_f^{\Chi})\not =0$,
and hence by \propref{P:HeegES2} and  \thmref{T:Descent.E}, we have
\begin{equation}\label{eq:H-L}
H^1_{\cL_\pp,\cL_{\frakpbar}}(K,V\ot\Chi)=F\cdot(\bfz_f^{\Chibar})^\tau=F\cdot\bfz_f^{\Chi}.
\end{equation}
Note that $\loc_{\frakpbar}(\bfz_f^{\Chi})^\tau=\loc_\frakp(\bfz_f^{\Chibar})$.
The nonvanishing of $\loc_\frakp(\bfz_f^{\Chibar})$ thus implies that $\loc_{\frakpbar}(\bfz_f^{\Chi})\not =0$,
and combined with $(\ref{eq:H-L})$ this shows that $H^1_{\cL_\frakp,0}(K,V\ot\Chi)=\{0\}$.
Finally, in light of the Poitou--Tate exact sequence
%\beq\label{E:exact}
\begin{align*}
0\longto H^1_{0,\emptyset}(K,V\ot{\Chibar})\longto &H^1_{\cL_{\frakp}^*,\emptyset}(K,V\ot{\Chibar})\overset{\loc_{\pp}}\longto\cL_\frakp^*\\
&\longto H^1_{\emptyset,0}(K,V\ot\Chi)^\vee\longto H^1_{\cL_{\frakp},0}(K,V\ot\Chi)^\vee\longto 0,
\end{align*}
we find that $H^1_{\emptyset,0}(K,V\ot\chi)=\Sel(K,V\ot\chi)=\stt{0}$.
%\eeq
\end{proof}

\subsection{Kolyvagin's descent: Proof of \thmref{T:Descent.E}}\label{SS:Descent.E}

Let $\Esys\in {\rm ES}(T,\chi,\cF)$ be an Euler system with $\Esys_K\not =0\in H^1(K,V\ot\Chibar)$,
or equivalently, with $\Esys_K\not \in H^1(K,T\ot\Chibar)_{\rm tor}$.

\subsubsection*{\ul{Preliminaries}}
Let $R_\rho= \cO[\rho^*_f(G_\Q)]\subset M_2(\cO)$ and define
\[
\Bone:=\inf\stt{n\in\Z_{\geq 0}\mid \uf^n M_2(\cO)\subset R_\rho }.
\]
Since $\rho_f^*$ is absolutely irreducible, we have $R_\rho\ot F=M_2(F)$, and hence $\Bone<\infty$.
%\begin{Remark}$\Bone=0$ if the residual Galois representation $\bar\rho_f$ is absolutely irreducible.
%\end{Remark}
\begin{lm}\label{L:EE}
Let $E\subset\rcf{\ctame p^\infty}$ be a $p$-ramified extension of $K$.
Then either $E=\Q$ or $\Q(\sqrt{p^*})$, where $p^*=(-1)^\frac{p-1}{2}p$.
\end{lm}
\begin{proof}
Since $p\ndivides D_K$, the fields $E$ and $K$ are linearly disjoint.
It follows that $EK$ is abelian and dihedral over $\Q$. Hence 
by class field theory we conclude that either $E=\Q$ or $\Q(\sqrt{p^*})$.
\end{proof}

Let $M$ be a positive integer. Then $\Chibar\pmod{\uf^M}$ factors through the Galois group $\Gal(\cH/K)$ for some ring class field $\rcf{\ctame p^\infty}/\cH/\rcf{\ctame}$. Let $\cH^\flat$ be the maximal pro-$p$ extension of $\rcf{\ctame}$ inside $\cH$.
Then $\Gal(\cH/\cH^\flat)$ is a cyclic group of order dividing $p\pm 1$. In addition, the ramification index of $\cH^\flat$ above $p$ is a $p$-th power,
and by \lmref{L:EE} is follows that $\Qbar^{\ker\rho_f^*}\cap\cH^\flat=\Q$. Thus we have $\rho_f^*(G_\Q)=\rho_f^*(G_{\cH^\flat})$ and hence
\[
R_\rho=\cO[\rho_f^*(G_{\cH^\flat})].
\]

\begin{lm}\label{L:7}
Let $\xi:G_{\cH^\flat}\to\cO^\x$ be a character. \begin{mylist}\item If $T'\subset \Epr\ot\xi$ is an $R_\rho$-submodule with $T'\not\subset\uf \Epr$, then $\uf^{\Bone}\Epr\subset T'$. \item $\uf^{\Bone}\Hom_{R_\rho}(\Epr\ot\xi,\Epr\ot\xi)=\uf^{\Bone}\cO\cdot I_2$, where $I_2$ is the identity map.
\end{mylist}
\end{lm}
\begin{proof}This is essentially \cite[Lemma\,12.3]{nekovar-invmath}.\end{proof}

\begin{lm}[\cite{nekovar-invmath}, proof of Prop.~12.2(b)]
\label{L:8}
Let $\xi_1,\ldots,\xi_s:G_{\cH^\flat}\to\cO^\x$ be characters, and set
\[
S=\bigoplus_{i=1}^s \cO_M/(\uf^{n_i}),\quad\cV=\bigoplus_{i=1}^s \uf^{M-n_i}\Epr\ot\xi_i.
\]
Let $\cW\subset \cV$ be an $R_\rho$-submodule. If the map $j:S\to \Hom(\cW,\Epr)$ given by
\[
a=(a_1,\ldots, a_s)\longto j(a):(w_1,\ldots,w_s)\longmapsto a_1w_1+\cdots a_sw_s
\]
is injective, then $\uf^{(2^{s+1}-2)\Bone}\cV\subset \cW$.
\end{lm}
\begin{proof}
We proceed by induction on $s$. For $s=1$, the result follows from \lmref{L:7}(1). Suppose $s>1$,
and let $\pi:\cV\to \cV':=\oplus_{i=2}^s \uf^{M-n_i}\Epr\ot\xi_i$ be the map projecting onto the last $s-1$ factors.
Let
\[
S'=\bigoplus_{i=2}^s \cO_M/(\uf^{n_i}),\quad\cW'=\pi(\cW)\subset \cV'.
\]
It is easy to see that $S'\to \Hom(\cW',\Epr)$ is also injective given the injectivity of $j$,
and hence by induction hypothesis we have $\uf^{\gamma}\cV'\subset \cW'$ with $\gamma=(2^s-2)\Bone$.
Let
\[
\cV_1=\uf^{M-n_1}\Epr\ot\xi_1\hookto V,\quad \cW_1=\cW\cap \cV_1=\ker \pi,
\]
and let $\cW'\to \cV_1/\cW_1$ be the $R_\rho$-module map $w'\mapsto pr_1(w)$,
where $w$ is a lifting of $w'$ in $\cW\subset\cV$, and $pr_1:\cV\to \cV_1$ is the first projection.
By \lmref{L:7}(1), there exists $m\leq n_1$ such that
\[
\uf^{m+\Bone}\cV_1\subset \cW_1\subset \uf^m \cV_1.
\]
Let $j':\cV'\to \cV_1/\uf^m \cV_1$ be the composition of $R_\rho$-module maps
\[j':\cV'\stackrel{\cdot\uf^{\gamma}}\longto \cW'\longto \cV_1/\cW_1\longto \cV_1/\uf^m \cV_1=\cO_M/(\uf^m)\]
By \lmref{L:7}(2), there exists $(a_2,\ldots,a_s)\in \cO^{s-1}$ such that
\[\uf^{\Bone}j'(v_2,\ldots v_s)=a_2v_2+\cdots+a_sv_s.\]
In particular, for every $(w_1,\ldots,w_s)\in\cW$, we have
 \[-\uf^{\gamma+\Bone}w_1+\uf^{\gamma}a_2w_2+\cdots+\uf^{\gamma}a_sw_s\in \uf^m \cO_M.\]
This shows that $(-\uf^{n_1-m+\gamma+\Bone},\uf^{n_1-m+\gamma}a_2,\ldots,\uf^{n_1-m+\gamma}a_s)\in S$ annihilates $\cW$.
By the injectivity of $j:S\hookto \Hom(\cW,\Epr)$, the equality $\uf^{n_1-m+\gamma+\Bone}=0\in\cO/(\uf^{n_1})$ implies that $m\leq \gamma+\Bone$.
Thus we have proved the inclusions $\uf^{\gamma+2\Bone}\cV_1\subset \cW_1$ and $\uf^{\gamma}\cV'\subset \cW'$,
and it follows that $\uf^{2\gamma+2\Bone}\cV\subset \cW$, concluding the proof of the lemma.
\end{proof}

Put $\rho_M:=\rho_{f}^*\ot\cO_M:G_\Q\to \Aut_\cO(\Epr)$, let
$\Q(\Epr):=\Qbar^{{\rm ker}\rho_M}$ be the splitting field of $\Epr$,
and set
\[
L=\cH(\Epr):=\cH\cdot \Q(\Epr).
\]
Consider the $\Gal(L/\Q)$-module
$H^1(L,\Epr)=\Hom(\Gal(\Qbar/L),\Epr)$, where $\sg\in \Gal(L/\Q)$ acts via
\[(\sg\cdot f)(s)=\sg f(\sg^{-1} s).\]
If $\cS\subset H^1(L,\Epr)^{\Gal(L/\cH)}$ is a $\cO[\Gal(\cH/\Q)]$-submodule, we let
$L_\cS:=\cap_{s\in \cS} \Qbar^{\ker s}$ be the splitting of $\cS$ over $L$, and put $\cG_\cS:=\Gal(L_\cS/L)$.
We then have an inclusion
\[
\cS\hookto H^1(\cG_\cS,\Epr)^{\Gal(L/\cH)}
\]
and a $\Gal(L/\Q)$-equivariant map
\[\cG_\cS\hookto V_\cS:=\Hom_\cO(\cS,\Epr).\]

\begin{lm}\label{L:nonfull} Let $s=\dim_\bbF S\ot \bbF$. Then $\uf^{(2^{s+1}-2)\Bone}V_\cS\subset \cO[\cG_\cS]$.
\end{lm}
\begin{proof} Since $\Gal(\cH/\cH^\flat)$ has order dividing $p\pm 1$, the $\Gal(\cH/\cH^\flat)$-module $\cS$
can be decomposed into a direct sum of cyclic $\cO$-modules:
\[
S=\bigoplus_{i=1}^s \cO/(\uf^{n_i})\ot\xi_i^{-1}
\]
for some $\xi_i:\Gal(\cH/\cH^\flat)\to \cO^\x$, and so $V_\cS=\oplus_{i=1}^s \uf^{M-n_i}\Epr\ot\xi_i$ as $R_\rho$-modules.
Applying \lmref{L:8} with $\cW:=\cO[\cG_\cS]$, the result follows.
\end{proof}

Let $\cG_\cS^+=\cG_\cS^{\tau=1}=(1+\tau)\cG_\cS$, where $\tau$ is the complex conjugation.
\begin{prop}\label{P:Nek122}
\hfill
\begin{mylist}\item $\uf^{\Btwo}H^1(\Gal(L/K),\Epr\ot\Chibar)=\{0\}$.
\item $L_\cS\cap \cH(T_{2M})\subset \cH(T_{M+{\Btwo}})$.
\item For each $g\in \cG_\cS^+$, there exist infinitely many primes $\ell$ inert in $K$ such that:
\begin{itemize}
\item{} $\Frob_{\ell}(L_{\cS}/K)(:=\Frob_\ell|_{L_\cS})=g$,
\item{} $\uf^M\divides \ell+1\pm \bfa_\ell(f)$
\item{} $\uf^{M+\Btwo+1}\ndivides \ell+1\pm \bfa_\ell(f)$.
\end{itemize}
\end{mylist}
\end{prop}
\begin{proof} This can be proved by the same argument as in \cite[Prop.\,12.2]{nekovar-invmath}.\end{proof}

\subsubsection*{\ul{The descent argument}}
Define the constants $\Bthree, \Bfour$ by
\begin{align*}\Bthree&:=\max\stt{n\in\Z_{\geq 0}\mid \Esys_K\in \uf^n H^1(K,T\ot\Chibar) }\\
&\;=\max\stt{n\in\Z_{\geq 0}\mid \Esys_K^\tau\in \uf^n H^1(K,T\ot\Chi) };\\
\Bfour&:=\left\{
\begin{array}{ll}
0&\text{ if }\chi^2=1,\\
\min_{\sg\in\Gal(\rcf{\ctame p^\infty}/K)}\Ord_\uf(\chi^2(\sg)-1)&\text{ if }\chi^2\not =1.
\end{array}
\right.
\end{align*}
Put $C_1:=6\Bone+\Btwo+\Bthree+\Bfour$, and choose a positive integer $M$ with
\[M>2C_1+2\Btwo.\]
Let $\Hg{1}{\Chi}=\Esys_K^\tau\pmod{\uf^M}$, and for each $x\in \Epr$ put
\[\Ord_\uf(x):=\max\stt{n\in\Z_{\geq 0}\mid x\in \uf^n \Epr}.\]

\begin{lm}\label{L:2}
There is an $M$-admissible prime $\ell_1$ such that
\[
\Ord_\uf(\al_{\ell_1}(\Hg{1}{\Chibar}))=\Ord_\uf(\al_{\ell_1}(\Hg{1}{\Chi}))\leq C_1.
\]
\end{lm}
\begin{proof}
Let $\res_{K,L}:H^1(K,\Epr\ot\Chibar)\to H^1(L,\Epr\ot\Chibar)=H^1(L,\Epr)$ be the restriction map.
Let $s_1=\res_{K,L}(\Hg{1}{\Chibar})\in H^1(L,\Epr)$, and consider the $\cO$-submodule
\[
S:=\cO s_1+\cO s_1^\tau\subset H^1(L,\Epr)^{\Gal(L/\cH)}.
\]
Take an element $t\in \uf^{\Btwo+\Bthree+\Bfour}\Epr$ with $\Ord_\uf(t)=\Btwo+\Bthree+\Bfour$,
and define $f\in V_S$ by $f(s_1)=t$ and $f(s_1^\tau)=0$ if $\Chi\not=\Chibar$.
Using \propref{P:Nek122}(1), it is easy to see that $f$ is well-defined.
Applying \lmref{L:nonfull}, we find that
\[\uf^{6\Bone}(1+\tau)f=\sum_{g\in \cG_S^+}a_g\cdot g,\quad(a_g\in \cO)\]
and evaluating at $s_1$ we obtain
\[\uf^{6\Bone}t=\sum_{g\in \cG_S^+}a_g\cdot \Hg{1}{\Chibar}(g).\]
This shows that there is an element $g\in \cG_S^+$ with $\Ord_\uf(\Hg{1}{\Chibar}(g))\leq C_1$,
and the existence of a prime $\ell_1$ as in the statement follows from \propref{P:Nek122}. \end{proof}

Fix an $M$-admissible prime $\ell_1$ as in Lemma~\ref{L:2}, %and let $\lam_1=\ell_1\OK$.
and let $\cS\subset H^1(L,\Epr)$ be the image
of the sum of $\Sel_\cF^{(\ell_1)}(K,\Epr\ot\Chibar)$ and its complex conjugate.
Then $S\subset\Hom(\cG_\cS,\Epr)^{\Gal(L/\cH)}$ is an $\cO[\Gal(\cH/\Q)]$-submodule.
We will apply the discussion in the preceding paragraphs to this $\cS$.

Setting
\[d_0:=\dim_{\bbF}(V/T\ot\Chibar)^{G_K}[\uf]+\dim_\bbF \Sel_\cF(K,V/T\ot\Chibar)[\uf],\]
we have
\[\dim_\bbF\cS\ot\bbF\leq 2\dim_\bbF\Sel_\cF^{(\ell_1)}(K,\Epr\ot\Chibar)[\uf]\leq 2d_0+4.\]
Let $B=2C_1+2\Btwo+2\Bfour$, define
\[C_2:=B+(2^{2d_0+5}-2)\Bone,\]
and let $\cY\subset V_\cS^+$ be the subset consisting of maps $f$ such that
$p^{2C_2}\Epr$ is contained in the $\cO$-submodule generated by $f(s_1)$ and $f(s_2)$,
where
\[s_1:=\res_{K,L}(\Hg{1}{\Chibar}),\quad s_2:=\res_{H,L}(\Hg{\ell_1}{\Chibar}).\]

\begin{lm}\label{L:3}
The set $\cG_\cS^+\cap \cY$ is non-empty.
\end{lm}

\begin{proof}
First suppose $\chi^2\not =1$. Define the $\cO$-module map
\[\xi:V_\cS^+\longto \Epr\oplus\Epr,\quad f\longmapsto \xi(f):=(f(s_1),f(s_2))=(f(\Hg{1}{\Chibar}),f(\Hg{\ell_1}{\Chibar})).\]
Let $\cV^+:=\xi(V_\cS^+)\subset \Epr\oplus\Epr$. We claim that
\[
\uf^B(\Epr\oplus\Epr)\subset \cV^+.
\]
Indeed, let $S_1\subset \cS$ be the $\cO$-submodule generated by $\stt{s_1,s_1^\tau,s_2,s_2^\tau}$ where
$s_i^\tau:=\tau\cdot s_i$. For $(t_1,t_2)\in \uf^B\Epr\oplus \uf^B\Epr$, we define $g:S_1\to \Epr$ by
\begin{align*}
g(xs_1+ys_2+zs_1^\tau+ws_2^\tau)&=xt_1+yt_2.\end{align*}
Note that if $xs_1+ys_2+zs_1^\tau+ws_2^\tau=0$,
then $\uf^{\Bfour}(xs_1+ys_2)=\uf^{\Bfour}(zs_1^\tau+ws_2^\tau)=0$,
and hence \begin{align*}
&\;\uf^{\Btwo+\Bfour}y\pd_{\ell_1}(\Hg{\ell_1}{\Chibar})=\uf^{\Btwo+\Bfour}z\pd_{\ell_1}(\tau\cdot\Hg{\ell_1}{\Chibar}))=0\\
\Longrightarrow &\;\uf^{2\Btwo+\Bfour}y\al_{\ell_1}(\Hg{1}{\Chibar})=\uf^{2\Btwo+\Bfour}z\al_{\ell_1}(\Hg{1}{\Chi})=0\\
\Longrightarrow &\;\Ord_\uf(y),\;\Ord_\uf(z)\geq M-C_1-2\Btwo-\Bfour\geq M-B,
\end{align*}
and similarly:
\begin{align*}
&\;\uf^{C_1+\Bfour+2\Btwo}x\Hg{1}{\Chibar}+\uf^{C_1+\Bfour+2\Btwo}z\Hg{1}{\Chi}=0\\
\Longrightarrow &\;\uf^{C_1+2\Bfour+2\Btwo}x\Hg{1}{\Chibar}=\uf^{C_1+2\Bfour+2\Btwo}z\Hg{1}{\Chi}=0\\
\Longrightarrow &\;\Ord_\uf(x),\;\Ord_\uf(z)\geq M-(2C_1+2\Bfour+2\Btwo)=M-B.
\end{align*}
We thus find that $xt_1=yt_2=0$, and so $g$ is well-defined.
Extending $g$ to a map $\wtd g:\cS\to\Epr$, we put $f:=\wtd g+\wtd g^\tau\in V_\cS^+$. Since we have
\[f(s_1)=g(s_1)+\tau g(s_1^\tau)=t_1,\quad f(s_2)=g(s_2)-\tau g(s_2^\tau)=t_2,\]
this verifies the claim.

Now let $q:\Epr\oplus \Epr\to\cO_M$ be the quadratic form defined by
$q(v)=v_1\wedge v_2$ for all $v=(v_1,v_2)$, and let $I\subset \cO_M$ be the ideal generated by $\stt{q(v)}_{v\in \cV^+}$.
Note that $I\supset \uf^{2B}\cO_M$. By \lmref{L:nonfull}, $\uf^{C_2-B}\cV^+$ is contained in the $\cO$-module generated by $\xi(\cG_\cS^+)$. This implies that $\uf^{2C_2}\cO_M\subset \uf^{2C_2-2B}I$ is contained in the ideal generated by $\stt{q(v)}_{v\in \xi(\cG_\cS^+)}$. We thus conclude that there exists $g\in \cG_\cS^+$ such that $\xi(g)=(v_1,v_2)$ with $v_1\wedge v_2\in \uf^r\cO_M^\x$ and $r\leq 2C_2$. This shows that
\[
\cO v_1+\cO v_2\supset \uf^r\Epr\supset \uf^{2C_2}\Epr,
\]
and hence $g\in \cY$.

Next we assume that $\chi^2=1$. Then we have \[s_1^\tau=\ep s_1,\quad s_2^\tau=(-\ep) s_2\]
for some $\ep\in\stt{\pm 1}$. Define the $\cO$-module map\[\xi:V_\cS^+\longto \Epr^{\ep}\oplus\Epr^{-\ep}=\Epr,\quad f\longmapsto \xi(f)=f(s_1)+f(s_2)=f(\Hg{1}{\Chibar})+f(\Hg{\ell_1}{\Chibar}),\] and let $\cV^+:=\xi(V_\cS^+)\subset \Epr$.
We now claim that $\uf^B\Epr\subset \cV^+$. Let $S_1\subset H^1(L,\Epr)$ be the submodule generated by $\stt{s_1,s_2}$. For each $(t_1,t_2)\in \uf^{B}\Epr^{\ep}\oplus\uf^{B}\Epr^{-\ep}=\uf^{B}\Epr$, define $g:S_1\to\Epr$ by \[g(x s_1+y s_2)=xt_1+yt_2\quad(x,y\in\cO).\]
%If $xs_1+ys_2=0$, then $\uf^{\Btwo}x\Hg{1}{\Chibar}+\uf^{\Btwo}y\Hg{\ell_1}{\Chibar}=0$ by \lmref{P:Nek122} (1), from which it follows that \[y\uf^{\Btwo}\pd_{\ell_1}(\Hg{\ell_1}{\Chibar})=0\imply y\uf^{2\Btwo}\al_{\ell_1}(\Hg{1}{\Chibar}))=0\imply \Ord_\uf(y)\geq M-C_1-2\Btwo\geq M-B,\] and hence \[\uf^{C_1+2\Btwo}x\Hg{1}{\Chibar}=0\imply \Ord_\uf(x)\geq M-2C_1-2\Btwo\geq M-B.\] We find that $xt_1=yt_2=0$.
One can verify that $g$ is well-defined as before, and extending $g$ to a map $\wtd g:\cS\to\Epr$, we set $f:=\wtd g+\wtd g^\tau$.
Then $f(s_1)=2t_1$ and $f(s_2)=2t_2$, proving the claim. By \lmref{L:nonfull}, $\uf^{C_2}\Epr\subset \uf^{C_2-B}\cV^+$ is contained in the $\cO$-module generated by $\xi(\cG_\cS^+)$, and we find that
\[\xi(\cG_\cS^+)\not\subset \left(\uf^{C_2}\Epr^+\oplus \uf^{C_2+1}\Epr^-\right)\cup \left(\uf^{C_2+1}\Epr^+\oplus \uf^{C_2}\Epr^-\right),\]
which implies that $\cG_\cS^+\cap \cY$ is non-empty.
\end{proof}

By \propref{P:Nek122} and \lmref{L:3}, there is a finite set $\Sg_\cY$ of $M$-admissible primes such that
\[
\stt{\Frob_\ell(L_\cS/K)}_{\ell\in\Sg_\cY}=\cG_\cS^+\cap\cY.
\]
Define the $\Sg_\cY$-restricted Selmer group $\Sel_{\Sg_\cY}$ by
\[\Sel_{\Sg_\cY}=\stt{s\in \Sel_\cF(K,\Epr\ot\Chi)\mid s(\Frob_\ell)=0\text{ for all }\ell\in\Sg_\cY}.\]
Then we have the exact sequence:
\beq\label{E:exact}  \bigoplus_{\ell\in\Sg_\cY}H^1_{\sing}(K_\lam,\Epr\ot\Chibar)
\longto \Sel_\cF(K,\Epr\ot\Chi)^\vee \longto \Sel_{\Sg_\cY}^\vee\longto 0\eeq

\begin{lm}\label{L:5} $p^{\Btwo+2C_2+1}\Sel_{\Sg_\cY}=\{0\}$.
\end{lm}
\begin{proof}
By definition, if $s\in\Sel_{\Sg_\cY}$ then $s(\cG_\cS^+\cap\cY)=0$.
Noting that $\cG_\cS\cap\cY+p^{2C_2+1} \cG_\cS^+\subset \cG_\cS^+\cap \cY$, we thus find that
\begin{align*}
s(\cG_\cS^+\cap\cY)=0\;&\Longrightarrow\;s(p^{2C_2+1} \cG_\cS^+)=0\\
\;&\Longrightarrow\;p^{2C_2+1}\res_{K,L}(s)=0\in H^1(L,\Epr).
\end{align*}
By \propref{P:Nek122}(1), it follows that $p^{\Btwo+2C_2+1}s=0$.
\end{proof}

\begin{lm}\label{L:6}For each $\ell\in \Sg_\cY$, we have
%\[\uf^{2B}\cdot H^1_{\fin}(K_\lam,\Epr\ot\Chibar)\subset \cO\Hg{1}{\Chibar}_\lam+\cO\Hg{\ell_1}{\Chibar}_\lam.\] In addition,
\[p^{2B+\Btwo}H^1_{\sing}(K_\lam,\Epr\ot\Chibar)\subset \cO\pd_{\ell}\Hg{\ell}{\Chibar}+\cO\pd_{\ell}\Hg{\ell\ell_1}{\Chibar}.\]
\end{lm}
\begin{proof}By the choice of $\ell\in \Sg_\cY$, we have \[\uf^{2B}\Epr\subset \cO(\al_\ell(\Hg{1}{\Chibar})+\cO(\al_\ell(\Hg{\ell_1}{\Chibar})).\]
This is equivalent to $\uf^{2B}H^1_{\fin}(K_\lam,\Epr\ot\Chibar)\subset \cO\loc_\ell(\Hg{1}{\Chibar})+\cO\loc_\ell(\Hg{\ell_1}{\Chibar})$.
The lemma thus follows from property \eqref{K2}.
\end{proof}

Now \thmref{T:Descent.E} is a consequence of the following result.
\begin{thm}\label{T:EulerSystem}There exists a positive integer $C$ such that
\[p^{C}\cdot \left(\Sel_{\cF^*}(K,V/T\ot\Chi)/(F/\cO\cdot\Esys_K^\tau)\right)=\{0\}.\]
\end{thm}
\begin{proof}
We denote by
\[
\pairing_\lam\colon H^1_{\fin}(K_\lam,\Epr\ot\Chi)\x H^1_{\sing}(K_\lam,\Epr\ot\Chibar)\longto \Z/M\Z
\]
the Tate local pairing. By the exact sequence \eqref{E:exact} combined with \lmref{L:5} and \lmref{L:6},
for every $f\in \Sel_\cF(K,\Epr\ot\Chi)^\vee$ we can write \[p^{C_3}\cdot f=\sum_{\lam\in\Sg_\cY}a_\ell \pd_{\ell}\Hg{\ell}{\Chibar} +b_\ell \pd_{\ell}\Hg{\ell\ell_1}{\Chibar},\quad C_3:=2C_2+2B+2\Btwo+1.\]
Thus for every $s\in\Sel_\cF(K,\Epr\ot\Chi)$ we have
\begin{align*}(p^{C_3}\cdot f)(s)&=f(p^{C_3}\cdot s)\\
&=\sum_{\ell\in\Sg_\cY}\pair{\loc_\lam(s)}{b_\lam\pd_{\ell}\Hg{\ell\ell_1}{\Chibar}}_\lam\\
&=\pair{\loc_{\lam_1}(s)}{t_{\lam_1}}_{\lam_1}\quad(t_{\lam_1}:=\sum_{\ell\in\Sg_\cY}-b_\lam\pd_{\ell_1}\Hg{\ell\ell_1}{\Chibar}).
\end{align*}
%In particular, there exists $t\in H^1_{\sing}(K_{\lam_1},\Epr\ot\Chibar)$ such that
%\[f(p^{C_3} s)=(p^{C_3}f)(s)=\pair{t}{\loc_{\lam_1}(s)}_{\lam_1},\]
This implies that $p^{C_3}$ annihilates the kernel of the localization map
\[\loc_{\lam_1}\colon\Sel_\cF(K,\Epr\ot\Chi)\longto H'_{\ell_1}:=\stt{s\in H^1_{\fin}(K_{\lam_1},\Epr\ot\Chi) \mid \pair{s}{\cO\pd_{\ell_1}(\Hg{\ell_1}{\Chibar})}_{\lam_1}=0}.\]
On the other hand, setting
\[
a_1:=\Ord_\uf(\al_{\ell_1}(\Hg{1}{\Chi})),\quad a_2:=\Ord_\uf \beta_\ell(\pd_{\ell_1}\Hg{\ell_1}{\Chibar}),
\]
by \lmref{L:2} and \eqref{K2} we have $a_1\leq C_1$ and $a_2\leq C_1+\Btwo$.
If $M>a_1+a_2$, an elementary argument shows that
\[\uf^{2C_1+\Btwo}H'_{\ell_1}\subset\uf^{a_1+a_2}H'_{\ell_1}\subset \uf^{a_2}\cO\al_{\ell_1}(\Hg{1}{\Chi}).\]
Combining these together, we deduce that
\[p^{2C_1+\Btwo+C_3}\Sel_\cF(K,\Epr\ot\Chi)\subset \cO\Hg{1}{\Chi}=\cO_M \Esys_K^\tau\]
for every $M>2C_1+2\Btwo$, and the theorem follows.
\end{proof}

%\input{finiteness}
%\end{appendix}
\bibliographystyle{amsalpha}
\bibliography{Heegner}
\end{document}